\documentclass{amsart}

\usepackage{float}
\usepackage{fourier}
\usepackage{graphicx}
\usepackage{appendix}
\usepackage{tikz-cd}
\usepackage{amsfonts}
\usepackage{amssymb}
\usepackage{graphicx}
\usepackage{mathrsfs}
\usepackage{pb-diagram}
\usepackage{amscd}
\usepackage[extra]{tipa}
\usepackage{mathtools}
\usepackage{hyperref}
\usepackage{stmaryrd}
\usepackage{caption}
\usepackage{subcaption}
\usepackage{fancybox}
\usepackage{wrapfig}
\usepackage{color}
\usepackage{xcolor}
\usepackage[all,cmtip]{xy}
\usepackage{tabularx}
\usepackage{graphbox,graphicx}

\newtheorem{theorem}{Theorem}[section]
\newtheorem{lemma}[theorem]{Lemma}
\newtheorem{corollary}[theorem]{Corollary}

\newtheorem{prop}[theorem]{Proposition}
\newtheorem{defn}[theorem]{Definition}
\newtheorem{example}[theorem]{Example}
\newtheorem{remark}[theorem]{Remark}

\newtheorem{construction}[theorem]{Construction}

\numberwithin{equation}{section}

\newcommand{\cat}{\mathscr}

\newcommand{\Z}{\mathbb{Z}}
\newcommand{\R}{\mathbb{R}}
\newcommand{\U}{\mathbb{U}}
\newcommand{\C}{\mathbb{C}}

\newcommand{\bb}{\boldsymbol{b}}

\newcommand{\CF}{\mathrm{CF}}

\newcommand{\bL}{\mathbb{L}}

\newcommand{\A}{\mathbb{A}}

\newcommand{\cF}{\mathscr{F}}

\newcommand{\cT}{\mathcal{T}}
\newcommand{\cX}{\mathcal{X}}

\newcommand{\cH}{\mathcal{H}}

\newcommand{\cY}{\mathcal{Y}}
\newcommand{\val}{\mathrm{val}}

\newcommand{\op}{\textrm{op}}

\newcommand{\Hom}{\mathrm{Hom}}

\newcommand{\Fuk}{\mathrm{Fuk}}

\newcommand{\Span}{\mathrm{Span}}

\newcommand{\tr}{\mathrm{tr} \,}

\newcommand{\Mor}{\mathrm{Mor}}

\newcommand{\one}{\mathbf{1}}

\newcommand{\ev}{\mathrm{ev}}

\newcommand{\bP}{\mathbb{P}}

\newcommand{\cM}{\mathcal{M}}
\newcommand{\cA}{\mathcal{A}}

\newcommand{\cC}{\mathcal{C}}

\newcommand{\cO}{\mathcal{O}}
\newcommand{\cI}{\mathcal{I}}

\newcommand{\cL}{\mathcal{L}}

\newcommand{\MF}{\mathrm{MF}}
\newcommand{\Tw}{\mathrm{Tw}}

\newcommand{\cd}{\check{\partial}}
\newcommand{\Id}{\mathrm{Id}}

\DeclarePairedDelimiter{\norm}{\lVert}{\rVert}

\begin{document}
	
\title[Mirror symmetry for quiver stacks]{Mirror symmetry for quiver algebroid stacks}
\author{Siu-Cheong Lau, Junzheng Nan and Ju Tan}

\begin{abstract}
	In this paper, we provide a new construction of quiver algebroid stacks and the associated mirror functors for symplectic manifolds.  First, we formulate the concept of a quiver stack, which is a geometric structure formed by gluing multiple quiver algebras together.  
	Next, we develop a representation theory of $A_\infty$ categories by quiver stacks.  The main idea is to extend the $A_\infty$ category over a quiver stack of a collection of nc-deformed objects.  The extension involves non-trivial gerbe terms.  It gives an application of symplectic geometry that bridges the study of sheaves and representation theory through mirror symmetry.
	
	We provide a general framework for constructing mirror quiver stacks. In particular, we develop a novel method of gluing Lagrangians which are disjoint from each other by using quasi-isomorphisms with a `global middle agent', which is a Lagrangian immersion that produces a mirror quiver. The method relies fundamentally on the use of quiver stacks.  We carry out this construction for compact immersed Lagrangians in a punctured elliptic curve, which results in a mirror nc local projective plane.  
\end{abstract}

\maketitle

{\footnotesize \tableofcontents}

\section{Introduction}
\label{chapter:Introduction}

Stack is an important notion in the study of moduli spaces. Roughly speaking, a stack is a fibered category, whose objects and morphisms can be glued from local objects. Besides, a stack can also be understood as a generalization of a sheaf that takes values in categories rather than sets.


An algebroid stack is a natural generalization of a sheaf of algebras.  It allows gluing of sheaves of algebras by a twist of a two-cocycle.  Such gerbe terms arise from deformation quantizations of complex manifolds with a holomorphic symplectic structure, which are controlled by DGLA of cochains with coefficients in the Hochschild complex.  By the work of Bressler-Gorokhovsky-Nest-Tsygan \cite{BGNT-dq}, an obstruction for an algebroid stack to be equivalent to a sheaf of algebras is the first Rozansky-Witten invariant.

In this paper, we define and study a version of algebroid stacks that are glued from quiver algebras for the purpose of mirror symmetry.  We will see that gerbe terms appear naturally and play a crucial role, when gluing the quivers that have different numbers of vertices.  See Figure \ref{fig:C3Q3}.  We will call these to be quiver algebroid stacks (or simply quiver stacks).

We construct quiver stacks as Maurer-Cartan deformation spaces of Lagrangian immersions in symplectic manifolds. The main result of this paper is the following:
\begin{theorem}[Theorem \ref{thm: alg- functor} and Proposition \ref{prop:inj}]\label{thm: intro}
	Let $\cX$ be the quiver algebroid stack obtained by gluing the Maurer-Cartan deformation spaces of a collection of Lagrangian immersions $\cL$, using isomorphisms in the (extended) Fukaya category. Then there exists an $A_\infty$ functor 
	$$\cF^\cL: \Fuk (M) \xrightarrow{} \Tw(\cX),$$ where $\Tw(\cX)$ is the category of twisted complexes over $\cX$. Furthermore, $\cF^\cL$ is injective on $\mathrm{HF}^\bullet((\cL',b_0),L)$ for any Lagrangian $L$ and any constant elements $b_0$ in the deformation space of $\cL'$, where $\cL'$ is a subset of $\cL$. 
\end{theorem}

In this paper, we focus on developing the general formalism and illustrating via the example of noncommutative deformations of the canonical line bundle $K_{\bP^2}$. In future works, we will develop applications to quiver varieties and their noncommutative deformations. In particular, we will obtain noncommutative deformations for the $A_n$ quiver recently studied by Kawamata \cite{Kaw2}.

\begin{figure}[htb!]
	\centering
	\includegraphics[scale=0.5]{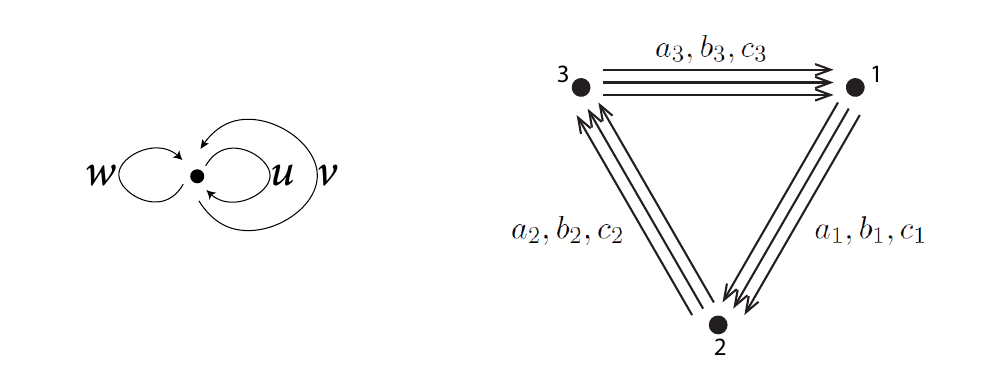}
	\caption{The quiver on the left corresponds to $\C^3$ and its noncommutative deformations.  The quiver on the right is used as a noncommutative resolution of $\C^3/\Z_3$.  These two quiver algebras with different numbers of vertices will be glued together in the context of quiver stacks.  
	}
	\label{fig:C3Q3}
\end{figure}

\subsection{A brief description and an example of a quiver stack}
Noncommutative geometry arises naturally from quantum mechanics and field theory, in which particles are modeled by noncommuting operators.  Connes \cite{Connes} has made a very deep foundation of the subject in terms of operator algebras and spectral theory. Moreover, the groundbreaking work of Kontsevich \cite{Kont-quant} has constructed deformation quantizations from Poisson structures on function algebras.  Deformation theory \cite{KS-deform,KR-nc} plays a central role.  The subject is rich and broad, contributed by many mathematicians and we do not attempt to make a full list here.

In this paper, we focus on noncommutative algebras that come from quiver gauge theory.  They are given by quiver algebras with relations 
$$\A = \C Q \,/R$$
where $Q$ is a quiver, $\C Q$ is the path algebra and $R$ is a two-sided ideal of relations.  
Such nc geometries have important physical meaning: vertices represent branes at a Calabi–Yau singularity, arrows represent string interactions between them, and the relations come from the \emph{spacetime superpotential}, which encodes the couplings.  Deformations of this spacetime superpotential produce interesting noncommutative geometries.  Such nc geometries provide the worldvolume theory for D-branes in a local Calabi-Yau twisted by non-zero B-fields \cite{SW,Furuuchi-Okuyama}.  

We are motivated from quiver crepant resolutions of singularities found by Van den Bergh \cite{vdBergh}, where quiver algebras served as noncommutative crepant resolutions.  
Van den Bergh showed that these quiver algebras and the usual geometric crepant resolutions have equivalent derived categories.  This proves a version of the Bondal-Orlov conjecture that two crepant resolutions of the same Gorenstein singularity have equivalent derived categories.

In this paper, we would like to find local-to-global descriptions for quiver algebras via mirror symmetry. We formulate the notion of a local chart of a quiver algebra, and find charts and chart transitions via quasi-ismorphisms of Lagrangian immersions in the Fukaya category.

We understand a quiver algebra $\A=\C Q/R$ as the homogeneous coordinate ring of a $Q_0$-graded noncommutative variety, where $Q_0$ denotes the vertex set.  It is natural to ask for affine local charts of such a variety, which we expect to be a path algebra with a single vertex. Motivated by this, we introduce the notion of quiver algebroid stack, see Definition \ref{def:qstack}, which is formed by gluing the path algebras via representations with possibly nontrivial gerbe terms.

\begin{defn}\label{def:rep}
	A representation $G_{21}$ of a quiver algebra $\cA_1$ by another quiver algebra $\cA_2$ consists of an assignation $f:V_{\cA_1}\rightarrow V_{\cA_2}$, together with a family of maps $$g_{h,t}:e_h \cdot \cA_1 \cdot e_t \rightarrow e_{f(h)}\cdot \cA_2 \cdot e_{f(y)}$$ indexed by the ordered pairs $(h,t)\in V_{\cA_1} \times V_{\cA_1}$, where $V_{\cA_k}$ are the sets of vertices for $k=1,2$. Moreover, the representation $G_{21}$ is required to preserve relations of $\cA_1$ and $\cA_2$.
	
\end{defn}

\begin{remark}
	If one understands a path algebra as a category, where objects are vertices and morphisms are arrows, then a representation $G$ is a functor preserving the relations.
\end{remark}

\begin{defn} \label{def:chart}
	An affine chart of a quiver algebra $\A$ is $$\left(A'=\C Q'/R',G_{01},G_{10}\right)$$
	where $Q'$ is a quiver with a single vertex and $R'$ is a two-sided ideal of relations; $$G_{01}: A' \to \A_{\mathrm{loc}} \textrm{ and } G_{10}: \A_{\mathrm{loc}} \to A'$$ 
	are representations that satisfy
	\begin{align*}
		G_{10} \circ G_{01} =& \mathrm{Id};\\
		G_{01} \circ G_{10}(a) =& c(h_a) \,a \,c(t_a)^{-1} 
	\end{align*}
	for some function $c: Q_0 \to (\A_{\mathrm{loc}})^{\times}$
	that satisfies $c(v) \in e_{v_0} \cdot \A_{\mathrm{loc}} \cdot e_{v}$, where $v_0$ denotes the image vertex of $G_{01}$.  Here, $\A_{\mathrm{loc}}$ is a localization of $\A$ at certain arrows (meaning to add corresponding reverse arrows $a^{-1}$ and imposing $aa^{-1}=e_{h_a}, a^{-1}a=e_{t_a}$) and $(\A_{\mathrm{loc}})^{\times}$ is the set of invertible elements in $\A_{\mathrm{loc}}$, see Definition \ref{def:unit}.  $e_v$ denotes the trivial path at the vertex $v$.
\end{defn}

\begin{example}[Free projective space] \label{ex:freeProj}
	Consider the quiver $Q$ with two vertices $0,1$ and several arrows $a_k, k=0,\ldots,n$ from vertex $0$ to $1$.  An affine chart of the path algebra $\C Q$ can be constructed by localizing $\C Q$ at one arrow $a_l$ for $l=0,\ldots,n$.  We take $A' = \C Q'$ where $Q'$ is the quiver with a single vertex and $n$ loops $X_k, k\in \{0,\ldots,n\}-\{l\}$.  We fix the image vertex of $G_{01}$ to be the vertex $0$.  Then define
	\begin{align*}
		G_{01}(X_k) =& a_l^{-1}a_k; \, G_{10}(a_k) = X_k\\
		c(0)=&0; \, c(1)=a_l^{-1}.
	\end{align*}
	One can easily check that the required equations are satisfied. In particular, the gerbe terms arise naturally.  This is a free algebra analog of the projective space, where $a_k,X_k$ are the homogeneous and inhomogeneous coordinates.
\end{example}

Gluing the quiver algebra $\A$ together with its affine charts, we get a quiver algebroid stack, see Definition \ref{def:qstack} for more details.


We will construct algebroid stacks and the universal complexes via mirror symmetry.  While our method of construction is general, this paper will focus on the case of $K_\bP^2$.  We will work out the construction for the resolved conifold and $A_n$ resolutions in a subsequent paper.

\subsection{Gluing of immersed Lagrangians with more than one components}
Mirror symmetry has been an active subject of research in recent decades,  with far-reaching impact on geometry and topology.  Homological mirror symmetry \cite{Kont-HMS} asserted a deep duality between Lagrangian submanifolds in a symplectic manifold and coherent sheaves over the mirror algebraic variety.  

The program of Strominger-Yau-Zaslow \cite{SYZ} has proposed a grand unified geometric approach to understand mirror symmetry via duality of Lagrangian torus fibrations. According to the SYZ program, mirror manifolds are expected to arise as the quantum-corrected moduli space of possibly singular fibers of a Lagrangian fibration, which motivates several important works, including the family Floer theory \cite{Fukaya-famFl,Tu-reconstruction, Ab-famFl} and the Gross-Siebert programs \cite{GS07}. In general, the singular fibers may have several components in their normalizations, and their deformations and obstructions are naturally formulated as quiver algebras (where the vertices correspond to the components).  This leads to the necessity of gluing quiver algebras associated with singular and smooth fibers.  Quiver stacks come up naturally as the quantum corrected moduli of Lagrangian fibers in such situations.

In \cite{CHL-nc}, Cho, Hong and the first author constructed quiver algebras as noncommutative deformation spaces of Lagrangian immersions in a symplectic manifold.  In another work \cite{CHL3}, the authors globalized the mirror functor construction in the usual commutative setting \cite{CHL}, by gluing local deformation spaces of Lagrangian immersions using isomorphisms in the (extended) Fukaya category.  

In this paper, we \emph{combine ideas and methods in HMS,  SYZ, and powerful techniques from Lagrangian Floer theory developed by Fukaya-Oh-Ohta-Ono} \cite{FOOO}, to construct mirror quiver algebroid stacks $\cX$ by finding noncommutative boundary deformations of Lagrangian immersions and isomorphisms between them.  We extend the Fukaya category over the quiver stack and develop a gluing scheme of local noncommutative mirrors. This produces a mirror functor to the dg category of twisted complexes over the quiver stack as in Theorem \ref{thm: intro}. This combines the methods of \cite{CHL-nc} and \cite{CHL3}. Besides,  we will explicitly compute the mirror functor in object and morphism levels and apply it to construct universal sheaves for the cases of nc $K_{\bP^2}$.

For the local-to-global construction of toric Calabi-Yau 3-folds, we take a \emph{pair-of-pants decomposition} of the Riemann surface, and consider a Seidel Lagrangian \cite{Seidel-g2,Sei-spec} $S_j$ in each copy of pair-of-pants.  See the left of Figure \ref{fig:ppdecomp} for the three-punctured elliptic curve that appears in Example \ref{ex:Seidel}.

\begin{figure}[htb!]
	\centering
	\begin{subfigure}[b]{0.4 \textwidth}
		\centering
		\includegraphics[width=\textwidth]{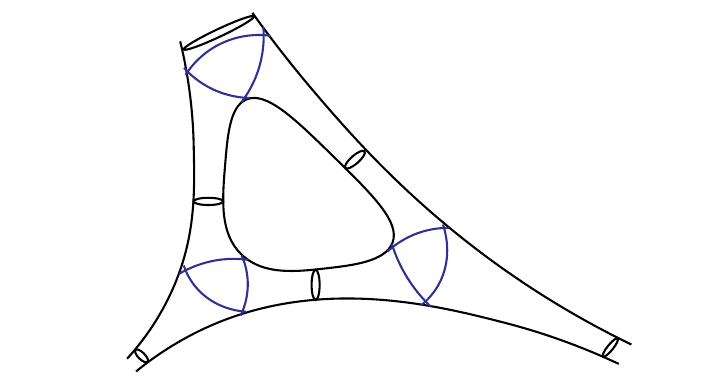}
		\caption{}
		\label{fig:ppdecomp}
	\end{subfigure}
	\hspace{40pt}
	\begin{subfigure}[b]{0.4 \textwidth}
		\centering
		\includegraphics[width=\textwidth]{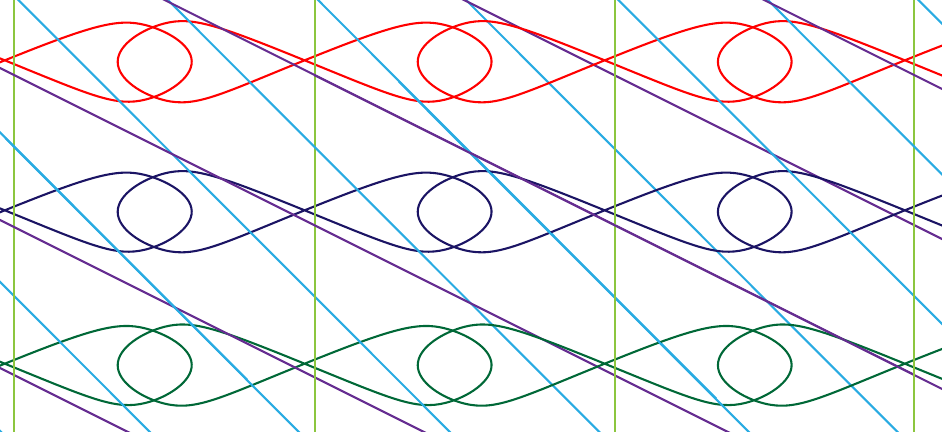}
		\caption{}
		\label{fig:S-L}
	\end{subfigure}
	\caption{The left shows a pair-of-pants decomposition of the three-punctured elliptic curve and Seidel Lagrangians.  The right shows a way to put Seidel Lagrangians so that they can be isomorphic to the `middle Lagrangian' $\bL$.}
\end{figure}

We want to glue up the noncommutative deformation spaces of the local Seidel Lagrangians $S_j$, which are nc $\Lambda_+^3$, in the pair-of-pants decomposition.  However, these Lagrangians do not intersect each other, implying that their deformations spaces over the Novikov ring $\Lambda_+$ do not intersect with each other.  




Here, we \emph{find a new method to get around the problem that the local Seidel Lagrangians $S_j$ `do not talk to each other'}.  Namely, we take the global Lagrangian $\bL$ shown in Figure \ref{fig:ellcurve} as a `\emph{middle agent}' that all $S_j$ can talk to.  Then the gluing maps between deformation spaces of different $S_j$'s can be found by composing that between $S_j$ and $\bL$.

More precisely, we shall find \emph{noncommutative isomorphisms} between $(S_j, \bb_j)$ and $(\bL,\bb)$, where the boundary deformations $\bb_j$ and $\bb$ are over \emph{different quiver algebras} $\mathscr{A}_j$ and $\A$ respectively. Here $\cA_j$ (resp. $\A$) is the deformation space of $S_j$ (resp. $\bL$).  We will solve for algebra embeddings $\mathscr{A}_j \to \A_{\textrm{loc}}$ (where $\A_{\textrm{loc}}$ is a certain localization of $\A$) such that the isomorphism equations hold for certain
$\alpha_j\in \CF^0(\bL,S_j)_{\A_{\textrm{loc}}}, \beta_j \in \CF^0(S_j,\bL)_{\A_{\textrm{loc}}}$,
\begin{align*}
	m_1^{\bb,\bb_j}(\alpha_j)&=0, m_1^{\bb_j,\bb}(\beta_j)=0;\\
	m_2^{\bb,\bb_j,\bb}(\alpha_j,\beta_j)&=\one_\bL, m_2^{\bb_j,\bb,\bb_j}(\beta_j,\alpha_j)=\one_{S_j}.
\end{align*}
In this method, the middle agent $\bL$ typically has more than one components in its normalization. Hence, its deformation space will be a quiver algebra with more than one vertices. This motivates us to develop a mirror construction of quiver algebroid stacks in Section \ref{section: mirror_algebroid}. In Section \ref{sec:example}, we carry out such a construction for mirror symmetry in three-punctured elliptic curve, which produces nc local projective plane.  We find non-trivial isomorphisms between $\bL$ and $S_i$, see Figure \ref{fig:KP2_m1alpha}.  It is interesting that we need to localize at the noncommutative quiver variables for the existence of isomorphisms.

\begin{theorem}[Theorem \ref{prop:isomorphism_KP2}]
	The mirror construction for the Seidel Lagrangians $S_i$ together with the middle Lagrangian $\bL$ in the three-punctured elliptic curve produces the nc deformed $K_{\bP^2}$ shown in Example \ref{ex:nclocP2}.
\end{theorem}

\subsection{Triality between symplectic geometry, complex geometry and representation theory}

Now we have two mirrors, namely $\cX$ constructed from $\cL_i:=S_i$, and $\A$ constructed from $\bL$. In good examples, $\cX$ realizes the commutative crepant resolution, while $\A$ provides its noncommutative counterpart.  Motivated by the analogy with algebraic geometry—where one often compares the derived categories of noncommutative and commutative crepant resolutions—it is natural to investigate the relationship between these two mirror constructions. To this end, we construct a twisted complex of $(\cA_i,\A)$-bimodules $\U$ over $\cX$ by taking the mirror transform of $(\bL,\bb)$. In some interesting cases, $\U$ is the universal bundle over the moduli space of stable $\A$-module. Besides, this twisted complex induces a functor $\cF^{\U}:=\Hom(\U,-):\Tw(\cX) \xrightarrow{} \mathrm{dg-mod}(\A)$.   
\begin{equation} \label{eq:diagram}
	\begin{tikzcd}[column sep=small,row sep=large]
		& \mathrm{Fuk}(M) \arrow{dl}[left,near start]{\cF^\cL} \arrow{dr}{\cF^{(\bL,\bb)}} & \\
		\mathrm{Tw}(\cX) \arrow{rr}{\cF^{\U}} & & \mathrm{dg-mod}(\A)
	\end{tikzcd}
\end{equation}
We show that:

\begin{theorem}[Theorem \ref{thm:nat-trans-X-A}]
	There exists a $A_\infty$-natural transformation $\cT:\cF^{(\bL,\bb)} \to \A \otimes (\cF^{\U}\circ \mathscr{F}^{\cL})$.
\end{theorem} 

Using the Fukaya isomorphisms between $(\bL,\bb)$ and $(\cL_j,\bb_j)$, we deduce the injectivity of the natural transformation $\cT$:


\begin{theorem}[Theorem \ref{thm:loc inj}]
	Suppose there exist $\alpha_{i} \in \cF^{\cL_i}(\bL), \beta_{i} \in \cF^{\bL}(\cL_{i})$ that satisfies the above equation for some $i$. Then the natural transformation $\cT:\cF^{(\bL,\bb)} \to \A \otimes (\cF^{\U}\circ \mathscr{F}^{\cL})$ has a left inverse.
\end{theorem}

\subsection{Related works}
In the beautiful work of Auroux-Katzarkov-Orlov \cite{AKO06,AKO08}, the Fukaya-Seidel category of the Landau-Ginzburg mirror $W = z+w+\frac{1}{zw}$ on $(\C^\times)^2$ and its non-exact deformations were computed, which was shown to be mirror to $\bP^2$ and its noncommutative deformations.  These lead to Sklyanin algebras \cite{AS87,ATV}, which also appear in the Landau-Ginzburg mirrors of elliptic $\bP^1$-orbifolds \cite{CHL-nc}.  In this paper, we construct algebroid stacks charts-by-charts by gluing local nc deformation spaces of immersed Lagrangians.  The main example of mirrors constructed in Section \ref{chapter:KP2} is a manifold version of noncommutative local projective planes, compared with the algebra counterparts constructed in \cite{AKO08,CHL-nc}.
Moreover, we construct a universal bundle via mirror symmetry that transforms sheaves over the algebroid stack to modules of the corresponding global algebra.

The gluing construction in this paper is a further development of the technique in the joint work \cite{CHL3} of the first author with Cheol-Hyun Cho and Hansol Hong, which is new to existing methods known to the authors. \cite{CHL3} concerned about commutative deformation spaces of Lagrangian immersions, and dealt with the case that any three distinct charts have empty common intersections (which was enough for the construction of mirrors of pair-of-pants decompositions for curves over the Novikov ring).  In this paper, using the language of quiver algebroid stacks, we allow local charts given by nc quiver algebras and also permit non-empty intersection of any number of charts.  We have also extended Floer theory over quiver stacks that allow gerbe terms.

In \cite{HLT24}, the authors used the technique of quiver stacks developed in this paper to construct the crepant resolutions of $A_n$ and $D_4$ singularities as the Maurer-Cartan deformation spaces of plumbings in affine type $A_n$ and $D_4$ respectively.

Recently, Kawamata has developed a series of important works in noncommutative deformations \cite{Kaw1,Kaw2,Kaw3}.
In these papers, he introduced the notion of noncommutative (NC) schemes by gluing NC deformations of algebras, which is quite similar to the perspective of this paper, in which we glue noncommutative deformation spaces of Lagrangian immersions into a quiver stack.
He proved that whenever a commutative crepant resolution and a tilting bundle exist, the derived equivalence between the commutative and noncommutative crepant resolutions is preserved under formal NC deformations. 
In this paper, we use non-exact deformations of Lagrangian Floer theory to construct noncommutative deformations of both the crepant resolution $K_{\bP^2}$ and the noncommutative crepant resolution of $\C^3/\Z_3$.

Below is the plan of this paper.
In Section 2, we define a version of algebroid stacks and twisted complexes that well adapts to quiver algebras.  The main ingredient is concerning the representation of a quiver algebra over another quiver algebra, in place of usual algebra homomorphisms, and isomorphisms between them.

Section 3 is the main part of our theory.  We further develop the gluing techniques in \cite{CHL3} to the noncommutative setting of \cite{CHL-nc}.  The key step is to extend the $A_\infty$ operations in Fukaya category over algebroid stacks.  In gluing quiver algebras with different numbers of vertices, gerbe terms $c_{ijk}$ in an algebroid stack will be unavoidable, and we need to carefully deal with them in extending the $m_k$ operations.  Another main construction is to compare functors constructed from two different reference Lagrangians.  We need to extend the $m_k$ operations for bimodules in a delicate way so that we have desired morphisms of modules and natural transformations.

In Section 4, we construct $\hbar$-deformed $K_{\bP^2}$ and twisted complexes over it using mirror symmetry.  The key difficulty is to find a (noncommutative) isomorphism between local Seidel Lagrangians and an immersed Lagrangian coming from a dimer model.  Another difficulty arises from the fact that the local Seidel Lagrangians do not intersect with each other.  We employ the method of `middle agent' to solve this problem.  This will be particularly important in the construction of the universal bundle.

\subsection*{Notations}
We will use the following notations for the Novikov ring
$$ \Lambda_+ = \left\{\left.\sum_{i=1}^{\infty} a_i T^{\lambda_i} \right| \lambda_i \in \mathbb{R}_{>0}, a_i \in \mathbb{C}, \lambda_i \textrm{ increases to } \infty \right\},  $$
and the maximal ideal
$$\Lambda_0 = \left\{\left.\sum_{i=1}^{\infty} a_i T^{\lambda_i}\right| \lambda_i \in \mathbb{R}_{\geq 0}, a_i \in \mathbb{C}, \lambda_i \textrm{ increases to } \infty \right\}$$
of the Novikov field
$$\Lambda = \left\{\left.\sum_{i=1}^{\infty} a_i T^{\lambda_i}\right| \lambda_i \in \mathbb{R}, a_i \in \mathbb{C}, \lambda_i  \textrm{ increases to } \infty \right\}.$$ 

\subsection*{Acknowledgments}
The first author is very grateful to Cheol-Hyun Cho, Hansol Hong and Dongwook Choa for the important discussions and collaborative works on the topic of gluing construction for immersed Lagrangians and mirror symmetry.  Moreover, he expresses his gratitude to The Chinese University of Hong Kong for hospitality during the preparation of this paper.  We thank Naichung Conan Leung, Xiao Zheng and Yan Lung Li for interesting discussions during the visiting period in December 2021.  The first author is supported by Simons Collaboration Grant.
  
\section{Quiver Algebroid Stacks}
\label{chapter:twisted_cochains}

In this section, we first recall the definition of algebroid stacks and twisted cochains following \cite{BGNT}. Next, we generalize the notions and define quiver algebroid stacks. This is necessary for gluing quiver algebras with different number of vertices, as they cannot be isomorphic to each other in the usual sense of algebras.

\subsection{Review on algebroid stacks and twisted cochains} 

\begin{defn} \label{def:stack}
	Let $B$ be a topological space.
	An algebroid stack $\mathcal{A}$ over $B$ consists of the following data:
	
	\begin{enumerate}
		\item An open cover $\{U_i: i \in I\}$ of $B$.
		\item
		A sheaf of algebras $\mathcal{A}_i$ over each $U_i$.
		
		\item
		An isomorphism of sheaves of algebras $G_{ij}: \mathcal{A}_j|_{U_{ij}} \stackrel{\cong}{\rightarrow} \mathcal{A}_i|_{U_{ij}}$ for every $i,j$.
		\item
		An invertible element $c_{ijk} \in \mathcal{A}_i|_{U_{ijk}}$ for every  $i,j,k$ satisfying 
		\begin{equation}
			G_{ij}G_{jk} =  Ad(c_{ijk}) G_{ik},
			\label{eq:cocycle}
		\end{equation}
		such that for any $i,j,k,l$,
		\begin{equation}
			c_{ijk} c_{ikl} = G_{ij}(c_{jkl})c_{ijl}.
			\label{eq:c}
		\end{equation}
	\end{enumerate}
\end{defn}

We call $\cA_i$ to be charts of $\cA$.  

Besides, we can define the isomorphism between two algebroid stacks.

\begin{defn}\label{def: equiv}
	An isomorphism between two algebroid stacks $(U',\cA',G',c')$ and $(U'',\cA'',G'',c'')$ consists of an open cover $M=\bigcup_i U_i$ that refines both covers $U'$ and $U''$, together with isomorphisms $H_i: \cA_i'(U_i) \to \cA_i''(U_i)$ and invertible elements $b_{ij}$ of $\cA_i'(U_i \cap U_j)$ such that $G_{ij}''=H_i Ad(b_{ij})G_{ij}' H_j^{-1}$ and $H_i^{-1}(c_{ijk}'')=b_{ij}G_{ij}'(b_{jk})c_{ijk}'b_{ik}^{-1}$.
\end{defn}

Given a refinement of the open cover of an algebroid stack, one gets an isomorphic algebroid stack simply by restriction (with $H_i$ and $b_{ij}$ taken to be the identity in the above definition).

Moreover, one can consider sheaves over an algebroid stack. Let $E^\bullet$ be a collection of graded sheaves $E_i^\bullet$ over $U_i$, where 
$E_i^\bullet(U_i)$ is a direct summand of a free graded 
$\mathcal{A}_i(U_i)$-module of finite rank, and $E_i^\bullet(V)$ is the image of $E_i^\bullet(U_i)$ under the restriction map $\cA_i(U_i) \to \cA_i(V)$ for any open $V\subset U_i$.  (And the restriction map $E_i^\bullet(V_1) \to E_i^\bullet(V_2)$ is induced from the restriction $\cA_i(V_1) \to \cA_i(V_2)$ for any open $V_2\subset V_1\subset U_i$.)
Let
$$
C^\bullet(\mathcal{A},E^\bullet) = \prod_{\substack{p \geq 0 \\ q\in \Z}} C^p(\mathcal{A}, E^q)
$$
where an element $a^{p,q}$ consists of sections $a^{p,q}_{i_0,...,i_p}$ of $E_{i_0}^q(U_{i_0,...,i_p})$ for all $i_0,\ldots,i_p$.

Consider another collection of graded sheaves $F=\{F_i^\bullet\}$ as above. Let
$$
C^\bullet(\mathcal{A},\Hom^\bullet(E,F)) = \prod_{\substack{p \geq 0 \\ q\in \Z}} C^p(\mathcal{A},\Hom^q(E,F)).
$$
An element $u^{p,q}\in C^{p}(\mathcal{A},\Hom^q(E,F))$ consists of sections $$u^{p,q}_{i_0,...,i_p} \in \Hom_{\mathcal{A}_{i_0}}^q(G_{i_0i_p}(E_{i_p}^\bullet),F_{i_0}^\bullet)$$ over $U_{i_0,...,i_p}$ for all $i_0,\ldots,i_p$, where $G_{i_0i_p}(E_{i_p}^\bullet)$ (restricted on $U_{i_0,...,i_p}$) is the $\cA_{i_0}$-module which is the same as $E_{i_p}^\bullet$ as a set, and the module structure is defined by $$a_{i_0}\cdot m = G^{-1}_{i_0i_p}(a_{i_0}) m.$$ Then for $G_{ji_0}: \cA_{i_0}(U_{ji_0})\rightarrow \cA_{j}(U_{ji_0})$, we have the induced module map  $$G_{ji_0}(u^{p,r}_{i_0,...,i_p}) :G_{ji_0}G_{i_0i_p}(E_{i_p}^\bullet) \to G_{ji_0}(F_{i_0}^\bullet)$$ over $U_{j,i_0,\ldots ,i_p}$.  

For an $\mathcal{A}_k-$module $M$, the multiplication by $G_{ik}^{-1}(c_{ijk})$ on $M$ defines an $\mathcal{A}_i$-morphism $G_{ij}G_{jk}(M) \to G_{ik}(M)$, which is denoted by $\hat{c}_{ijk}$, or simply again by $c_{ijk}$ if there is no confusion. 
(Note that $G_{ik}^{-1}(c_{ijk}) = G_{jk}^{-1}G_{ij}^{-1}(c_{ijk})$ by applying the equation $G_{ij}G_{jk} = \mathrm{Ad}(c_{ijk})G_{ik}$ to $G_{ik}^{-1}(c_{ijk})$.  Hence this can also be understood as multiplication of $c_{ijk}$ on the $\cA_i$-module $G_{ij}G_{jk}(M)$.)
This is a morphism of $\mathcal{A}_i$-modules because for any element $e \in G_{ij}G_{jk}(M)$,
\begin{align*}\hat{c}_{ijk}(a_i\cdot e) =& \hat{c}_{ijk}(G^{-1}_{jk}G^{-1}_{ij}(a_i)e) = G^{-1}_{ik}(c_{ijk})G^{-1}_{jk}G^{-1}_{ij}(a_i)e\\=&G^{-1}_{ik}(c_{ijk})G^{-1}_{ik}(c^{-1}_{ijk}a_ic_{ijk})e=G^{-1}_{ik}(a_i)G^{-1}_{ik}(c_{ijk})e=a_i\cdot \hat{c}_{ijk}(e).
\end{align*}

Next, we turn to the structure of the complex of coherent sheaves over the algebroid stack, which will be described in terms of twisted complexes. In order to define it, we recall the notions of product and \v{C}ech differential.

\begin{defn}
	Given $u^{p,r}\in C^p(\mathcal{A},\Hom^r(F',F'')), v^{q,s}\in C^q(\mathcal{A},\Hom^s(F,F'))$, we define the product
	\begin{equation}
		(u \cdot v)_{i_0,...,i_{p+q}}^{p+q,r+s} =  (-1)^{qr} u^{p,r}_{i_0,...,i_p} \cup_c v^{q,s}_{i_{p},...,i_{p+q}}. \label{cupproduct}
	\end{equation}
	
	and
	
	\begin{equation}
		u^{p,r}_{i_0,...,i_p} \cup_c v^{q,s}_{i_{p},...,i_{p+q}} =   u^{p,r}_{i_0,...,i_p} G_{i_0i_{p}}(v^{q,s}_{i_{p},...,i_{p+q}}) c^{-1}_{i_0i_pi_{p+q}}.
	\end{equation}
\end{defn}

\begin{defn}
	For $u\in C^\bullet(\mathcal{A},\Hom^\bullet(E,F))$, the \v{C}ech differential is defined as 
	$$(\cd u)_{i_0,...,i_{p+1}} = \sum_{k=1}^p (-1)^k u_{i_0,...\hat{i_k}...,i_{p+1}}.$$ 
	 In particular, $k=0$ and $k=p+1$ are not included in the summation in the definition.
\end{defn}

For the completeness, we will introduce some properties of $\hat{c}_{ijk}$. The reader may skip this part during their first reading. We use $\cdot$ to denote the multiplication between two elements in an algebra and use $\circ$ for the composition of module maps.

\begin{lemma} \label{lem:c1}
	Let $X_l$ be an $\cA_l$-module. The composition $\hat{c}_{ikl}\circ \hat{c}_{ijk}:  G_{ij}G_{jk}G_{kl}(X_l)\rightarrow G_{il}(X_l)$ is given by the multiplication by $G^{-1}_{il}(c_{ijk}\cdot c_{ikl}) \in \cA_l$ on $X_l$.  (Note that as sets, $G_{ij}G_{jk}G_{kl}(X_l)$, $G_{il}(X_l)$ and $X_l$ are all the same.) 
	
\end{lemma}

\begin{proof}
	$\hat{c}_{ikl}\circ \hat{c}_{ijk}(e)  = G^{-1}_{il}(c_{ikl})G^{-1}_{kl}G^{-1}_{ik}(c_{ijk})e= G^{-1}_{il}(c_{ikl})G^{-1}_{il}(c^{-1}_{ikl}c_{ijk}c_{ikl})e=G^{-1}_{il}(c_{ijk}\cdot c_{ikl})\cdot e$.
\end{proof}

\begin{lemma} \label{lem:c2}
	$G_{li}(\hat{c}_{ijk}): G_{li} G_{ij}G_{jk}(X_k)\rightarrow G_{li}G_{ik}(X_k)$ equals to the multiplication by $G_{li}(c_{ijk})$ on the $\cA_l$-module $G_{li}G_{ij}G_{jk}(X_k)$.
\end{lemma}

\begin{proof}
	$G_{li}(\hat{c}_{ijk})(e) = \hat{c}_{ijk}(e) = G^{-1}_{jk}G^{-1}_{ij}(c_{ijk})e=G^{-1}_{jk}G^{-1}_{ij}G^{-1}_{li}G_{li}(c_{ijk})e$ which equals to acting $G_{li}(c_{ijk})$ on $e \in G_{li}G_{ij}G_{jk}(X_k)$ as $\cA_l$-module.
\end{proof}

Applying the above two lemmas,
$$\hat{c}_{ikl}\circ \hat{c}_{ijk}(e) = G_{il}^{-1}(c_{ijk}\cdot c_{ikl})e= G_{il}^{-1}(G_{ij}(c_{jkl})\cdot c_{ijl})e = \hat{c}_{ijl}\circ G_{ij}(\hat{c}_{jkl})(e).$$
For our purpose later, we take the inverse of this equation:

\begin{corollary} \label{cor:c}
	$G_{ij}(\hat{c}^{-1}_{jkl}) = \hat{c}^{-1}_{ijk}\circ \hat{c}^{-1}_{ikl}\circ \hat{c}_{ijl}$.
\end{corollary}

\begin{lemma}\label{thm:associativity}
	Given any $s,p,q,r$ and $\mathcal{A}_q$-morphism $w:G_{qr}(X_r)\rightarrow X_q$, $$\hat{c}_{spq}\circ G_{sp}G_{pq}(w) \circ \hat{c}^{-1}_{spq}=G_{sq}(w) : G_{sq}G_{qr}(X_r)\rightarrow G_{sq}(X_q).$$ Furthermore, 
	\begin{equation}
		\label{eqn:associativity}
		\hat{c}_{spq} \circ (G_{sp}G_{pq}(w)) \circ G_{sp}(\hat{c}^{-1}_{pqr}) \circ \hat{c}^{-1}_{spr}=G_{sq}(w) \circ \hat{c}^{-1}_{sqr}
	\end{equation}
	as $\cA_s$-morphisms $G_{sr} (X_r) \to G_{sq}(X_q)$.
\end{lemma}

\begin{proof}
	Given any $e \in G_{qr}(X_r) = G_{sq}G_{qr}(X_r)$,
	\begin{align*}
		&
		\quad \hat{c}_{spq}\circ (G_{sp}G_{pq}(w)) \circ \hat{c}^{-1}_{spq}(e)
		\\
		&
		=G_{sq}^{-1}(c_{spq}) w(G^{-1}_{sq}(c^{-1}_{spq})e)
		\\
		&
		=G^{-1}_{pq}(G^{-1}_{sp}(c_{spq})) \cdot w(G^{-1}_{sq}(c^{-1}_{spq})e)
		\\
		&
		= w\left(G^{-1}_{pq}(G^{-1}_{sp}(c_{spq})) \cdot G^{-1}_{sq}(c^{-1}_{spq})e\right)
		\\
		&
		=w\left(G^{-1}_{sq}(c^{-1}_{spq}c_{spq}c_{spq})G^{-1}_{sq}(c^{-1}_{spq})e\right) \, \text{ since } G_{pq}^{-1} \circ G_{sp}^{-1}=G_{sq}^{-1} \circ \mathrm{Ad}(c_{spq}^{-1})
		\\
		&
		= w(e).
	\end{align*}
	Thus we get $\hat{c}_{spq}\circ G_{sp}G_{pq}(w) \circ \hat{c}^{-1}_{spq}=G_{sq}(w)$.  By composing the equality with $\hat{c}^{-1}_{sqr}$ on the right and applying Corollary \ref{cor:c}, we get the required equation.
	
\end{proof}



From now on, we will take the abuse of notation of writing the morphism $\hat{c}_{ijk}$ as $c_{ijk}$.

\begin{prop} \label{prop:associative}
	The product defined by Equation ~\ref{cupproduct} is associative.
\end{prop}

\begin{proof}
	We can ignore signs for the moment, since we know the cup product is associative without $G$ and $c$; including $G,c$ does not affect signs. 
	\begin{align*}
		&
		\quad(u\cdot (v\cdot w))_{i_0\ldots i_r}
		\\
		=&
		\sum_p u_{i_0\ldots i_p}G_{i_0i_p}(v\cdot w)_{i_p\ldots i_r}c^{-1}_{i_0i_pi_r}
		\\
		=&
		\sum_{p\leq q} u_{i_0\ldots i_p}G_{i_0i_p}(v_{i_p\ldots i_q}G_{i_pi_q}(w_{i_q\ldots i_r})c^{-1}_{i_pi_qi_r})c^{-1}_{i_0i_pi_r}
		\\
		=&
		\sum_{p\leq q} u_{i_0\ldots i_p}G_{i_0i_p}(v_{i_p\ldots i_q})\cdot c^{-1}_{i_0i_pi_q}c_{i_0i_pi_q}\cdot \left(G_{i_0i_p}G_{i_pi_q}(w_{i_q\ldots i_r})\right)G_{i_0i_p}(c^{-1}_{i_pi_qi_r})c^{-1}_{i_0i_pi_r}
		\\
		=&
		\sum_{q} (u\cdot v)_{i_0\ldots i_q}G_{i_0i_q}(w_{i_q\ldots i_r})c^{-1}_{i_0i_qi_r}
		\qquad \text{by Equation \eqref{eqn:associativity}}
		\\
		=&((u\cdot v)\cdot w)_{i_0\ldots i_r}.
	\end{align*}
\end{proof}

\begin{defn}
	A twisting complex is a collection of graded sheaves $E^\bullet$ over the algebroid stack $\cA$, together with  an element $a \in C^\bullet(\mathcal{A}, \Hom^\bullet(E,E))$ with total degree being 1 that satisfies the Maurer-Cartan equation
	\begin{equation}
		\cd a + a\cdot a = 0.
		\label{eq:MC}
	\end{equation}
\end{defn}

Explicitly, the first few equations are:
\begin{align}
	a_i^{0,1}G_{ii}(a_i^{0,1}) = 0,\\
	a_i^{0,1}G_{ii}(a_{ij}^{1,0}) c^{-1}_{iij} + a_{ij}^{1,0}G_{ij}(a_j^{0,1})c^{-1}_{ijj} = 0,\\
	-a_{ik}^{1,0} + 
	a_{ij}^{1,0} G_{ij}(a_{jk}^{1,0}) c^{-1}_{ijk}+
	a_{i}^{0,1} G_{ii}(a_{ijk}^{2,-1}) c^{-1}_{iik}+a_{ijk}^{2,-1} G_{ik}(a_{k}^{0,1}) c^{-1}_{ikk} = 0.
\end{align}
The last equation is the cocycle condition, which is stating that
$
a_{ik}^{1,0}$ and  $a_{ij}^{1,0} G_{ij}(a_{jk}^{1,0})c^{-1}_{ijk}
$ are equal up to homotopy.

For morphisms, $\Hom((E,a),(F,b)) := C^\bullet(\mathcal{A},\Hom^\bullet(E,F))$, which is a bi-graded complex using the \v{C}ech differential and the differential induced by $a_i^{0,1}$ and $b_i^{0,1}$. More precisely, the differential, denoted by $d_\cA$, of a morphism $\phi$ is defined as:
\begin{equation} \label{eq:d-cA}
	d_{\cA}\phi = \cd \phi + b\cdot \phi - (-1)^{|\phi|} \phi \cdot a.
\end{equation}

This form a dg-category of twisted complex, denoted by $\Tw(\mathcal{A})$. For convenience, we also denote $\Mor_{\Tw(\cA)}((E,a),(F,b)) = C^\bullet(\mathcal{A},\Hom^\bullet(E,F))$ by $C^\bullet_\cA(E,F)$, which may also be abbreviated as $C^\bullet_\cA$ where $(E,a)$ and $(F,b)$ are fixed. 

$d_\cA$ contains all the higher terms.  The `usual differential' is the following.

\begin{defn}
	Given a morphism $\phi^{p,q} \in C^\bullet_\cA$, we define 
	$$d\phi^{p,q} := b^0\cdot \phi - (-1)^{|\phi|}\phi\cdot a^0 $$
	where $|\phi|=p+q$ denotes the total degree. 
\end{defn}

Then we can rewrite 
\begin{equation} \label{eq:d_A}
	d_{\mathcal{A}}\phi =d\phi + (b^{>0}\cdot \phi) - (-1)^{|\phi|} (\phi \cdot a^{>0}) + \check{\partial}\phi .
\end{equation}

\begin{lemma}[Leibniz's Rule]
	Given $$\mu\in \Mor_{\mathcal{A}_{i_0}}(G_{i_0i_p}(E'',a''),(E',a'))$$ 
	and 
	$$\nu \in \Mor_{\mathcal{A}_{i_p}}(G_{i_pi_{p+r}}(E,a),(E'',a'')),$$
	we have
	$$d(\mu\cdot \nu) = (d\mu)\cdot \nu + (-1)^{|\mu|}\mu\cdot (d\nu).$$
	
	In particular,
	$$d(\mu^{p,q}_{i_0\ldots i_p}\cup_c \nu^{r,s}_{i_p\ldots i_{p+r}}) = (-1)^r (d\mu^{p,q}_{i_0\ldots i_p})\cup_c \nu^{r,s}_{i_p\ldots i_{p+r}}+(-1)^{|\mu|}\mu^{p,q}_{i_0\ldots i_p} \cup_c (d\nu^{r,s}_{i_p\ldots i_{p+r}})$$
	
\end{lemma}

\begin{proof}
	This is a direct application of associativity of the product.
	$d(\mu\cdot \nu)$ equals to
	\begin{align*}
		& (a')^0 \cdot (\mu \cdot \nu ) - (-1)^{|\mu|+|\nu|}(\mu\cdot \nu)\cdot a^0\\
		=& ((a')^0 \cdot \mu) \cdot \nu - (-1)^{|\mu|+|\nu|}\mu\cdot (\nu \cdot a^0)\\
		=& ((a')^0 \cdot \mu) \cdot \nu - (-1)^{|\mu|}\mu\cdot (a'')^0 \cdot \nu+ (-1)^{|\mu|}\mu\cdot (a'')^0 \cdot \nu - (-1)^{|\mu|+|\nu|}\mu\cdot (\nu \cdot a^0)\\
		=& d\mu \cdot \nu + (-1)^{|\mu|}\mu \cdot d(\nu).
	\end{align*}
	
	Take $\mu^{p,q} = \mu^{p,q}_{i_0\ldots i_p},\nu^{r,s} = \nu^{r,s}_{i_p\ldots i_{p+r}}$ (and zero at all other indices). Then, $(-1)^{qr}d(\mu^{p,q}\cup_c \nu^{r,s})=d(\mu\cdot \nu) = (-1)^{(q+1)r}d(\mu^{p,q})\cup_c \nu^{r,s}+(-1)^{|\mu|}(-1)^{qr}\mu^{p,q} \cup_c d(\nu^{r,s})$. Thus, $d(\mu^{p,q}\cup_c \nu^{r,s}) = (-1)^r d(\mu^{p,q})\cup_c \nu^{r,s}+(-1)^{|\mu|}\mu^{p,q} \cup_c d(\nu^{r,s})$.
\end{proof}

\subsection{Algebroid Stacks for quiver algebras}
\label{section:modified algebroid}

In this subsection, we generalize the definition of an algebroid stack in the context of quiver algebras.  We call this a quiver algebroid stack, see Definition \ref{def:qstack}. To define twisted complexes (Definition \ref{def:twisted}) over a quiver algebroid stack, we need to consider intertwining maps (Definition \ref{def:intertwine}) in place of module morphisms, and define the cup product \eqref{eq:cup-gen} for intertwining maps. We justify the definition by comparing it with the cup product for module maps. Moreover, we generalize the cup product for multiple entries in \eqref{eq:mult}, which is a preparation for the mirror construction of the next section.

We will use this setup for gluing localized mirrors which are quiver algebras. When two quivers have different number of vertices, their associated quiver algebras cannot be isomorphic. This is why we need to generalize the definition of an algebroid stack. We will see that gerbe terms naturally come up in this context and are unavoidable when the quivers have different numbers of vertices.

Sheaves of quiver algebras will be one of the main ingredients. Localization of quiver algebras provides a useful technique to construct them. First, we define invertible elements in a quiver algebra.

\begin{defn} \label{def:unit}
	Let $\cA$ be a quiver algebra and $e_i$ the trivial path at $i$-th vertex. A non-zero element $\gamma \in e_i \cdot \cA \cdot e_j$ is said to be invertible if there exists an element $\beta \in e_j \cdot \cA \cdot e_i$ such that $\gamma \beta = e_i$ and $\beta \gamma = e_j$. $\beta$ is called the inverse of $\gamma$.
	
	More generally, for an element $\gamma \in \cA$, let $I$ be the set of all vertices $i$ such that $e_i \gamma \not= 0$, and $J$ be the set of all vertices $j$ such that $\gamma e_j \not= 0$. In other words $\left( \sum_{i \in I} e_i \right) \gamma \left( \sum_{j \in J} e_j \right) = \gamma$. We define the head and the tail of $\gamma$ to be $e_{h_\gamma} := \left( \sum_{i \in I} e_i \right)$ and $e_{t_\gamma} := \left( \sum_{j \in J} e_j \right)$ respectively (assuming $e_{h_\gamma}$ and $e_{t_\gamma}$ are non-zero, or otherwise they are undefined). $\beta$ is called to be the inverse of $\gamma$ if $\beta \gamma = \sum_{j \in J} e_j$ and $\gamma \beta = \sum_{i \in I} e_i$. In particular $e_{t_\beta} = \sum_{i \in I} e_i$ and $e_{h_\beta} = \sum_{j \in J} e_j$.
	
	The set of all invertible elements in $\cA$ will be denoted by $\cA^{\times}$.
\end{defn}

Next, we define localizations of a quiver algebra $\cA$.

\begin{defn}\label{def:localization}
	Let $S \subset \cA=\C Q/R$ be a finite subset of elements $\gamma$ which are not zero divisors, in the sense that $\gamma x \not=0 \in \cA$ for all $x \in \cA$ with $h_x = t_\gamma$ and $y\gamma\not=0$ for all $t_y = h_\gamma$. For each $\gamma \in S$, we adjoin an element $\gamma^{-1}$ to the quiver algebra with $s(\gamma^{-1})=t(\gamma)$, $t(\gamma^{-1})=s(\gamma)$ and the defining relations  $\gamma \gamma^{-1} = e_{t(\gamma)}, \gamma^{-1} \gamma = e_{s(\gamma)}$.  The resulting algebra is denoted by $\cA (S^{-1})$.
	
	In particular, when $S$ consists of arrows, we adjoin the inverse arrows $a^{-1}$ of $a\in S$ to the quiver $Q$ and the generators $a a^{-1} - e_{t_a}, a^{-1} a - e_{s_a}$ to the ideal of relations.
\end{defn}

\begin{remark}
	The definition of localization of a quiver algebra was also introduced in Section 4.2 of \cite{AH}. It is different to the localization of an associate algebra: the product of an arrow and its inverse equals to the idempotent associated to a vertex instead of 1.
\end{remark}

Now we can define a presheaf $\cA_i$ over a topological space $U_i$ with a base of open subsets $\{U_{i_k}\}$. We assign to each $U_{i_0,\cdots,i_p}$ a subset $S_{i_0,\cdots,i_p}\subset \cA_{i_0}$ such that $S_{I}\subset S_J$ whenever $J\subset I$. We define $\cA_{i_0}(U_{i_0\cdots i_p}) := \cA_{i_0}(S^{-1}_{i_0,\cdots,i_p})$. Then, the restriction maps $\cA_{i_0}(S_J) \to \cA_{i_0}(S_I)$ are given by $a \mapsto a$. 

In this way, each $U_i$ is associated with a presheaf of quiver algebras $\cA_i$, where $\cA_i(U_i)$ is a quiver algebra of $Q^{(i)}$ with relations, and $\cA_i(V)$ are certain localizations at arrows of $Q^{(i)}$ for $V \stackrel{\textrm{open}}{\subset} U_i$. Correspondingly, we have quivers $Q_{V}^{(i)}$ corresponding to these localizations, which are obtained by adding the corresponding reverse arrows to $Q^{(i)}$. For our purpose, we assume the presheaf $\cA_i$ is a sheaf over $U_i$.

Next, we want to generalize the conditions on transition maps. In Definition \ref{def:stack}, we require $G_{ij}(U_{ij}): \cA_j(U_{ij}) \cong \cA_i(U_{ij})$ be isomorphisms.  
Here, we relax the condition and define $G_{ij}(U_{ij})$ as the representation of a quiver algebra by another quiver algebra.  

A representation of a quiver algebra by another quiver algebra means the following, see Definition \ref{def:rep}. First, we associate each vertex $v$ of $Q^{(j)}$ with a vertex $G_{ij}(v)$ of $Q^{(i)}$.  Next, represent each arrow from $v$ to $w$ in $Q_{U_{ij}}^{\left(j\right)}$ by elements in $e_{G_{ij}(w)}\cdot \cA_i(U_{ij}) \cdot e_{G_{ij}(v)}$ such that the relations for the paths are respected upon substitution.  
Note that this is different from a homomorphism $\cA_j(U_{ij}) \to \cA_i(U_{ij})$: for instance, an arrow $a$ with $t(a)\not=h(a)$ can be represented by a loop $x \in \cA_i(U_{ij})$, which cannot be a homomorphism since $e_{t(a)} e_{h(a)}=0$ while $e_{h(x)} e_{t(x)}=e_{h(x)}\not=0$.  On the other hand, a loop at $v$ must be represented by a cycle in $e_{G_{ij}(v)}\cdot \cA_i(U_{ij}) \cdot e_{G_{ij}(v)}$.

A more conceptual way to put $G_{ij}(U_{ij})$ is defining it as an $\cA_j(U_{ij})$-$\cA_i(U_{ij})$ bimodule of the form $\bigoplus_{v\in Q^{(j)}_0} e_{G_{ij}(v)} \cdot \cA_i(U_{ij})$, where $a \in \cA_j(U_{ij})$ acts on the left by left multiplication by $G_{ij}(a)$.


\begin{defn}
	$G_{ij}: \cA_j|_{U_{ij}} \to \cA_i|_{U_{ij}}$ is called a representation of sheaf of quiver algebras over $U_{ij}$ if for every open set $V\subset U_{ij}$, we have a representation $G_{ij}(V)$ of $\cA_j(V)$ over $\mathcal{A}_{i}(V)$, such that $G_{ij}(V)$ restricted to $\cA_j(U_{ij})$ equals to $G_{ij}(U_{ij})$. Sometimes we will call it a representation for short.
\end{defn}

\begin{remark}
	Notice that since $\cA_i$ and $\cA_j$ are sheaves, the representation $G_{ij}(U_{ij})$ can be glued from the local charts (open cover) of $U_{ij}$. On the other hand, since we assume $\cA_j(V)$ is the localization of $\cA_j(U_i)$ for any open subset $V \subset U_i$, $G_{ij}$ is determined by $G_{ij}(U_{ij})$. By abuse of notation, we may also denote $G_{ij}(U_{ij})$ as $G_{ij}$.
\end{remark}

For our purpose, we fix a base vertex $v^{\left(j\right)}$ of $Q^{\left(j\right)}$ for every $j$, and require $G_{ij}$ preserves the base vertices, i.e. $G_{ij}(v^{(j)})$ $=v^{(i)}$ for all $i,j$.  We denote the corresponding trivial paths by $e^{(j)}:=e_{v^{(j)}}$.

Notice that the representations can compose. Given a representation of sheaf of quiver algebras $G_{ij}$ of $\cA_j|_{U_{ij}}$ by $\mathcal{A}_{i}|_{U_{ij}}$, and a representation $G_{jk}$ of $\cA_k|_{U_{jk}}$ by $\cA_j|_{U_{jk}}$, we can restrict to the common intersection $U_{ijk}$ and compose them to get the representation $G_{ij} \circ G_{jk}$ of $\cA_k|_{U_{ijk}}$ over $\cA_i|_{U_{ijk}}$.  We will simply denote it by $G_{ij}\circ G_{jk}$ for simplicity.

The cocycle condition is that $G_{ij}\circ G_{jk}|_{U_{ijk}}$ and $G_{ik}|_{U_{ijk}}$ are isomorphic as representations. Recall that they are determined by $G_{ij}\circ G_{jk}(U_{ijk})$ and $G_{ik}(U_{ijk})$ respectively under the assumption. Thus, being isomorphic means there exists an assignment of $$c_{ijk}\left(v\right)\in \left(e_{G_{ij}(G_{jk}(v))}\cdot\mathcal{A}_{i}(U_{ijk})\cdot e_{G_{ik}(v)}\right)^{\times }$$ to each vertex $v$ of $Q^{\left(k\right)}$, such that
\begin{equation}
	G_{ij}\circ G_{jk}\left(a\right)=c_{ijk}\left(h_{a}\right)\cdot G_{ik}\left(a\right)\cdot c_{ijk}^{-1}\left(t_{a}\right).
	\label{eq:compose-gen}
\end{equation}
This is a change of basis for representations.  Gerbe terms $c_{ijk}$ arise in this way naturally, and unavoidably, since $Q^{(i)},Q^{(j)},Q^{(k)}$ are quivers of different sizes in general and the localized quiver algebras cannot be isomorphic.

In particular, at the base point $v^{(k)}$, $c_{ijk}(v^{(k)})$ is a cycle in $e^{(i)}\cdot\mathcal{A}_{i}(U_{ijk})\cdot e^{(i)}$.

As in the previous section, we assume that
$$ 	c_{ijk}(G_{kl}(v)) c_{ikl}(v) = G_{ij}(c_{jkl}(v))c_{ijl}(v).$$
Besides, $G_{ij}(c_{jkl}(v))$ is taken as $e_{G_{ij}(w)}$ if $c_{jkl}(v)$ is a trivial path at $w$.

\begin{lemma}
	Under the above condition on $c_{ijk}$, $(G_{ij}\circ G_{jk})\circ G_{kl}(a) = G_{ij}\circ (G_{jk}\circ G_{kl})(a)$ for all $a$.
\end{lemma}

Take $i=k$ in Equation \eqref{eq:compose-gen}.  In this paper, we always take $G_{ii}=\Id$.  Then,
\begin{equation*}
	G_{ij}\circ G_{ji}\left(a\right)=c_{iji}\left(h_{a}\right)\cdot a\cdot c_{iji}^{-1}\left(t_{a}\right).
\end{equation*}
This replaces the condition of invertibility for $G_{ij}$.  Note that $$c_{iji}\left(v\right)\in \left( e_{G_{ij}(G_{ji}(v))} \cdot \mathcal{A}_{i}(U_{ijk})\cdot e_v\right)^{\times }$$ for each vertex $v$ of $Q^{\left(i\right)}$.

Take $i=j$ in Equation \eqref{eq:compose-gen}.  Since we assume $G_{ii}=\Id$, we simply get $$G_{jk}\left(a\right)=c_{jjk}\left(h_{a}\right)\cdot G_{jk}\left(a\right)\cdot c_{jjk}^{-1}\left(t_{a}\right).$$
Then $c_{jjk}(v)=1$ for all $v$ satisfies this equation.  We will always take $c_{jjk}\equiv 1$ in this paper.  Similarly, we take $c_{ikk}\equiv 1$.

We summarize as follows.

\begin{defn}\label{def:qstack}
	Let $B$ be a topological space.  A quiver algebroid stack consists of the following data:
	\begin{enumerate}
		\item An open cover $\{U_i: i \in I\}$ of $B$.
		\item
		A sheaf of algebras $\mathcal{A}_i$ over each $U_i$, coming from localizations of a quiver algebra $\cA_i(U_i) = \C Q^{(i)}/ R^{(i)}$.
		\item
		A representation of sheaf of quiver algebras $G_{ij}$ of $\cA_j$ over $\mathcal{A}_i$ for every $i,j$.
		\item
		An invertible element $c_{ijk}\left(v\right)\in \left(e_{G_{ij}(G_{jk}(v))}\cdot\mathcal{A}_{i}(U_{ijk})\cdot e_{G_{ik}(v)}\right)^{\times}$ for every  $i,j,k$ and $v \in Q^{(k)}_0$, that satisfies
		\begin{equation}
			G_{ij}\circ G_{jk}\left(a\right)=c_{ijk}\left(h_{a}\right)\cdot G_{ik}\left(a\right)\cdot c_{ijk}^{-1}\left(t_{a}\right)
			\label{eq:cocycle-gen}
		\end{equation}
		such that for any $i,j,k,l$ and $v$,
		\begin{equation}
			c_{ijk}(G_{kl}(v)) c_{ikl}(v) = G_{ij}(c_{jkl}(v))c_{ijl}(v).
			\label{eq:c-gen}
		\end{equation}
		In this paper, we always set $G_{ii}=\Id, c_{jjk}\equiv 1 \equiv c_{jkk}$.
	\end{enumerate}
\end{defn}

\begin{remark}\label{rem:sheaf}
	In the examples of this paper, we take $B$ to be a polyhedral set, whose open subsets are the complements of faces, to record the local charts and transition maps just like in toric geometry. In particular, the topological space $B$ only contains finitely many open subsets. 
	
	In this case, we can obtain a sheaf of quiver algebras using the following construction. Given a quiver algebra $\cA$. First, we define the sections over the complement of edges $U_e$ by localizing a set of arrows in $\cA$. Similarly for complement of the faces, which form a basis of the topology. We require the localized arrows has no torsion. In other words, given a localized arrow $\gamma$, it has no torsion in $e_{s(y)}\cA$ and $\cA e_{t(y)}.$ This will later make sure the restriction map $\cA(U) \to \oplus \cA(U_\alpha)$ is injective, where $\{U_\alpha\}$ is an open cover of $U$.
	
	Secondly, we define the sections over the intersection of the basis by localizing the union of the localized arrows. Finally, for the union of the above open sets $\{U_\alpha\}$, we define the section to be the Kernel of the alternating sum $\cA_i(U_{\alpha}) \to \oplus_{\alpha,\beta} \cA_i(U_{\alpha\beta})$. One can check that this gives a sheaf of quiver algebras.	
\end{remark}

Below we show an example of noncommutative crepant resolution and an important example of quiver algebroid stack, which will be the main focus in the application part of this paper. 

\begin{example}[NC local projective plane as an algebra] \label{ex:nclocP2}
	Consider the quiver $Q$ given on the right of Figure \ref{fig:C3Q3}.  We have the quiver algebra $\A=\C Q/R$, where the ideal $R$ are generated by $a_2b_1 - b_2a_1$ and other similar relations, which are the cyclic derivatives of the spacetime superpotential $$(a_3b_2-b_3a_2)c_1 + (a_1b_3-b_1a_3)c_2 + (a_2b_1-b_2a_1)c_3.$$
	$\A$ is derived equivalent to the total space of the canonical line bundle $X=K_{\bP^2}$ \cite{BKR01,vdBergh}, which is the crepant resolution of the orbifold $\C^3/\Z_3$.
	
	$\A$ admits interesting noncommutative deformations.  The simplest one is given by the following deformation of the spacetime superpotential:
	\begin{equation}
		(a_3b_2-e^{\hbar}b_3a_2)c_1 + (a_1b_3-e^{\hbar}b_1a_3)c_2 + (a_2b_1-e^{\hbar}b_2a_1)c_3.
	\end{equation}
	For instance, this gives the commuting relation $a_2 b_1= e^\hbar b_2a_1$.  Let's denote the resulting algebra by $\A^\hbar$.
	
	Indeed, Sklyanin algebras \cite{AS87,ATV} provide an even more interesting class of deformations of $\A$.  Such deformations were constructed in \cite{CHL-nc} using mirror symmetry.  One of the relations take the form $p(\hbar) a_2 b_1 + q(\hbar) b_2a_1 + r(\hbar) c_2 c_1$, where $(p(\hbar),q(\hbar),r(\hbar))$ is given by theta functions and produces an embedding of an elliptic curve in $\bP^2$.
\end{example}

Van den Bergh \cite{vdBergh} showed that the quiver algebra $\A$ is derived equivalent to the usual geometric crepant resolution $X=K_{\bP^2}$. 

\begin{example}[NC local projective plane as a quiver stack] \label{ex:algst}
	Consider three copies of noncommutative $\C^3$ \eqref{eq:ncC3}, denoted by $\cA^{\tilde{\hbar}}_i$ for $i=1,2,3$, which correspond to the three corners of the polyhedral set as shown in Figure \ref{fig:ncKP2}. Later, we will see that they are the nc deformation spaces of some immersed Lagrangians.  We use $(x_1,y_1,w_1)$, $(y_2,z_2,w_2)$ and $(z_3,x_3,w_3)$ to denote their generating variables.
	
	We glue these three copies of nc $\C^3$ with localizations of the quiver algebra 
	$$\cA_0^\hbar := \A^\hbar=\C Q/R^\hbar$$ 
	given in Example \ref{ex:nclocP2}, where the left-right ideal $R^\hbar$ is generated by the cyclic derivatives of 	$(a_3b_2-e^{\hbar}b_3a_2)c_1 + (a_1b_3-e^{\hbar}b_1a_3)c_2 + (a_2b_1-e^{\hbar}b_2a_1)c_3$.  (For instance, $b_1c_3 = e^{\hbar} c_1b_3$, by taking cyclic derivative in $a_2$.)
	
	We take the localizations $$\cA_{0}^\hbar(U_{01}):=\A^\hbar\langle a_1^{-1},a_3^{-1} \rangle,\, \cA_{0}^\hbar(U_{02}):=\A^\hbar\langle c_1^{-1},c_3^{-1} \rangle,\, \cA_{0}^\hbar(U_{03}):=\A^\hbar\langle b_1^{-1},b_3^{-1} \rangle.$$
	Here, $U_{03}$ denote the neighborhoods of the corners of the base polytope, so that the union of $U_{0i}$ for $i=1,2,3$ equals to the polytope.
	
	For the gluing direction $\cA_i^{\tilde{\hbar}} \to \cA_{0}^\hbar(U_{0i})$, we take the homomorphisms defined by:

	\begin{equation}
		G_{01}:\begin{cases}
			x_1 \mapsto c_1a_1^{-1} \\
			y_1 \mapsto b_1a_1^{-1} \\
			w_1 \mapsto a_1a_3a_2;
		\end{cases}
		G_{02}:\begin{cases}
			y_2 \mapsto b_1c_1^{-1}\\
			z_2 \mapsto a_1c^{-1}_1\\
			w_2 \mapsto c_1c_3c_2;
		\end{cases}
		G_{03}:\begin{cases}
			z_3 \mapsto a_1b_1^{-1}\\
			x_3 \mapsto c_1b^{-1}_1\\
			w_3 \mapsto b_1b_3b_2.
		\end{cases}
		\label{eq:G0j}
	\end{equation}
	It can be checked explicitly that the above is a homomorphism, once we set 
	$$\tilde{\hbar} = -3\hbar.$$
	For instance, $x_1y_1 - e^{-3\hbar} y_1x_1=0$ is sent to $c_1a_1^{-1}b_1a_1^{-1} - e^{-3\hbar}b_1a_1^{-1}c_1a_1^{-1}=0$.
	
	However, for the reverse direction, there is no algebra homomorphism $\cA_{0}^\hbar(U_{0i}) \to \cA^{\tilde{\hbar}}_i$.  Thus the gluing cannot make sense using algebra homomorphisms.  Rather, we need to use representations of $\cA_{0}^\hbar(U_{0i})$ over $\cA^{\tilde{\hbar}}_i$, see Definition \ref{def:rep}.
	
	We take the following representation of $\cA_{0}^\hbar(U_{03})$ by $\cA_3^{\hbar}$:
	\begin{equation}
		G_{30}:\begin{cases}
			(a_1,b_1,c_1)\mapsto (z_3,1,x_3)\\
			(a_2,b_2,c_2)\mapsto (e^{\hbar}w_3z_3,w_3,e^{-\hbar}w_3x_3)\\
			(a_3,b_3,c_3)\mapsto (e^{-\hbar}z_3,1,e^{\hbar}x_3).\\
		\end{cases}
	\end{equation}
	The representations $G_{i0}$ of $\cA_{0}^\hbar(U_{0i})$ by $\cA_i^{\hbar}$ for $i=2,1$ are obtained by cyclic permutation $(a,b,c) \mapsto (b,c,a) \mapsto (c,a,b)$ and $(z_3,x_3,w_3) \mapsto (y_2,z_2,w_2)\mapsto (x_1,y_1,w_1)$ respectively.
	
	It is easy to check that $G_{i0} \circ G_{0i} = \Id_{\cA_i^\hbar}$.  However,
	$$ G_{0i} \circ G_{i0} \not= \Id_{\cA_{0}^\hbar(U_{0i})}.$$
	In general, when $\cA_0$ has more vertices than $\cA_i$, such equality cannot hold simply because the representation of vertices is not a bijection.  For instance, 
	$$G_{03} \circ G_{30}(a_2) = e^{\hbar} (b_1b_3b_2)(a_1b_1^{-1}) = b_1b_3\cdot a_2 \not=a_2.$$
	
	Rather, we have
	$$ G_{0i} \circ G_{i0}(a) = c_{0i0}(h_a) G_{00}(a) c_{0i0}^{-1}(t_a) $$
	for all arrows $a$, if we set 
	\begin{align*}
		c_{030}(v_3) =& b_1b_3, c_{030}(v_1) = b_1, c_{030}(v_2) = e_2;\\
		c_{020}(v_3) =& c_1c_3, c_{020}(v_1) = c_1, c_{020}(v_2) = e_2;\\
		c_{010}(v_3) =& a_1a_3, c_{010}(v_1) = a_1, c_{010}(v_2) = e_2.\\	
	\end{align*}
	For instance,
	$$ G_{03} \circ G_{30}(a_3) = e^{-\hbar} a_1b_1^{-1} = b_1\cdot a_3\cdot (b_1b_3)^{-1}.$$
	Thus, gerbe terms $c_{0i0}$ are necessary for gluing quivers with different numbers of vertices.
	
	Now for any $i,j \in \{1,2,3\}$, we define
	$$ G_{ij} := G_{i0} \circ G_{0j} : \cA_j(U_{ij}) \to \cA_i(U_{ij}). $$
	The localizations $\cA_j(U_{ij})$ are the standard toric ones and can be read from the polytope picture (Figure \ref{fig:ncKP2}).  Explicitly, $\cA_1(U_{12}) = \cA_1\langle x_1^{-1} \rangle$ and $\cA_1(U_{13}) = \cA_1\langle y_1^{-1} \rangle$.  The others $\cA_2(U_{2j})$ and $\cA_3(U_{3j})$ are obtained by the substitution $(x_1,y_1) \leftrightarrow (y_2,z_2) \leftrightarrow (z_3,x_3)$.
	
	Then we have
	$$ G_{ij} \circ G_{jk} (x) = G_{i0} \circ (G_{0j} \circ G_{j0}) \circ G_{0k}(x) = G_{i0}\left(c_{0j0}(h_{G_{0k}(x)}) \cdot G_{0k}(x) \cdot c_{0j0}^{-1}(t_{G_{0k}(x)})\right).$$
	Note that in our definition \eqref{eq:G0j} for $G_{0k}$, $G_{0k}(x)$ are loops at vertex 2 for all $x$.  Moreover, $c_{0j0}(v_2)=e_2$.  Hence $c_{0j0}(h_{G_{0k}(x)}) \cdot G_{0k}(x) \cdot c_{0j0}^{-1}(t_{G_{0k}(x)})=G_{0k}(x)$, and we obtain the cocycle condition
	$$ G_{ij} \circ G_{jk} = G_{ik}$$
	for any $i,j,k \in \{1,2,3\}$.  Explicitly, one can check that the gluing maps $G_{ij}$ are the one given in Figure \ref{fig:ncKP2}, producing the noncommutative local $\bP^2$.  This is an example of a noncommutative toric variety.  Deformation quantizations of toric varieties were studied in \cite{CLS1,CLS2}.
	
	In summary, we obtain a quiver algebroid stack consisting of four charts, $\cA_i$ for $i=0,1,2,3$.  If we forget the chart $\cA_0$, then the remaining three charts glue up to an algebroid stack $X^{\tilde{\hbar}}$ that has trivial gerbe term, that is, a sheaf of algebras.  
	
	Interesting phenomena arise as we turn on $\hbar$, due to the existence of a compact divisor.  First, the deformation parameters of the algebra $\A^\hbar$ and the algebroid stack $X^{\tilde{\hbar}}$ are related in the non-trivial way
	$$\tilde{\hbar} = -3\hbar.$$
	Second, the toric gluing also needs to be deformed (by the factor $e^{-2\tilde{\hbar}}$ in this example) in order to satisfy the cocycle condition.
	
	
	\begin{figure}[htb!]
		\centering
		\includegraphics[scale=0.6]{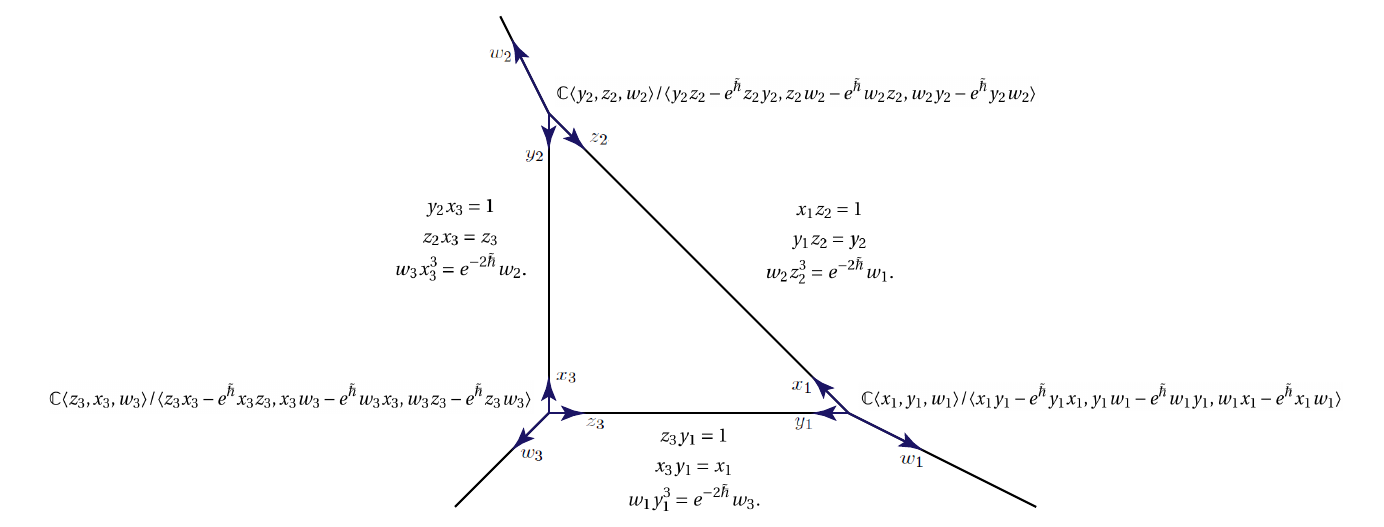}
		\caption{An algebroid stack which is a noncommutative deformation of $K_{\bP^2}$.}
		\label{fig:ncKP2}
	\end{figure}
	
	These non-trivial factors only manifest when we turn on the deformation $\hbar\not=0$.
\end{example}

The quiver algebra $\A$ in the above example (quiver resolution of the orbifold $\C^3/\Z_3$ and its nc deformations) is the formal deformation space of a Lagrangian immersion in a three-punctured elliptic curve \cite{CHL-nc}, which has mirror symmetry meaning.  In Section \ref{chapter:KP2},  we will see that taking affine charts of $\A$ is mirror to a pair-of-pants decomposition of the three-punctured elliptic curve. Furthermore, the nc $\C^3$ is the deformation space of the Seidel Lagrangian in the pair-of-pant.  

\begin{remark}
	It is natural to ask what derived equivalence between a commutative crepant resolution and a noncommutative crepant resolution corresponds to on the mirror symplectic side. We propose that this equivalence can be constructed from isomorphisms between two different classes of immersed Lagrangians on the mirror side.
	
	In \cite{CHL-nc}, quiver algebras which are known as quiver crepant resolutions of toric Gorenstein singularities, together with Landau-Ginzburg superpotentials which are central elements of the algebras, were constructed as mirrors of certain Lagrangian immersions $\bL$ in punctured Riemann surfaces.  
	
	On the other hand, usual commutative crepant resolutions (together with superpotentials) were constructed as mirrors by gluing deformation spaces of Seidel's immersed Lagrangians $\cL_i$ \cite{Seidel-g2,Sei-spec} in pair-of-pants decompositions of the surfaces.  Such mirror pairs are Landau-Ginzburg counterparts of the toric Calabi-Yau mirror pairs constructed in \cite{CLL,AAK} using wall-crossing.  Homological mirror symmetry for these mirror pairs was proved by \cite{Lee,Bocklandt}.
	
	In this paper, we find an isomorphism between the immersed Lagrangian $\bL$ that produces quiver crepant resolutions, and the Seidel Lagrangians $\cL_i$ in a pair-of-pants decomposition, in mirrors of crepant resolutions of $\C^3/\Z_3$.  The advantage of the mirror approach is that, the equivalence that it produces naturally extends to deformation quantizations of the crepant resolutions, which correspond to non-exact deformations on the symplectic side.  The method is general, and we will study other toric Calabi-Yau manifolds in a future paper.
\end{remark}

Now let's define the twisted complexes over the quiver algebroid stack.  In the previous section, 
$C_{i}(U_{ij})$, an $\mathcal{A}_{i}(U_{ij})$-module, can be treated as $\mathcal{A}_{j}(U_{ij})$-module via $G_{ij}$, and the transition map $$\phi _{ji}\colon C_{i}(U_{ij})\rightarrow C_{j}(U_{ij})$$ is required to be $\mathcal{A}_{j}(U_{ij})$-module map.  However, in the current generalized setup, $C_{i}(U_{ij})$ can no longer be treated as $\mathcal{A}_{j}(U_{ij})$-module since $G_{ij}$ is no longer an algebra map.  We consider the following instead.

\begin{defn} \label{def:intertwine}
	Let $C_1$ and $C_2$ be modules of $\cA_1$ and $\cA_2$ respectively.  A $\C$-linear map $\phi_{21}$ is said to be intertwining if
	\begin{equation*}
		\phi _{21}\left(h\cdot x\right)=G_{21}\left(h\right)\cdot \phi _{21}\left(x\right)
	\end{equation*}
	for all $h\in \mathcal{A}_{1}\left(U_{12}\right)$.
\end{defn}

One can check that the space of intertwining chain maps between $\cA_1$ and $\cA_2$-modules forms a vector space.  This is defined to be the morphism space.

In the remaining part of this subsection, we will compare the intertwining maps with module maps we use in the last section and develop some operators we would use in the enlarged Fukaya category. To connect with module maps, we can enlarge $C_{i}(U_{ij})$ to make an $\mathcal{A}_{j}(U_{ij})$-module $\hat{G}_{ji}(C_{i}(U_{ij}))$ as follows.
Define 
$$\hat{G}_{ji}(C_i(U_{ij})) := \left(C_{i}(U_{ij})\right)^{\oplus \left| Q_{0}^{\left(j\right)}\right| },$$
which is endowed with a structure of $\mathcal{A}_{j}(U_{ij})$-module:
$$a\cdot \left(x_{v\in Q_{0}^{\left(j\right)}}\right)\coloneqq \left(G_{ij}\left(a\right)x_{t\left(a\right)}\right)_{h\left(a\right)}. $$ Here $Q_0^{j}$ stands for the set of vertices in $Q^{j}.$
\begin{lemma}
	The above defines a $\cA_{j}(U_{ij})$-module $\hat{G}_{ji}(C_i(U_{ij}))$.
\end{lemma}
\begin{proof}
	$$
	b\cdot a\cdot \left(x_{v\in Q_{0}^{\left(j\right)}}\right)=\left(G_{ij}\left(b\right)G_{ij}\left(a\right)x_{t\left(a\right)}\right)_{h\left(b\right)}=\left(ba\right)\cdot \left(x_{v\in Q_{0}^{\left(j\right)}}\right) 
	$$
	if $t\left(b\right)=h\left(a\right)$, and both sides are zero otherwise.
\end{proof}

Then $\phi _{ji}\colon C_{i}(U_{ij})\rightarrow C_{j}(U_{ij})$ induces a map
$\hat{\phi}_{ji}:\hat{G}_{ji}(C_i(U_{ij}))\rightarrow C_{j}(U_{ij})$ by 
\begin{equation}
	\hat{\phi}_{ji}\left(x_{v}: v \in Q_{0}^{\left(j\right)} \right) :=\sum _{v\in Q_{0}^{\left(j\right)}}c_{jij}^{-1}\left(v\right)\cdot \phi _{ji}\left(x_{v}\right).
	\label{eq:phi-hat}
\end{equation}

\begin{prop}
	The induced linear map $\hat{\phi }_{ji}$ is an $\mathcal{A}_{j}(U_{ij})$-module map iff $\phi _{ji}$ is intertwining.
\end{prop}

\begin{proof}
	Suppose $\phi _{ji}$ is intertwining.
	\begin{align*}
		\hat{\phi }_{ji}\left(a\cdot \left(x_{v}\right)\right)&=\hat{\phi }_{ji}\left(\left(G_{ij}\left(a\right)x_{t\left(a\right)}\right)_{h\left(a\right)}\right)=c_{jij}^{-1}\left(h\left(a\right)\right)\phi _{ji}\left(G_{ij}\left(a\right)x_{t\left(a\right)}\right) \\
		&=c_{jij}^{-1}\left(h\left(a\right)\right)G_{ji}\left(G_{ij}\left(a\right)\right)\cdot \phi _{ji}\left(x_{t\left(a\right)}\right)=ac_{jij}^{-1}\left(t\left(a\right)\right)\cdot \phi _{ji}\left(x_{t\left(a\right)}\right) 
	\end{align*}
	which equals to
	\begin{equation*}
		a\cdot \hat{\phi }_{ji}\left(\left(x_{v}\right)\right)=ac_{jij}^{-1}\left(t\left(a\right)\right)\cdot \phi _{ji}\left(x_{t\left(a\right)}\right).
	\end{equation*}
	The converse is based on the same calculation.
\end{proof}

We make the following useful observation.
\begin{lemma} \label{lem:right-mul}
	If $C_i = \bigoplus_p \cA_i \cdot e_{v_p}$ and $C_j = \bigoplus_q \cA_j \cdot e_{v_q}$, and the components of $\phi _{ji}\left(x\right) \in C_j$ are given as a sum of terms in the form
	\begin{equation*}
		G_{ji}\left(x_p\cdot y\right)\cdot a
	\end{equation*}
	for some $y\in \mathcal{A}_{i}(U_{ij})$ and $a\in \mathcal{A}_{j}(U_{ij})$ (and $x_p$ are the components of $x\in C_i$), then 
	$\phi _{ji}\left(x\right)$ is intertwining.
\end{lemma}

The relation between intertwining maps and module maps is delicate.  An intertwining map $\phi_{ji}$ lifts as a module map $\hat{\phi}_{ji}$.  In the reverse way, given a map $$\psi_{ji}: \hat{G}_{ji}(C_i(U_{ij}))\rightarrow C_{j},$$ we can always restrict to define
$$(\psi_{ji})_\# :=c_{jij}(v^{(j)})\cdot \psi_{ji}|_{(C_i(U_{ij}))_{v^{(j)}}}: C_i(U_{ij}) \to C_j(U_{ij}).$$  
However, $\psi_{ji}$ being an $\mathcal{A}_{j}(U_{ij})$-module map does not imply that $(\psi_{ji})_\#$ is intertwining.  It is obvious that $(\hat{\phi}_{ji})_\# = \phi_{ji}$.  But it is not necessarily true that $\widehat{(\psi_{ji})_\#} = \psi_{ji}$.

To have a better relation, consider the situation that $$Q_{0}^{\left(j\right)}= \left\{v\in Q_{0}^{\left(j\right)}\colon G_{ji}\left(G_{ij}\left(v\right)\right)=v^{\left(j\right)}\right\}.$$ (This is always the case when $Q^{\left(i\right)}$ consists of a single vertex $v^{(i)}$.)

\begin{prop}
	Assume that $Q_{0}^{\left(j\right)}= \left\{v\in Q_{0}^{\left(j\right)}\colon G_{ji}\left(G_{ij}\left(v\right)\right)=v^{\left(j\right)}\right\}$.  If $$\psi_{ji}: \hat{G}_{ji}(C_i(U_{ij}))\rightarrow C_{j}(U_{ij})$$ is an $\mathcal{A}_{j}(U_{ij})$-module map and $(\psi_{ji})_\#$ is intertwining, then $\psi_{ji} = \widehat{(\psi_{ji})_\#}$.  In other words, the space of intertwining maps $C_i(U_{ij}) \to C_{j}(U_{ij})$ equals to the space of those module maps $\psi_{ji}:\hat{G}_{ji}(C_i(U_{ij}))\rightarrow C_{j}(U_{ij})$ with $(\psi_{ji})_\#$ being intertwining.
\end{prop}

\begin{proof}
	Since for any $v\in Q^{(j)}_0$, $G_{ji}\left(G_{ij}\left(v\right)\right)=v^{\left(j\right)}$,  we have
	$c_{jij}\left(v\right)\in \left(v^{\left(j\right)}\cdot \mathcal{A}_{j,\left\{ijk\right\}}\cdot v\right)^{\times }$ and
	$$G_{ji}\circ G_{ij}\left(a\right)=c_{jij}\left(h_{a}\right)\cdot a\cdot c_{jij}^{-1}\left(t_{a}\right)\in v^{\left(j\right)}\mathcal{A}_{j,\left\{ijk\right\}}v^{\left(j\right)}.$$
	In particular, $G_{ji}\circ G_{ij}\left(c_{jij}\left(v\right)\right)=c_{jij}\left(v^{\left(j\right)}\right).$
	
	Let $\phi _{ji}'\left(x\right)\coloneqq \psi _{ji}\left(\left(x\right)_{{v^{\left(j\right)}}}\right)=c^{-1}_{jij}(v^{(j)})\cdot (\psi_{ji})_\#$.  It is intertwining by assumption.  Since $\psi_{ji}$ is a module map,
	\begin{equation*}
		c_{jij}^{-1}\left(v\right)\phi _{ji}'\left(x\right)=c_{jij}^{-1}\left(v\right)\psi _{ji}\left(\left(x\right)_{{v^{\left(j\right)}}}\right)=\psi _{ji}\left(c_{jij}^{-1}\left(v\right)\cdot \left(x\right)_{{v^{\left(j\right)}}}\right)=\psi _{ji}\left(G_{ij}\left(c_{jij}^{-1}\left(v\right)\right)x\right)_{v}.
	\end{equation*}
	Replacing $x$ by $G_{ij}\left(c_{jij}\left(v\right)\right)x$, we get
	\begin{equation*}
		c_{jij}^{-1}\left(v\right)\phi _{ji}'\left(G_{ij}\left(c_{jij}\left(v\right)\right)x\right)=\psi _{ji}\left(\left(x\right)_{v}\right).
	\end{equation*}
	On the other hand,
	\begin{equation*}
		c_{jij}^{-1}\left(v\right)\phi _{ji}'\left(G_{ij}\left(c_{jij}\left(v\right)\right)x\right)=c_{jij}^{-1}\left(v\right)G_{ji}\left(G_{ij}\left(c_{jij}\left(v\right)\right)\right)\phi _{ji}'\left(x\right)=c_{jij}^{-1}\left(v\right)c_{jij}\left(v^{\left(j\right)}\right)\phi _{ji}'\left(x\right).
	\end{equation*}
	Thus, $\psi _{ji}\left(\left(x\right)_{v}\right)=c_{jij}^{-1}\left(v\right)c_{jij}\left(v^{\left(j\right)}\right)\phi _{ji}'\left(x\right)$.
	That is, $\psi_{ji} = \widehat{(\psi_{ji})_\#}$.
\end{proof}

Now we get back to the general situation (that $Q_{0}^{\left(j\right)}$ may not equal to $$\left\{v\in Q_{0}^{\left(j\right)}\colon  G_{ji}\left(G_{ij}\left(v\right)\right)=v^{\left(j\right)}\right\}).$$
The higher terms $\phi _{I}\colon C_{{i_{k}}}(U_I)\rightarrow C_{{i_{0}}}(U_I)$ (which are graded $\C$-linear maps) in defining a twisted complex are also required to be intertwining.  Then it induces the $\cA_{i_0}(U_I)$-module map
$$\hat{\phi} _{I}:\hat{G}_{i_0i_k}(C_{i_k}(U_I))\rightarrow C_{{i_{0}}}(U_I)$$ 
(where $\hat{\phi}_{I}$ is defined from $\phi_I$ by \eqref{eq:phi-hat}).

Let $I=(i_0,\ldots,i_k)$ and $I' = (i_k,\ldots,i_l)$.
Given intertwining maps $\phi _{I}\colon C_{{i_{k}}}(U_I)\rightarrow C_{{i_{0}}}(U_I)$ and $\psi _{I'}\colon C_{{i_{l}}}(U_{I'})\rightarrow C_{{i_{k}}}(U_{I'})$, we can take their composition $$\phi_I \circ \psi_{I'}: C_{{j_{l}}}(U_{I\cup I'}) \to C_{{i_{0}}}(U_{I\cup I'}).$$  Unfortunately, $\phi_I \circ \psi_{I'}$ is not intertwining.  Rather,
\begin{align*}
	&\phi _{I}\circ \psi _{I'}\left(ax\right)\\=&G_{{i_{0}}{i_{k}}}\left(G_{{i_{k}}{i_{l}}}\left(a\right)\right)\phi _{I}\circ \psi _{I'}\left(x\right)\\=&c_{{i_{0}}{i_{k}}{i_{l}}}\left(h_{a}\right)G_{{i_{0}}{i_{l}}}\left(a\right)c_{i_{0}i_{k}i_{l}}^{-1}\left(t_{a}\right)\phi _{I}\circ \psi _{I'}\left(x\right)\neq G_{{i_{0}}{i_{l}}}\left(a\right)\phi _{I}\circ \psi _{I'}\left(x\right).
\end{align*}
The above calculation tells us how to modify to make it intertwining. Namely, let $C_{i_l} = \bigoplus_{p=1}^N \cA_{i_l} e_{v_p}$ for some vertices $v_p \in Q_0^{(i_l)}$, and let $(X_1,\ldots,X_N)$ be the standard basis.  Write $x=\sum_p x_p X_p$.  Then take
\begin{equation}
	\phi _{I}\cup \psi _{I'}\left(x\right):= \sum _{p}c_{i_{0}i_{k}i_{l}}^{-1}\left(h_{{x_{p}}}\right)\phi _{I}\circ \psi _{I'}\left(x_{p}X_p\right).
	\label{eq:cup-gen}
\end{equation}

\begin{prop}
	The above defined $\phi _{I}\cup \psi _{I'}$ is intertwining.
\end{prop}
\begin{proof}
	\begin{align*}
		\phi _{I}\cup \psi _{I'}\left(x\right)=
		&\sum_{p}c_{i_{0}i_{k}i_{l}}^{-1}\left(h_{{x_{p}}}\right)G_{{i_{0}}{i_{k}}}\left(G_{{i_{k}}{i_{l}}}\left(x_{p}\right)\right)\phi _{I}\circ \psi _{I'}\left(X_p\right) \\
		=& \sum_{p}G_{{i_{0}}{i_{l}}}\left(x_{p}\right)c_{i_{0}i_{k}i_{l}}^{-1}\left(t_{{x_{p}}}\right)\phi _{I}\circ \psi _{I'}\left(X_p\right).
	\end{align*}
	Thus,
	\begin{align*}
		\phi _{I}\cup \psi _{I'}\left(ax\right)=&\sum_{p}G_{{i_{0}}{i_{l}}}\left(ax_{p}\right)c_{i_{0}i_{k}i_{l}}^{-1}\left(t_{{x_{p}}}\right)\phi _{I}\circ \psi _{I'}\left(X_p\right)\\
		=&\sum_{p}G_{{i_{0}}{i_{l}}}\left(a\right)G_{{i_{0}}{i_{l}}}\left(x_{p}\right)c_{i_{0}i_{k}i_{l}}^{-1}\left(t_{{x_{p}}}\right)\phi _{I}\circ \psi _{I'}\left(X_p\right)=G_{{i_{0}}{i_{l}}}\left(a\right)\phi _{I}\cup \psi _{I'}\left(x\right).
	\end{align*}
\end{proof}
To simplify, we may write the short form $\phi _{I}\cup \psi _{I'}\left(x\right)=c_{i_{0}i_{k}i_{l}}^{-1}\left(h_{x}\right)\phi _{I}\circ \psi _{I'}\left(x\right)$. However, note that $x$ is a module element rather than an element in $\cA^{(i_l)}$, and we need to write in basis like above in order to talk about $h_{x}$.

This can also be deduced in a systematic way like in last section, by considering the composition
$
\hat{\phi }_{I}\circ \hat{G}_{{i_{0}}{i_{k}}}\left(\hat{\psi }_{I'}\right)\circ \zeta _{{i_{0}}{i_{k}}{i_{l}}}
$
as explained below.

Given the module maps $\hat{\phi}_{I}\colon \hat{G}_{i_0i_k}(C_{{i_{k}}}(U))\rightarrow C_{{i_{0}}}(U)$ and $\hat{\psi} _{I'}\colon \hat{G}_{i_ki_l}(C_{{i_{l}}}(U))\rightarrow C_{{i_k}}(U)$ where $U=U_{I\cup I'}$, we have the $\cA_{i_0}$-module map $$\hat{G}_{i_0i_k}(\hat{\psi}_{I'}): \hat{G}_{i_0i_k}(\hat{G}_{i_ki_l}(C_{{i_{l}}}(U))) \to \hat{G}_{i_0i_k}(C_{{i_k}}(U)),$$ where $\hat{G}_{i_0i_k}(\hat{G}_{i_ki_l}(C_{{i_{l}}}(U))) = (C_{{i_{l}}}(U))^{\oplus  Q_{0}^{\left(i_k\right)}\times  Q_{0}^{\left(i_0\right)}}$, and $\hat{G}_{i_0i_k}(\hat{\psi}_{I'})$ is simply taking $\hat{\psi}_{I'}$ on each component labeled by an element in $Q_{0}^{\left(i_0\right)}$.  By composition, we get an $\cA_{i_0}$-module map $\hat{\phi}_{I} \circ \hat{G}_{i_0i_k}(\hat{\psi}_{I'}): \hat{G}_{i_0i_k}(\hat{G}_{i_ki_l}(C_{{i_{l}}}(U)))\to C_{{i_{0}}}(U)$.  Next, we need to change the domain to $\hat{G}_{i_0i_l}(C_{{i_{l}}}(U))$.

\begin{prop}
	There exist $\cA_i$-module maps
	$$ \zeta^-_{ijk}: \hat{G}_{ij}(\hat{G}_{jk}(C_k(U_{ijk}))) \to \hat{G}_{ik}(C_k(U_{ijk}))$$
	given by $\zeta^-_{ijk}\left(x_{v,w}:v \in Q_0^{(j)}, w \in Q_0^{(i)}\right) := \left(c_{kji}^{-1}(w) \cdot x_{G_{ji}(w),w}: w \in Q_0^{(i)}\right)$,
	and
	\begin{align*}
		\zeta_{ijk}: \hat{G}_{ik}(C_k(U_{ijk})) \to& \hat{G}_{ij}(\hat{G}_{jk}(C_k(U_{ijk}))),\\
		\left.\zeta_{ijk}\left(x_u:u \in Q_0^{(i)}\right)\right|_{v,w} :=& \left\{\begin{array}{ll}
			c_{kji}(w)\cdot x_w & \textrm{ if } v = G_{ji}(w) \\
			0 & \textrm{ otherwise.}
		\end{array}\right.
	\end{align*}
	Moreover, $\zeta_{ijk}^- \circ \zeta_{ijk} = \Id$.
\end{prop}

Then we take the composition
$$ \hat{\phi}_{I} \circ \hat{G}_{i_0i_k}(\hat{\psi}_{I'}) \circ \zeta_{i_0i_ki_l}: \hat{G}_{i_0i_l}(C_{i_l}(U)) \to C_{i_0}(U).$$
This is the desired $\cA_{i_0}$-module map.
\begin{prop}
	$\hat{\phi}_{I} \circ \hat{G}_{i_0i_k}(\hat{\psi}_{I'}) \circ \zeta_{i_0i_ki_l}$ equals to the lifting $\widehat{\phi _{I}\cup \psi _{I'}}$.
\end{prop}
\begin{proof}
	As in \eqref{eq:cup-gen}, we take a basis to write $x_w=\sum_p x_{w,p} X_p$.
	By definition,
	$$ \hat{\phi}_{I} \circ \hat{G}_{i_0i_k}(\hat{\psi}_{I'}) \circ \zeta_{i_0i_ki_l}(x_w) = \sum_{w,p} c_{i_0i_ki_0}^{-1}(w) \phi_{I}\left(c_{i_ki_li_k}^{-1}(G_{i_ki_0}(w)) \psi_{I'}(c_{i_li_ki_0}(w) x_{w,p} X_p)\right).$$
	First, we note that $c_{iji}^{-1}(w)$ can be expressed in terms of $c_{iki}^{-1}(w)$:
	$$ c_{iji}^{-1}(w) = c_{iki}^{-1}(w) c_{ijk}^{-1}(G_{ki}(w))G_{ij}(c_{jki}(w))$$
	by taking $i=l$ in \eqref{eq:c-gen}. 
	Next, we use the intertwining property of $\phi_I$ and $\psi_{I'}$.  Also, note that $c_{i_li_ki_0}(w)x_w=0$ if $G_{i_li_0}(w) \not= h(x_{w,p})$. Then the right hand side equals to
	\begin{align*}\sum_{w,p} & c_{i_0i_li_0}^{-1}(w) c_{i_0i_ki_l}^{-1}(h(x_{w,p}))G_{i_0i_k}\left(c_{i_ki_li_0}(w)c_{i_ki_li_k}^{-1}(G_{i_ki_0}(w))G_{i_ki_l}(c_{i_li_ki_0}(w))\right) \\ \phi_{I}\circ \psi_{I'}(x_{w,p}X_p).\end{align*}
	Now we simplify $G_{i_0i_k}\left(c_{i_ki_li_0}(w)c_{i_ki_li_k}^{-1}(G_{i_ki_0}(w))G_{i_ki_l}(c_{i_li_ki_0}(w))\right)$.
	Note that $$c_{i_ki_li_k}^{-1}(G_{i_ki_0}(w))G_{i_ki_l}(c_{i_li_ki_0}(w)) = c_{i_ki_ki_0}(w)c_{i_ki_li_0}^{-1}(w)=c_{i_ki_li_0}^{-1}(w)$$
	by taking $k=i$ in \eqref{eq:c-gen}.  Thus
	$$ G_{i_0i_k}\left(c_{i_ki_li_0}(w)c_{i_ki_li_k}^{-1}(G_{i_ki_0}(w))G_{i_ki_l}(c_{i_li_ki_0}(w))\right)=1.$$
	Thus,
	$$ \hat{\phi}_{I} \circ \hat{G}_{i_0i_k}(\hat{\psi}_{I'}) \circ \zeta_{i_0i_ki_l}(x_w) = \sum_{w,p}  c_{i_0i_li_0}^{-1}(w)\cdot c_{i_0i_ki_l}^{-1}(h(x_{w,p})) \phi_{I}\circ \psi_{I'}(x_{w,p}X_p) $$
	and the right hand side is exactly $\widehat{\phi _{I}\cup \psi _{I'}}$.
\end{proof}

Once we have $\phi_I \cup \psi _{I'}$, we can define $\phi_I \cdot \psi _{I'}$ as in Equation \eqref{cupproduct} and the twisted complexes over a quiver algebroid stack. 
\begin{defn}\label{def:twisted}
	A twisted complex $(C_\bullet,a)$ over a quiver algebroid stack $\cX$ is a collection of graded projective modules $C_\bullet(U_i)$ (locally direct summands of free modules) over $U_i$, together with a collection of intertwining maps $a_I^{p,q}$ that satisfy the Maurer-Cartan equation \eqref{eq:MC}.  
\end{defn} Similarly morphisms of twisted complexes are defined as in the last section.  The essential changes are replacing module maps by intertwining maps, and defining their product by \eqref{eq:cup-gen}.

In concrete applications, the product is given as follows, which can be checked directly using \eqref{eq:cup-gen}.

\begin{lemma}
	Let $C_{m} = \bigoplus_{p=1}^{N_m} \cA_{m} \cdot e_{v_{p}^{(m)}}$ for $m=i,j,k$, and write every element in terms of the standard basis.  Let
	\begin{align*}
		\phi _{ij}\left(x_s\right) &= \left(\sum_{s=1}^{N_j} G_{ij}\left(x_s\cdot a^{(j)}_{rs}\right)\cdot a^{(i)}_{rs}\right)_{r=1}^{N_i},\\
		\psi _{jk}\left(y_t\right) &= \left(\sum_{t=1}^{N_k} G_{jk}\left(y_t\cdot b^{(k)}_{st}\right)\cdot b^{(j)}_{st}\right)_{s=1}^{N_j}
	\end{align*}
	for some $a^{(i)}_{rs}\in \mathcal{A}_{i}(U_{ijk}),\, a^{(j)}_{rs},b^{(j)}_{st}\in \mathcal{A}_{j}(U_{ijk}),\,b^{(k)}_{st}\in \mathcal{A}_{k}(U_{ijk})$.  Then
	$$ \phi _{ij} \cup \psi _{jk} (y_t) = \left(\sum_{s,t=1}^{N_j,N_k} G_{ik}(y_t b^{(k)}_{st}) c^{-1}_{ijk}(t_{b^{(k)}_{st}}) G_{ij}(b^{(j)}_{st}a^{(j)}_{rs})a^{(i)}_{rs}\right)_{r=1}^{N_i}.$$
\end{lemma}

\begin{remark}
	In applications, we take $a^{(i)}_{rs}\in e^{(i)}\mathcal{A}_{i}(U_{ijk})$, $a^{(j)}_{rs} \in \mathcal{A}_{j}(U_{ijk}) e^{(j)}$,
	$b^{(j)}_{st}\in e^{(j)}\mathcal{A}_{j}(U_{ijk})$,$b^{(k)}_{st}\in \mathcal{A}_{k}(U_{ijk})e^{(k)}$.  In particular, $t_{b^{(k)}_{st}}=e^{(k)}$.  If the gerbe term at base vertex $c^{-1}_{ijk}(e^{(k)})$ is taken to be $1$, the above product formula becomes $G_{ik}(y_t b^{(k)}_{st})  G_{ij}(b^{(j)}_{st}a^{(j)}_{rs})a^{(i)}_{rs}$.
\end{remark}

In general, for $\cA_0,\ldots,\cA_k$, let $U=U_{0,\ldots,k}$, and define $\mathcal{M}_{k,\ldots ,0}: \cA_k(U) \otimes \ldots \otimes \cA_0(U) \to \cA_0(U)$,
\begin{equation}
	\mathcal{M}_{k,\ldots ,0}\left(z^{\left(k\right)}\otimes \ldots \otimes z^{\left(0\right)}\right)\coloneqq G_{0k}\left(z^{\left(k\right)}\right)c_{0,k-1,k}^{-1}\left(t_{{z^{\left(k\right)}}}\right)G_{0,k-1}\left(z^{\left(k-1\right)}\right)\ldots c_{012}^{-1}\left(t_{{z^{\left(2\right)}}}\right)G_{01}\left(z^{\left(1\right)}\right)z^{\left(0\right)}.
	\label{eq:mult}
\end{equation}

\begin{prop} \label{prop:comp}
	Take any $0 \leq p<q \leq k$.  Let $y^{(i)}, z^{(i)} \in \cA_i(U)$ with $t_{y^{(i)}} = h_{z^{(i)}}$ for $i=0,\ldots,k$.  Then the product $\mathcal{M}_{k,\ldots ,0}\left(y^{\left(k\right)}z^{\left(k\right)}\otimes \ldots \otimes y^{\left(0\right)}z^{\left(0\right)}\right)$ equals to the decomposition
	{\small
		\begin{equation*}
			\mathcal{M}_{k,\ldots ,q,p,\ldots ,0}\left(y^{\left(k\right)}z^{\left(k\right)}\otimes \ldots \otimes y^{\left(q\right)}\otimes \mathcal{M}_{q,\ldots ,p}\left(z^{\left(q\right)}\otimes y^{\left(q-1\right)}z^{\left(q-1\right)}\otimes\ldots \otimes y^{\left(p\right)}\right)z^{\left(p\right)}\otimes \ldots \otimes y^{\left(0\right)}z^{\left(0\right)}\right).
		\end{equation*}
	}
\end{prop}

\begin{proof}
	$\mathcal{M}_{k,\ldots ,0}\left(y^{\left(k\right)}z^{\left(k\right)}\otimes \ldots \otimes y^{\left(0\right)}z^{\left(0\right)}\right)$ equals to
	\begin{multline*}
		G_{0k}\left(y^{\left(k\right)}z^{\left(k\right)}\right)c_{0,k-1,k}^{-1}\left(t_{{z^{\left(k\right)}}}\right)G_{0,k-1}\left(y^{\left(k-1\right)}z^{\left(k-1\right)}\right)\ldots G_{0,q}\left(y^{\left(q\right)}\right)\\
		\cdot \phi '\cdot G_{0,p}\left(z^{\left(p\right)}\right)c_{0,p-1,p}^{-1}\left(t_{{z^{\left(p\right)}}}\right)\ldots c_{012}^{-1}\left(t_{{z^{\left(2\right)}}}\right)G_{01}\left(y^{\left(1\right)}z^{\left(1\right)}\right)y^{\left(0\right)}z^{\left(0\right)}
	\end{multline*}
	where
	\begin{equation*}
		\phi '=G_{0,q}\left(z^{\left(q\right)}\right)c_{0,q-1,q}^{-1}\left(t_{{z^{\left(q\right)}}}\right)G_{0,q-1}\left(y^{\left(q-1\right)}z^{\left(q-1\right)}\right)\ldots G_{0,p}\left(y^{\left(p\right)}\right).
	\end{equation*}
	We have $G_{0,q}\left(z^{\left(q\right)}\right)c_{0,q-1,q}^{-1}\left(t_{{z^{\left(q\right)}}}\right)=c_{0,q-1,q}^{-1}\left(h_{{z^{\left(q\right)}}}\right)G_{0,q-1}\left(G_{q-1,q}\left(z^{\left(q\right)}\right)\right)$. Thus
	\begin{align*}
		\phi '=&c_{0,q-1,q}^{-1}\left(h_{{z^{\left(q\right)}}}\right)G_{0,q-1}\left(G_{q-1,q}\left(z^{\left(q\right)}\right)y^{\left(q-1\right)}z^{\left(q-1\right)}\right)c_{0,q-2,q-1}^{-1}\left(t_{{z^{\left(q-1\right)}}}\right)\ldots G_{0,p}\left(y^{\left(p\right)}\right)\\ =&c_{0,q-1,q}^{-1}\left(h_{{z^{\left(q\right)}}}\right)c_{0,q-2,q-1}^{-1}\left(h_{{G_{q-1,q}}\left(z^{\left(q\right)}\right)}\right)\\
		\cdot& G_{0,q-2}\left(G_{q-2,q-1}\left(G_{q-1,q}\left(z^{\left(q\right)}\right)y^{\left(q-1\right)}z^{\left(q-1\right)}\right)y^{\left(q-2\right)}z^{\left(q-2\right)}\right)\ldots G_{0,p}\left(y^{\left(p\right)}\right).
	\end{align*}
	Then using
	$$c_{0,q-1,q}^{-1}\left(h_{{z^{\left(q\right)}}}\right)c_{0,q-2,q-1}^{-1}\left(h_{{G_{q-1,q}}\left(z^{\left(q\right)}\right)}\right)=c_{0,q-2,q}^{-1}\left(h_{{z^{\left(q\right)}}}\right)G_{0,q-2}\left(c_{q-2,q-1,q}^{-1}\left(h_{{z^{\left(q\right)}}}\right)\right),$$ we get
	\begin{align*}
		\phi'=&\scriptstyle c_{0,q-2,q}^{-1}\left(h_{{z^{\left(q\right)}}}\right)G_{0,q-2}\left(c_{q-2,q-1,q}^{-1}\left(h_{{z^{\left(q\right)}}}\right)G_{q-2,q-1}\left(G_{q-1,q}\left(z^{\left(q\right)}\right)y^{\left(q-1\right)}z^{\left(q-1\right)}\right)y^{\left(q-2\right)}z^{\left(q-2\right)}\right)\ldots G_{0,p}\left(y^{\left(p\right)}\right)\\
		=&\scriptstyle c_{0,q-2,q}^{-1}\left(h_{{z^{\left(q\right)}}}\right)G_{0,q-2}\left(G_{q-2,q}\left(z^{\left(q\right)}\right)c_{q-2,q-1,q}^{-1}\left(t_{{z^{\left(q\right)}}}\right)G_{q-2,q-1}\left(y^{\left(q-1\right)}z^{\left(q-1\right)}\right)y^{\left(q-2\right)}z^{\left(q-2\right)}\right)\\
		\scriptstyle \cdot &\scriptstyle c_{0,q-3,q-2}^{-1}\left(t_{{z^{\left(q-2\right)}}}\right)\ldots G_{0,p}\left(y^{\left(p\right)}\right).
	\end{align*}
	Keep on doing this, we obtain
	{\small
		\begin{equation*}
			\phi '=c_{0,p,q}^{-1}\left(h_{{z^{\left(q\right)}}}\right)G_{0,p}\left(G_{p,q}\left(z^{\left(q\right)}\right)c_{p,q-1,q}^{-1}\left(t_{{z^{\left(q\right)}}}\right)\ldots c_{p,p-1,p}^{-1}\left(t_{{z^{\left(p\right)}}}\right)G_{p,p+1}\left(y^{\left(p+1\right)}z^{\left(p+1\right)}\right)y^{\left(p\right)}\right).
		\end{equation*}
	}
	Note that $h_{{z^{\left(q\right)}}}=t_{{y^{\left(q\right)}}}$. Thus $\mathcal{M}_{k,\ldots ,0}\left(y^{\left(k\right)}z^{\left(k\right)}\otimes \ldots \otimes y^{\left(0\right)}z^{\left(0\right)}\right)$ equals to
	{\small
		\begin{multline*} G_{0k}\left(y^{\left(k\right)}z^{\left(k\right)}\right)c_{0,k-1,k}^{-1}\left(t_{{z^{\left(k\right)}}}\right)G_{0,k-1}\left(y^{\left(k-1\right)}z^{\left(k-1\right)}\right)\ldots G_{0,q}\left(y^{\left(q\right)}\right)\cdot c_{0,p,q}^{-1}\left(t_{{y^{\left(q\right)}}}\right)\\
			\cdot G_{0,p}\left(\phi \cdot z^{\left(p\right)}\right)c_{0,p-1,p}^{-1}\left(t_{{z^{\left(p\right)}}}\right)\ldots c_{012}^{-1}\left(t_{{z^{\left(2\right)}}}\right)G_{01}\left(y^{\left(1\right)}z^{\left(1\right)}\right)y^{\left(0\right)}z^{\left(0\right)}
		\end{multline*}
	}
	where $\phi =G_{p,q}\left(z^{\left(q\right)}\right)c_{p,q-1,q}^{-1}\left(t_{{z^{\left(q\right)}}}\right)\ldots c_{p,p-1,p}^{-1}\left(t_{{z^{\left(p\right)}}}\right)G_{p,p+1}\left(y^{\left(p+1\right)}z^{\left(p+1\right)}\right)y^{\left(p\right)}$.
	This gives the desired expression.
\end{proof}	

\begin{remark}
	In particular, 
	\begin{align*}
		&\mathcal{M}_{k\ldots 0}\left(y^{\left(k\right)}z^{\left(k\right)}\otimes \ldots \otimes y^{\left(0\right)}z^{\left(0\right)}\right)\\
		=&\mathcal{M}_{k,p,\ldots ,0}\left(1\otimes \mathcal{M}_{k,\ldots ,p}\left(y^{\left(k\right)}z^{\left(k\right)}\otimes y^{\left(k-1\right)}z^{\left(k-1\right)}\otimes \ldots \otimes y^{\left(p\right)}\right)z^{\left(p\right)}\otimes \ldots \otimes y^{\left(0\right)}z^{\left(0\right)}\right).
	\end{align*}
	RHS reads as
	\begin{multline*}
		\scriptstyle c_{0,p,k}^{-1}\left(h_{{y^{\left(k\right)}}}\right)G_{0,p}\left(G_{p,k}\left(y^{\left(k\right)}z^{\left(k\right)}\right)c_{p,k-1,k}^{-1}\left(t_{{z^{\left(k\right)}}}\right)\ldots c_{p,p+1,p+2}^{-1}\left(t_{{z^{\left(p+2\right)}}}\right)G_{p,p+1}\left(y^{\left(p+1\right)}z^{\left(p+1\right)}\right)y^{\left(p\right)}\cdot z^{\left(p\right)}\right)\\
		\scriptstyle c_{0,p-1,p}^{-1}\left(t_{{z^{\left(p\right)}}}\right)\ldots c_{012}^{-1}\left(t_{{z^{\left(2\right)}}}\right)G_{01}\left(y^{\left(1\right)}z^{\left(1\right)}\right)y^{\left(0\right)}z^{\left(0\right)}.
	\end{multline*}
	In application, $y^{\left(k\right)}$ is taken as a coefficient of an input module element.  A linear combination of the product $\mathcal{M}_{k,\ldots ,0}\left(y^{\left(k\right)}z^{\left(k\right)}\otimes \ldots \otimes y^{\left(0\right)}z^{\left(0\right)}\right)$ for various coefficients gives an intertwining map from an $\cA_k$-module to an $\cA_0$-module.  The above equation tells us that it can be written as the cup product \eqref{eq:cup-gen} of intertwining maps from the $\cA_k$-module to a $\cA_p$-module and from the $\cA_p$-module to the $\cA_0$-module, where the maps are defined by $$G_{p,k}\left((-)z^{\left(k\right)}\right)c_{p,k-1,k}^{-1}\left(t_{{z^{\left(k\right)}}}\right)\ldots c_{p,p+1,p+2}^{-1}\left(t_{{z^{\left(p+2\right)}}}\right)G_{p,p+1}\left(y^{\left(p+1\right)}z^{\left(p+1\right)}\right)y^{\left(p\right)}$$ 
	and 
	$$G_{0,p}\left((-)\cdot z^{\left(p\right)}\right)
	c_{0,p-1,p}^{-1}\left(t_{{z^{\left(p\right)}}}\right)\ldots c_{012}^{-1}\left(t_{{z^{\left(2\right)}}}\right)G_{01}\left(y^{\left(1\right)}z^{\left(1\right)}\right)y^{\left(0\right)}z^{\left(0\right)}$$ respectively.  This will be important to establish $A_\infty$-equations over an algebroid stack.
\end{remark}

Similarly, we can define
{\small
	\begin{equation}
		\cM_{k,\ldots ,0}^{\mathrm{op}}\left(z^{\left(k\right)}\otimes \ldots \otimes z^{\left(0\right)} \right)\coloneqq z^{\left(0\right)}G_{01}\left(z^{\left(1\right)}\right)c_{012}\left(h_{{z^{\left(2\right)}}}\right)\ldots G_{0,k-1}\left(z^{\left(k-1\right)}\right)c_{0,k-1,k}\left(h_{{z^{\left(k\right)}}}\right)G_{0k}\left(z^{\left(k\right)}\right). \label{eq:mult-op}
	\end{equation}
}Similar to Proposition \ref{prop:comp}, it satisfies the following composition formula.  The proof will not be repeated.

\begin{prop} \label{prop:comp-rev} 
	$\cM_{k,\ldots ,0}^{\mathrm{op}}\left(y^{\left(k\right)}z^{\left(k\right)}\otimes \ldots \otimes y^{\left(0\right)}z^{\left(0\right)}\right)$ equals to
	\begin{multline*}
		\cM_{k,\ldots ,q,p,\ldots ,0}^{\mathrm{op}}\left(y^{\left(k\right)}z^{\left(k\right)}\otimes \ldots \otimes z^{\left(q\right)}\otimes \right.\\
		\left.y^{\left(p\right)}\cM^{\mathrm{op}}_{q,\ldots ,p}\left(y^{\left(q\right)}\otimes \ldots \otimes y^{\left(p+1\right)}z^{\left(p+1\right)} \otimes z^{\left(p\right)}\right) \otimes \ldots \otimes  y^{\left(0\right)}z^{\left(0\right)}\right).
	\end{multline*}
\end{prop}

Consider the case $k=1$.  Then 
$$\mathcal{M}_{1,0}\left(z^{\left(1\right)}\otimes z^{\left(0\right)}\right)= G_{01}\left(z^{\left(1\right)}\right)z^{\left(0\right)} \textrm{ and } 
\cM_{1,0}^{\mathrm{op}}\left(z^{\left(1\right)}\otimes z^{\left(0\right)} \right)= z^{\left(0\right)}G_{01}\left(z^{\left(1\right)}\right).$$
$\mathcal{M}_{1,0}\left((-)\cdot z^{\left(1\right)}\otimes z^{\left(0\right)}\right)$ can be used to define an intertwining map from $\cA_1$-modules to $\cA_0$-modules, but $\cM_{1,0}^{\mathrm{op}}\left((-)\cdot z^{\left(1\right)}\otimes z^{\left(0\right)} \right)$ cannot.  On the other hand, $\cM_{1,0}^{\mathrm{op}}$ preserves the left module structure of $\cA_0$ on $\cA_1 \otimes \cA_0$ (where the module structure is defined by inserting $a \in \cA_0$ in the middle of $z^{(1)}\otimes z^{(0)}$).  But $\mathcal{M}_{1,0}$ destroys this module structure.  $\cM_{k,\ldots ,0}^{\mathrm{op}}\left(z^{\left(k\right)}\otimes \ldots \otimes z^{\left(0\right)} \right)$ will be used in Section \ref{section:extended fukaya} for comparing two quiver algebras, while $\mathcal{M}_{k,\ldots ,0}\left(z^{\left(k\right)}\otimes \ldots \otimes z^{\left(0\right)}\right)$ will be used in Section \ref{section: mirror_algebroid} for gluing mirror algebroid stacks.  

\section{Representation theory of $A_\infty$ category by algebroid stacks}
\label{chapter:body}

In recent decades, the program of Strominger-Yau-Zaslow \cite{SYZ} has triggered a lot of groundbreaking developments in geometry. In particular, the family Floer theory, see the works of Fukaya \cite{Fukaya-famFl}, Tu \cite{Tu-reconstruction} and Abouzaid \cite{Ab-famFl}, applies homotopy techniques of Floer theory to Lagrangian torus fibers to construct a family Floer functor for mirror symmetry.

In \cite{CHL-nc}, the authors introduced a non-commutative mirror functor from the Fukaya category to the category of matrix factorizations of the corresponding Landau-Ginzburg model. Later, in \cite{CHL3}, they developed a method of gluing the local mirror functors.  

In this chapter, we will combine these two techniques.  Namely, we will develop a gluing method for local nc mirror charts.  We will use this to construct mirror algebroid stacks in later chapters.  Moreover, we define the mirror transform of an nc family of Lagrangians, see Remark \ref{rem:nc family}. In Theorem \ref{thm:nat-trans-X-A}, we show that there exists a natural transformation that relates the functors constructed from two different families of reference Lagrangians.

\subsection{Review on NC mirror functor}
\label{section:NCMF}

In this section, we firstly review some concepts about filtered $A_{\infty}$-algebra and bounding cochains in \cite{FOOO}.  Then we review the nc mirror functor construction in \cite{CHL-nc}. 

The Novikov ring is defined as $$\Lambda_0 = \left\{\sum_{i=1}^{\infty} a_i T^{\lambda_i}\mid a_i \in \mathbb{C}, \lambda_i \in \mathbb{R}_{\geq 0}, \lambda_i \textrm{ increases to } \infty \right\}$$
with maximal ideal
$$ \Lambda_+ = \left\{\sum_{i=1}^{\infty} a_i T^{\lambda_i}\mid a_i \in \mathbb{C}, \lambda_i \in \mathbb{R}_{>0},  \lambda_i \textrm{ increases to } \infty \right\}$$
and the universal Novikov field $\Lambda$ is defined as its field of fraction of $\Lambda_0$. The filtration on $\Lambda$ is given by $$ F^\lambda \Lambda = \left\{\sum_{i=1}^{\infty} a_i T^{\lambda_i} \in \Lambda | \lambda_i \geq \lambda \right\}.$$

\begin{defn}
	A filtered $A_\infty-$category $\cat{C}$ consists of a collection of objects $Ob(\cat{C})$, and torsion-free filtered graded $\Lambda_0$-module $\cat{C}(A_1,A_2)$ for each pair of objects $A_1,A_2\in Ob(\cat{C})$, equipped with a family of degree one operations $m_k:\cat{C}[1](A_0,A_1)\otimes \cdots \cat{C}[1](A_{k-1},A_k) \rightarrow \cat{C}[1](A_0,A_k)$ for all k and for $A_i\in Ob(\cat{C})$, $i = 0,1,\cdots,k$ , where $m_k$ is assumed to respect the filtration and satisfies the $A_\infty$-equations for $v_i\in \cat{C}[1](A_i,A_{i+1})$: 
	$$
	\sum_{k_1+k_2=n+1}\sum_{i=1}^{k_1}(-1)^{\epsilon_i} m_{k_1}(v_1,\cdots,m_{k_2}(v_i,\cdots,v_{i+k_2-1}),v_{i+k_2},\cdots,v_n) = 0
	$$ where $\epsilon_i = \sum_{j=1}^{i-1}(|v_j|')$, and $|v|'=|v|-1$, the shifted degree of $v$.
\end{defn}

\begin{remark}
	In this paper, we will denote the unshifted degree $d$ component of $\cat{C}(A_1,A_2)$ by $\cat{C}^d(A_1,A_2)$, and a Novikov term $T^A$ shows up to represent area of a polygon counted in $m_k$.
\end{remark}

When a filtered $A_\infty-$category consists of only a single object, it is called a filtered $A_\infty-$algebra. Let $A$ be an $A_\infty$ algebra. When $m_{\ge 3} =m_0= 0$, $A$ becomes a differential graded algebra, where $m_1$ and $m_2$ stand for differential and composition operation respectively according to $A_\infty-$equations.

With this understanding, we can also define unit in $\cat{C}^0(A,A)$, denoted by $1_A$, which has unshifted degree 0 and satisfies 
\begin{equation*}
	\begin{cases}
		m_2(1_A,v) = v  \quad\quad\quad\quad \text{\quad $v\in \cat{C}(A,A')$}\\
		(-1)^{|w|}m_2(w,1_A) = w  \text{\quad $w\in \cat{C}(A',A)$}\\
		m_k(\cdots,1_A,\cdots) = 0 \quad \quad \text{otherwise.}
	\end{cases}
\end{equation*}

\begin{defn}[\cite{FOOO}]
	An element in $b\in F^+\cat{C}^1(A,A)$ is a weak Maurer-Cartan element if $m_0^b := m(e^b) := \sum_{k=0}^\infty m_k(b,\cdots,b) = W(A,b)\cdot 1_A$ for some $W(A,b)\in \Lambda$.  
\end{defn}

Given $b\in F^+\cat{C}^1(A,A)$, we can define \begin{equation}
	m_k^{b}(v_1,\cdots, v_k) = m(e^{b},v_1,e^{b},v_2,\cdots, e^{b},v_k,e^{b}).
\end{equation}
In a similar fashion, one can also define $m_k$ for several $(L_i,b_i)$, and we shall not repeat.
The introduction of weak Maurer-Cartan elements gives a way to deform the $A_\infty$-algebra $\cat{C}(A,A)$ such that Floer cohomology is well-defined, even in the case that $m_0$ may not be zero. 

In this paper, we will use the  Fukaya category that also includes compact oriented spin immersed Lagrangians as objects.  Their Floer theory was defined in \cite{AJ}, generalizing the construction of \cite{FOOO} for smooth Lagrangians.

Let $X$ be a symplectic manifold, $\mathbb{L}\rightarrow X$ a compact spin oriented unobstructed Lagrangian immersion with transverse doubly self-intersection points. Recall that $\mathbb{L}$ is said to be unobstructed if $m_0^\bL = 0$. The space of Floer cochains is
$$
\CF^\bullet(\mathbb{L}):=\CF^\bullet(\mathbb{L},\mathbb{L}):= C^\bullet(\mathbb{L})\oplus \bigoplus_p \Span\{(p_{-},p_+),(p_+,p_{-})\}
$$ where $p$ are doubly self-intersection points and $p_-,p_+$ are its preimage. $(p_{-},p_+),(p_+,p_{-})$ are treated as Floer generators that jump from one connected component in the normalization to the other at the angles of a holomorphic polygon. 
For $C^\bullet(\bL)$, we shall use Morse model.  Namely, we take a Morse function on each component of (the domain of) $\bL$, and $C^\bullet(\bL)$ is defined as the formal $\Lambda$-span of the critical points. The Floer theory is defined by counting pseudo-holomorphic pearl trajectories \cite{Oh-Zhu,BC,FOOO-can,Sheridan-CY}. The chain model depends on the choice of Morse function and other auxiliary data such as almost complex structure and Kuranishi perturbations.

If the Lagrangian has trivial Maslov class, we can take the Morse grading as the grading for Floer theory. In general, due to the presence of discs with different Maslov indices, grading is only well-defined over $\Z_2$ and we take the Morse grading modulo two. 

By using homotopy method \cite{FOOO,CW15}, the algebra can be made to be unital.  See \cite[Section 2.2 and 2.3]{KLZ} for detail in the case of Morse model.  
The unit is denoted by $1_\mathbb{L}$.  It is homotopic to the formal sum of the maximum points of the Morse functions on all components (representing the fundamental class), denoted by $1^{\blacktriangledown}_\mathbb{L}$.  Namely, $1_\mathbb{L}-1_\mathbb{L}^{\blacktriangledown}=m_1(1_\mathbb{L}^h)$
(assuming $\mathbb{L}$ bounds no non-constant disc of Maslov index zero).

The space of Floer cochains $\CF^\bullet(L_1,L_2)$ for two Lagrangians (assuming they intersect cleanly) is similar and we shall not repeat.  In general, $\CF^\bullet(L_1,L_2)$ is only $\Z_2$-graded.  On the other hand, in Calabi-Yau situations where graded Lagrangians are taken, $\CF^\bullet(L_1,L_2)$ is $\Z$-graded, meaning that each Floer generator is assigned an integer degree, compatible with the $\Z_2$-grading, in such a way that the $A_\infty$-operations have the correct grading and satisfy $A_\infty$ equations.  Generators of degree one (which means odd degree when only $\Z_2$-grading exists) play a particularly important role in deformation theory.

\cite{CHL-nc} has made a construction of \emph{noncommutative deformation space of a spin oriented Lagrangian immersion} $\bL \subset M$.  The construction is summarized as follows.
\begin{construction} \label{constr:nc-single}
	
	\begin{enumerate}
		\item 	
		Associate a quiver $Q$ to $\CF^1(\bL)$. Namely, each component of (the domain of) $\bL$ is associated with a vertex, and each generator in $\CF^1(\bL)$ is associated with an arrow.
		\item 
		Extend the Fukaya algebra $A$ of $\bL$ over the path algebra $\Lambda Q$ and obtain a non-commutative $A_\infty-$algebra $$\tilde{A}^\mathbb{L} = \Lambda Q\otimes_{\Lambda^\oplus} \CF(\mathbb{L}),$$ 
		whose unit is $1_{\mathbb{L}} = \sum 1_{L_i}$.  $\Lambda^\oplus \subset \Lambda Q$ denotes $\bigoplus_i \Lambda \cdot e_i$ where $e_i$ are the trivial paths at vertices of $Q$.  The fibered tensor product means that an element $a \otimes X$ is non-zero only when tail of $a$ corresponds to the source of $X$.  The $A_\infty$ operations are defined by
		\begin{equation} \label{eq:mk}
			m_k (f_1 X_1,\ldots,f_k X_k) := f_k \ldots f_1 \, m_k (X_1,\ldots,X_k)
		\end{equation}
		where $X_l \in \CF(\bL)$  and $f_l \in \Lambda Q$.
		\item Extend the formalism of bounding cochains of \cite{FOOO} over $\Lambda Q$.  Namely, we take 
		\begin{equation} \label{eq:b}
			b = \sum_l b_l B_l
		\end{equation}
		where $B_l$ are the generators of $\CF^1(\bL)$, and $b_l$ are the corresponding arrows in $Q$.  Then define the deformed $A_\infty$ structure $m_k^{b}$ as in \cite{FOOO} and via Equation \eqref{eq:mk}.
		\item Quotient out the quiver algebra by the two-sided ideal $R$ generated by coefficients of the obstruction term $m_0^b$, so that $m_0^b = W\cdot 1_\bL$ over 
		$$\A := \Lambda Q/R.$$  $\A$ is called the noncommutative space of weakly unobstructed deformations of $\bL$.
		We call $(\A,W)$ to be a noncommutative localized mirror of $X$ probed by $\bL$.
		\item Extend the Fukaya category over $\A$, and enlarge the Fukaya category by including the objects $(\bL,b)$ where $b$ in \eqref{eq:b} is now defined over $\A$. We call $(\bL,b)$ a noncommutative family of Lagrangians parameterized by $\A$.  This means for $L_1,L_2$ in the original Fukaya category, the morphism space is now extended as $\A \otimes \CF(L_1,L_2)$.  The morphism spaces between $(\bL,b)$ and $L$ are enlarged to be $\CF((\bL,b),L) := \A\otimes_{\Lambda^\oplus} \CF(\mathbb{L},L)$ (and similarly for $\CF(L,(\bL,b))$).  We already have $\CF((\bL,b),(\bL,b))$ in Step 2 (except that $\Lambda Q$ is replaced by $\A$).  The $m_k$ operations are extended in a similar way to \eqref{eq:mk}.
	\end{enumerate}
\end{construction}

\begin{remark}\label{rem:nc family}
	$(\bL,b)$ is taken as a noncommutative family of objects over $\A$ as a whole; we have a family of Floer theories over $\A$. In general $\A$ is noncommutative. In such a case $b$ cannot be regarded as a point and one cannot make sense of $(\bL,b)$ for each individual value of $b$.
	
	When $Q$ is the quiver of one vertex with $n$ arrows and $R$ is the ideal of commutator relations $ab-ba$ for any two arrows $a,b$, $\A$ is simply the polynomial algebra $\Lambda[b_1,\ldots,b_n]$. In this commutative case we can talk about the individual $(\bL,b)$ parametrized by $b \in \Lambda^n$ and each of them is weakly unobstructed. 
\end{remark}

\begin{remark}
	$m_k^b$ in Step 3 is no longer linear over $\Lambda Q$.  For instance, suppose we have $m_1^b(X) = m_3(bB,X,bB) = b^2 \cdot \mathrm{out}$ where $\mathrm{out} = m_3(B,X,B)$.  Then
	$$ m_1^b(aX)=m_3(bB,aX,bB) = bab \cdot \mathrm{out} \not= a \cdot m_1^b(X). $$
	Boundary deformations are more non-trivial over noncommutative algebras in this sense.
	
	On the other hand, if we consider $m_k^{b,0,\ldots,0}$ on $\CF((\bL,b),L_1) \otimes \CF(L_1,L_2)\otimes\CF(L_{2},L_3)\otimes \ldots \\ \otimes \CF(L_{k-1},L_k)$ where none of $L_j$ is $(\bL,b)$, then $m_k^{b,0,\ldots,0}$ is still linear over $\A$.  This is important in defining the mirror functor.
\end{remark}

Using this, we obtain a canonical mirror transformation, which is analogous to the Yoneda functor, as follows.



\begin{defn} \label{def:mirMF}
	For an object $L$ of $\Fuk(X)$, its mirror matrix factorization of $(\A,W)$ is defined as 
	$$
	\mathscr{F^\mathbb{L}}(L) := \left( \A \otimes_{\Lambda^\oplus} \CF^\bullet(\mathbb{L},L), d = (-1)^{|\cdot|}m_1^{b,0}(\cdot)\right).
	$$
	The mirror of morphisms is given as follows:
	Given $L_1,L_2\in \Fuk(X)$ and an intersection point between them, $X\in \CF(L_1,L_2)$, $\mathscr{F}^\mathbb{L}(X): = (-1)^{(|X|-1)(|\cdot|-1)}m_2^{b,0,0}(\cdot,X): \mathscr{F^\mathbb{L}}(L_1)\rightarrow \mathscr{F^\mathbb{L}}(L_2)$.
\end{defn}

\begin{theorem}[\cite{CHL-nc}]
	The above definition of $\cF^\bL$ extends to give a well-defined $A_\infty$ functor
	$$ \Fuk(X) \to \MF(\A,W).$$ 
\end{theorem}

\begin{remark}
	Notice that $m_0^b=W\cdot 1_\bL$ has degree $2$.  Thus in the $\Z$-graded situation, $W=0$, and the above $\MF(\A,W)$ reduces to the dg category of complexes of $\A$-modules.
\end{remark}

We will often refer to $\A$ simply as the deformation space, or as the unobstructed deformation space. 

Intuitively, $\A$ can be understood via Strominger-Yau-Zaslow Conjecture \cite{SYZ}, which predicts that the mirror space is constructed as the moduli space of (special) Lagrangians. Roughly, a Lagrangian $\bL$ corresponds to a point of the mirror, while its deformation space $\A$ forms a neighborhood of that point. Thus, $\A$ is also refered as the localized mirror.

\begin{example}
	When $X$ is a symplectic surface, any compact oriented immersed curve (together with a weak bounding cochain) is an object inside $\Fuk(X)$.  The generators $(p_-,p_+)$ and $(p_+,p_-)$ can be visualized as angles at self-intersection points $p$, see Figure \ref{fig:angles}.  The parity of degrees of generators are determined by orientation as shown in the figure. 
	
	\begin{figure}[htb!]
		\centering
		\includegraphics[scale=0.6]{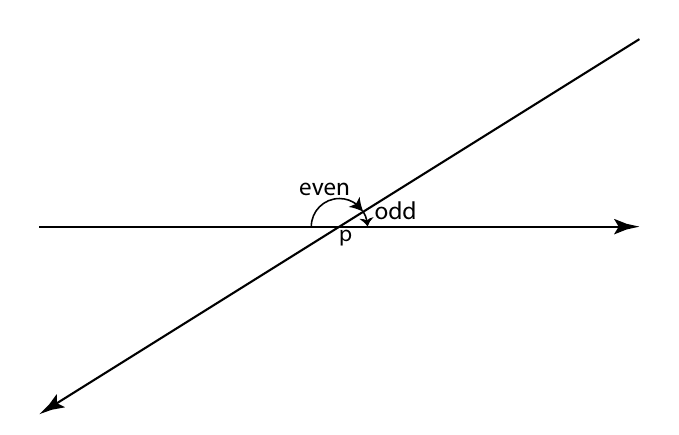}
		\caption{Each transverse intersection point corresponds to two Floer generators.}
		\label{fig:angles}
	\end{figure}
	
	For surfaces, we will use the following sign rule for a holomorphic polygon bounded by $\bL$ constructed by Seidel \cite{Seidel-book}.  The spin structure is given by fixing spin points (marking where the non-triviality of the spin bundle occurs) in (the domain of) $\bL$.  Denote the input angles of the polygon $P$ by $X_1,\ldots,X_k$, and the output angle by $X_0$.  If there is no spin point on the boundary of $P$ and the orientations of all edges of $P$ agree with that of $\bL$, then the contribution of $P$ (via output evaluation) takes a positive sign.
	Otherwise, disagreement of the orientations on $\widearc{X_i X_{i+1}}$, for $i=2,\ldots,k-1$, affects the sign by $(-1)^{|X_i|}$. Whether the orientation on $\widearc{X_1 X_2}$ agrees with $\bL$ or not is irrelevant. If the orientations are opposite on $\widearc{X_0 X_1}$, then we multiply by $(-1)^{|X_1| + |X_0|}$. Finally, we multiply by $(-1)^{l}$ where $l$ is the number of times $\partial P$ passes through the spin points. 
	
\end{example}

\begin{remark}
	In many important situations, $\A$ takes the form
	$$Jac(Q,\Phi) = \frac{\Lambda Q}{(\partial_{x_e}\Phi:e\in E )},$$ 
	where $\Phi$ is called spacetime superpotential.  The cases that we consider in this paper belong to this scenario.
\end{remark}
In \cite{Seidel-book,Seidel-g2,Seidel_HMS}, Seidel has made groundbreaking contributions to homological mirror symmetry. The Lagrangian immersion that he has invented plays a central role in the mirror symmetry part of this paper, whose deformation space is the building block of our mirror construction, namely nc $\C^3$.
\begin{example} \label{ex:Seidel}
	The \emph{immersed Lagrangian constructed by Seidel} \cite{Seidel-g2} is the most important source of motivation.  See Figure \ref{fig:Seidel_Lag}.  It is descended from a union of three circles in a three-punctured elliptic curve, as shown in Figure \ref{fig:ellcurve}.  The configuration in the elliptic curve is also interesting from a physics perspective \cite{BHLW,JL,GJLW}.
	
	\begin{figure}[htb!]
		\centering
		\begin{subfigure}[b]{0.15 \textwidth}
			\centering
			\includegraphics[width=\textwidth]{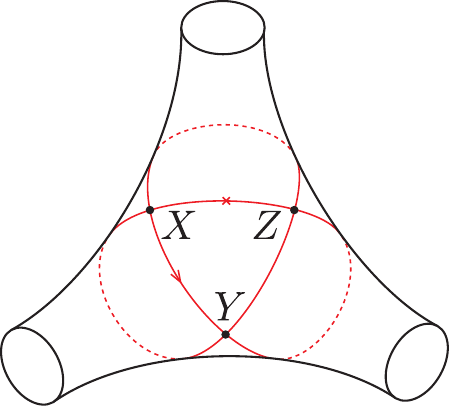}
			\caption{}
			\label{fig:Seidel_Lag}
		\end{subfigure}
		\hspace{40pt}
		\begin{subfigure}[b]{0.2 \textwidth}
			\centering
			\includegraphics[width=\textwidth]{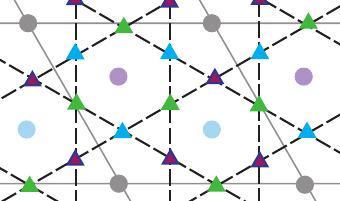}
			\caption{}
			\label{fig:ellcurve}
		\end{subfigure}
		\caption{The left hand side shows the Seidel Lagrangian in a pair-of-pants.  The right hand side shows a lifting to 3-to-1 cover by a three-punctured elliptic curve.}
	\end{figure}
	
	The Seidel Lagrangian has three degree-one immersed generators.  It gives the free algebra $\C \langle x,y,z\rangle$.  In the obstruction term $m_0^b$ of Floer theory, where $b=xX+yY+zZ$ is a formal linear combination of the degree-one generators, the front and back triangles bounded by $\bL$ contribute $e^A xy - e^B yx$ at the generator $\bar{Z}$ (and similar for the other generators $\bar{X}$ and $\bar{Y}$), where $A$ and $B$ are the areas of the back and front triangles respectively.  We quotient out these relations coming from obstructions and obtain the nc $\C^3$ 
	\begin{equation}
		\C\langle x,y,z \rangle / (e^A xy - e^B yx, e^A yz - e^B zy, e^A zx - e^B xz).
		\label{eq:ncC3}
	\end{equation}  
	
	Note that when $A\not=B$, the equation $e^A xy - e^B yx$ has no commutative solution.  We are forced to consider deformations over a noncommutative algebra.
	
	In a similar reasoning, for the $3:1$ lifting in punctured elliptic curve in Figure \ref{fig:ellcurve}, $\bL$ produces the quiver algebra  in Example \ref{ex:nclocP2}. More interestingly, \cite{CHL-nc} constructed a family of Sklyanin algebras over an elliptic curve by taking symplectic compactification of  the punctured elliptic curve.	
\end{example}

\begin{remark}
	In the above example, we take the Seidel Lagrangian together with a specific $\Z$-grading.  Namely, the point class and fundamental class are assigned to be in degree $0$ and $3$, and the generators at the self-intersection points are assigned to be in degree $1$ and $2$, depending on the parity.  Such a grading indeed comes from the fact that the Seidel Lagrangian corresponds to an immersed three-sphere in the threefold $\{(u,v,x,y) \in \C^2 \times (\C^\times)^2: uv = 1+x+y\}$ via the coamoeba picture \cite{FHKV}.  This is mirror to the toric Calabi-Yau threefold $\C^3 - \{xyz=1\}$ \cite{CLL,AAK}. The pair-of-pants is identified as the mirror curve $\{1+x+y=0\} \subset (\C^\times)^2$. 
	
	Homological mirror symmetry between noncommutative deformations of an algebra and non-exact deformations of a symplectic manifold was found by Aldi-Zaslow \cite{AZ} for Abelian surfaces and Auroux-Katzarkov-Orlov \cite{AKO06, AKO08} for weighted projective spaces and del Pezzo surfaces.  Quiver algebras mirror to a symplectic manifold is systematically constructed in \cite{CHL-nc}, by extending the Maurer-Cartan deformations of \cite{FOOO,FOOO-T,FOOO-T2,FOOO-MS}.  In Section \ref{section: mirror_algebroid}, we glue local nc mirrors to an algebroid stack, by extending the gluing technique of \cite{CHL3} over quiver algebras. 
\end{remark}

\subsection{Fukaya category enlarged by two nc families of Lagrangians}
\label{section:extended fukaya}

In the last section, we have reviewed the weakly unobstructed nc deformation space of an immersed Lagrangian \cite{CHL-nc}.  In this section, we consider two immersed Lagrangians $\bL_1, \bL_2$ over their weakly unobstructed nc deformation spaces $\A_1$ and $\A_2$.  The construction is important for relating different mirrors of the same symplectic manifold, for instance, the situation of twin Lagrangian fibrations \cite{LY-FM,Leung-Li}.

There are two closely related constructions in this situation.  The first one is taking product.  Namely, we take $(\bL_1,b_1)$ as probes and transform $(\bL_2,b_2)$ to a left $\A_1$-module over $\A_2$, or in other words, an $(\A_1,\A_2)$-bimodule. 
In commutative analog, this gives a universal sheaf over the product of local moduli of $\bL_1$ and that of $\bL_2$, whose fiber is the Floer cohomology $\mathrm{HF}^\bullet ((\bL_1,b_1),(\bL_2,b_2))$.  We concern about this in the current section.

The second construction is that we want to glue up the nc deformation spaces of $\bL_1$ and $\bL_2$ by finding an nc family of isomorphisms between $(\bL_1,b_1)$ and $(\bL_2,b_2)$ over certain localizations $(\A_1)|_{12} \cong (\A_2)|_{12}$.  $(\bL_i,b_i)$ are treated as objects in the same family.  The construction is presented in the next section.

In Definition \ref{def:mirMF}, we transform a single object $L$ using $(\bL_1,b_1)$.  Now we transform an nc family of objects $(\bL_2,b_2)$.  Let's define
\begin{equation} \label{eq:U-single}
	\U := \mathscr{F}^{(\bL_1,b_1)}((\bL_2,b_2)) := \left(\A_1 \otimes_{(\Lambda^\oplus)_1} \CF^\bullet(\bL_1,\bL_2) \otimes_{(\Lambda^\oplus)_2} \A_2^\op, d = (-1)^{|\cdot|}m_1^{b_1,b_2}(\cdot)\right).
\end{equation}
For an algebra $\A$, recall that $\A^\op$ is the opposite algebra which is the same as $\A$ as a set (and the corresponding elements are denoted as $a^\op$), with multiplication $a^\op b^\op := (ba)^\op$. The concatenation is read from left to right with $h(a^\op)=h(a).$
$ \U$ is a (graded) $(\A_1,\A_2)$-bimodule, where the right $\A_2$-module structure on $\A_2^\op$ is by taking
$ a^\op \cdot b := (ab)^\op = b^\op a^\op$.  The tensor product over $(\Lambda^\oplus)_2$ and $(\Lambda^\oplus)_1$ means that an element $ a_1 X a_2^\op$ is non-zero only when the source of $X$ matches with that of $a_1$ and the target of $X$ matches with target of $a_2^\op$.  

Indeed, as a generalization of Step (5) to two algebras in Construction \ref{constr:nc-single}, we shall extend the whole Fukaya category over $$T(\A_1,\A_2):=\widehat{\bigoplus_{k\geq 0}} \bigoplus_{|I|=k} \A_{i_1}\otimes\ldots\otimes \A_{i_k}$$
where $I=(i_1,\ldots,i_k)$ runs over multi-indices with entries in $\{1,2\}$ with no repeated adjacent entries. We think of this as the function algebra over the product.

The hat notation above denotes a completion with respect to a chosen non-Archimedean norm on $\A_1$ and $\A_2$, which induces a norm on $\bigoplus_{k\geq 0} \bigoplus_{|I|=k} \A_{i_1}\otimes\ldots\otimes \A_{i_k}$ via product $\norm{a_1a_2} := \norm{a_1}\norm{a_2}$ for $a_i \in \A_i$.  An element in $T(\A_1,\A_2)$ is a convergent series with respect to the non-Archimedean norm, which means the $k$-th term of the series has norm converging to zero as $k\to\infty$. We refer to Section \ref{sec: narch} for more about valuations, norms and completion.  

\begin{defn}\label{def: extended Fuk}
	The Fukaya category bi-extended over $T(\A_1,\A_2)$ has the same objects as $\Fuk(M)$, and morphism spaces between any two objects $L,L'$ are defined as  $T(\A_1,\A_2) \otimes \CF(L,L') \otimes (T(\A_1,\A_2))^\op$.  The $m_k$-operations are defined by   
	\begin{align} \label{eq:mk-bimod}
		m_k (f_1 X_1 h_1^\op,\ldots,f_k X_k h_k^\op) :=& f_k \otimes \ldots \otimes f_1 \, m_k (X_1,\ldots,X_k) \, h_1^\op \otimes \ldots \otimes h_k^\op\\
		=& f_k \otimes \ldots \otimes f_1 \, m_k (X_1,\ldots,X_k) \, (h_k \otimes \ldots \otimes h_1)^\op. \nonumber
	\end{align}	
	
	
	The enlarged Fukaya category has two more objects $(\bL_1,b_1)$ and $(\bL_2,b_2)$.  The morphism spaces involving these objects are  $ (T(\A_1,\A_2) \otimes \A_i) \otimes_{(\Lambda^\oplus)_i} \CF^\bullet(\bL_i,\bL_j) \otimes_{(\Lambda^\oplus)_j} (T(\A_1,\A_2)\\ \otimes \A_j )^\op$ for $i,j=1,2$, and $T(\A_1,\A_2)\otimes \A_i \otimes_{(\Lambda^\oplus)_i} \CF^\bullet(\bL_i,L)$, $ \CF^\bullet(L,\bL_i) \otimes_{(\Lambda^\oplus)_i} (T(\A_1,\A_2) \otimes \A_i )^\op$.  The $m_k$ operations are extended like above.  $m_k^{b_0,\ldots,b_k}$ is defined in the usual way, where $b_i \in \A_i \otimes_{(\Lambda^\oplus)_i} \CF^\bullet(\bL_i,\bL_i) \otimes_{(\Lambda^\oplus)_i} \A_i^\op$ is in the form \eqref{eq:b} (with non-trivial coefficients placed on the left; the coefficients on the right being simply $1$).
	
\end{defn}

It is easy to show that the extended $m_k^{b_0,\ldots,b_k}$ satisfy $A_\infty$ equations.  For notation simplicity, we will focus on the $\Z$-graded situation where $W^{(\bL_1,b_1)}=W^{(\bL_2,b_2)}=0$.  In particular, by the $A_\infty$ equation for $d_\U := m_1^{b_1,b_2}$, $\U$ satisfies $d_\U^2 = 0.$ Note that the original Fukaya category $\Fuk(M)$ is fully faithful embedded into the enlarged one, because the composition of setting the deformation parameters to zero and the natural inclusion is identity.

Once we have extended and enlarged the Fukaya category, we can take further steps in (family) Yoneda embedding construction.  
We have two $A_\infty$-functors $$\mathscr{F}^{(\bL_1,b_1)}: \Fuk(M) \to \mathrm{dg}(\A_1-\mathrm{mod})$$ and $$\mathscr{F}^{(\bL_2,b_2)}: \Fuk(M) \to \mathrm{dg}(\A_2-\mathrm{mod}).$$

Moreover, we have the dg functor 
$$\cF^{\U}:=\Hom_{\A_1}(\U,-): \mathrm{dg}(\A_1-\mathrm{mod}) \to \mathrm{dg}(\A_2-\mathrm{mod})$$
where $\U$ is a complex of $(\A_1,\A_2)$-bimodules defined by \eqref{eq:U-single}.  It takes $\Hom_{\A_1}(\U,E)$ for each entry $E$ in a complex of $\A_1$-modules. We modify the signs as follows.  The differential $(d_{\cF^{\U}(E)}(\phi))$ is defined as $(-1)^{|\phi|}$ times the usual differential of $\phi$ as a homomorphism from $\U$ to $E$.
Given $C,D\in \mathrm{dg}(\A_1-\mathrm{mod})$, $f\in \Hom_{\A_1}(C,D)$ and $\phi \in \Hom_{\A_2}(\U,C)$, $$\cF^\U(f)(\phi)(\cdot) = (-1)^{|\cdot|'}f\circ \phi (\cdot).$$

We want to compare $\mathscr{F}^{(\bL_2,b_2)}$ and $\cF^{\U}\circ \mathscr{F}^{(\bL_1,b_1)}$. They are related by a natural transformation. Let's first recall the definition.  

Recall that given two $A_\infty$-categories $\cA$ and $\mathcal{B}$, the $A_\infty$-functors form an $A_\infty$-category $\mathcal{Q}:=\mathrm{Fun}(\cA,\mathcal{B}).$

\begin{defn}
	Given two $A_\infty$-functors $\cF_0$ and $\cF_1$. A pre-natural transformation $T$ of degree $g$ from $\cF_0$ to $\cF_1$ is an element $T\in \mathrm{Hom}_\mathcal{Q}^g(\cF_0,\cF_1)$ of the chain space of morphisms in $\mathcal{Q}$, which is a sequence $(T^0,T^1,\cdots)$ such that $T^d$ be a family of multilinear maps $$\Hom_\cA(X_0,X_1) \otimes \cdots \otimes \Hom_\cA(X_{d-1},X_d) \to \Hom_\mathcal{B}(\cF_0X_0,\cF_1X_1)[g-d],$$ for all $(X_0,\cdots,X_d).$
\end{defn}

The boundary operator is \begin{align*}
	m_{1,\mathcal{Q}}(T)^d(a_1,\ldots,a_d)&= \sum_{r,i}\sum_{s_1,\cdots,s_r}(-1)^{\dagger}m_{r,\mathcal{B}}\big(\cF_0^{s_1}(a_1,\ldots,a_{s_1}), \ldots, \cF_0^{s_i-1}(\ldots,a_{s_1+\cdots+s_{i-1}}), \\
	&T^{s_i}(a_{s_1+\cdots+s_{i-1}+1}, \ldots, a_{s_1+\cdots+s_{i}}),\cF_1^{s_i+1}(a_{s_1+\cdots+s_{i}+1}, \ldots),\ldots,\\ &\cF_1^{s_r}(a_{d-s_r+1},\ldots,a_d)\big)- \sum_{k,l}(-1)^{|a_1|+\ldots+|a_l|-l+|T|-1}T^{d-k+1}(a_1,\ldots,a_k,\\&m_{k,\mathcal{A}}(a_{l+1},\ldots,a_{k+l}),a_{k+l+1},\ldots,a_d).
\end{align*}
The first sum is over $1 \leq i \leq r$ and partitions $s_1+\cdots +s_r=d$, where $s_i$ may be zero; and $\dagger=(|T|-1)(|a_1|+\cdots +|a_{s_1+\cdots+s_{i-1}}|-s_1-\cdots -s_{i-1})$.

\begin{defn}
	A natural transformation $T$ is a pre-natural transformation such that it's a cocycle i.e. $m_{1,\mathcal{Q}}(T)=0.$
\end{defn}

For the computation in the following proof, we define the notation for simplicity: \begin{equation}
	\sum_1^r:= \sum_{i=1}^r |\phi_i|'.
	\label{eq:Sigma}
\end{equation}

\begin{theorem} \label{thm:nat-trans}
	There exists a natural $A_\infty$-transformation from $\cF_1 = 
	\mathscr{F}^{(\bL_2,b_2)}$ to $\cF_2 = 
	\A_2 \otimes (\cF^{\U}\circ \mathscr{F}^{(\bL_1,b_1)})$.
\end{theorem}
\begin{proof}
	First consider object level.  Given an object $L$ of $\Fuk(M)$, we have a morphism (of objects in $\mathrm{dg} (\A_2-\textrm{mod})$) from $
	\mathscr{F}^{(\bL_2,b_2)}(L) = 
	\A_2 \otimes_{(\Lambda^\oplus)_2} \CF(\bL_2,L)$ to $
	\A_2 \otimes \cF^{\U} \left(\mathscr{F}^{(\bL_1,b_1)}(L)\right) = 
	\Hom_{\A_1}(\U,\A_2 \otimes\A_1 \otimes_{(\Lambda^\oplus)_1} \CF(\bL_1,L))$ (which is a left $\A_2$-module by the right multiplication of $\A_2$ on $\U$), given by
	$$\cT_L(\phi) := (-1)^{|\phi|'\cdot|-|'}R\left(m_2^{b_1,b_2,0}(-,\phi)\right),$$ for each $\phi \in  \mathscr{F}^{(\bL_2,b_2)}(L)$. 
	On the RHS of the above expression, $m_2^{b_1,b_2,0}(-,\phi) \in \A_2 \otimes \A_1 \otimes_{(\Lambda^\oplus)_1} \CF(\bL_1,L) \otimes \A_2^\op$.  The operator 
	\begin{equation} \label{eq:R}
		R: \A_2 \otimes \A_1 \otimes_{(\Lambda^\oplus)_1} \CF(\bL_1,L) \otimes \A_2^\op \to \A_2 \otimes \A_1 \otimes_{(\Lambda^\oplus)_1} \CF(\bL_1,L)
	\end{equation}
	moves an element $a_2^\op \in \A_2^\op$ on the right to $a_2$ multiplying on the left.  More explicitly, let $p Q q^\op \in \U$ and $\phi = \phi_i X_i$ for $\phi_i \in \A_2$.  Then $m_2^{b_1,b_2,0}(p Q q^\op,\phi)$ takes the form
	$$ m_2^{b_1,b_2,0}(p Q q^\op,\phi) = \phi_i f_i(b_2)\otimes p g_i(b_1) \,\mathrm{out}_i\, q^\op$$
	where $\mathrm{out}_i$ stands for the output, $f_i$ and $g_i$ are certain Novikov series.  We get
	$$ R\left(m_2^{b_1,b_2,0}(p Q q^\op,\phi)\right) = q\phi_i f_i(b_2)\otimes p g_i(b_1) \,\mathrm{out}_i.$$
	Note that $\cT_L(\phi)$ is an element in $
	\A_2 \otimes \cF^{\U} \left(\mathscr{F}^{(\bL_1,b_1)}(L)\right) = 
	\Hom_{\A_1}(\U,\A_2 \otimes\A_1 \otimes_{(\Lambda^\oplus)_1} \CF(\bL_1,L))$, i.e. $\cT_L(\phi)$ is an $\A_1$-module morphism. Since for $k \in \A_1$, we have
	$$ R\left(m_2^{b_1,b_2,0}(kp Q q^\op,\phi)\right) = k\cdot R\left(m_2^{b_1,b_2,0}(p Q q^\op,\phi)\right).$$  Besides, this defines an $\A_2$-module morphism.
	Let $c \in \A_2$, we have
	$$ \cT_L(c\phi)(p Q q^\op) = R\left(m_2^{b_1,b_2,0}(p Q q^\op,c\phi)\right) = qc\phi_i f_i(b_2)\otimes p g_i(b_1) \,\mathrm{out}_i = R\left(m_2^{b_1,b_2,0}(p Q (qc)^\op,\phi)\right).$$ Recall that $c \cdot \cT_L(c\phi)(p Q q^\op)= \cT_L(c\phi)(p Q c^\op q^\op )$ defines an left $\A_2$-module structure for any $c \in \A_2$. Therefore, we have $$R\left(m_2^{b_1,b_2,0}(p Q (qc)^\op,\phi)\right)= \cT_L(\phi) (p Q (qc)^\op)= c \cdot \cT_L(\phi)(p Q q^\op).$$
	Thus $\cT_L(c\phi) = c \cdot \cT_L(\phi)$.
	
	For morphisms and higher morphisms, let $L_0,\ldots,L_k$ be objects of $\Fuk(M)$ and $\phi_1 \otimes \ldots \otimes \phi_k \in \CF(L_0,L_1) \otimes \ldots \otimes \CF(L_{k-1},L_k)$.  Then we have a corresponding morphism from $
	\A_2 \otimes_{(\Lambda^\oplus)_2} \CF(\bL_2,L_1)$ to $
	\Hom_{\A_1}(\U,\A_2 \otimes \A_1 \otimes_{(\Lambda^\oplus)_1} \CF(\bL_1,L_k))$ given by
	\begin{equation}
		\cT(\phi_1,\cdots,\phi_k)(\phi)(\cdot):= (-1)^{|\cdot|'+\sum_1^k}R\left(m_{k+2}^{b_1,b_2,0,...,0}(\cdot,\phi,\phi_1,\cdots,\phi_k)\right). \label{eq:nat-trans}
	\end{equation}
	(Recall that $\sum_1^r= \sum_{i=1}^r |\phi_i|'$ in \eqref{eq:Sigma}.)  For simplicity, let's denote 
	$$\bar{m}_{k+2}^{b_1,b_2,0,...,0} := R\circ m_{k+2}^{b_1,b_2,0,...,0}.$$ 
	
	We want to check the equations for the $A_\infty$-natural transformation $\cT$:
	\begin{align*}
		&\delta \circ \cT(\phi_1,\ldots,\phi_k)\\
		+&\sum_{r=0}^{k-1}  (-1)^{|\cT|'\sum_1^r}\cF_2(\phi_{r+1},\ldots,\phi_{k}) \circ \cT(\phi_1,\ldots,\phi_r)\\
		+&\sum_{r=0}^{k}  \cT(\phi_{r+1},\ldots,\phi_{k}) \circ \cF_1(\phi_1,\ldots,\phi_r)\\
		-& \sum_{r=0}^{k-1}\sum_{l=1}^{k-r} (-1)^{\sum_{1}^r } \cT(\phi_1,\ldots,\phi_{r},m_l(\phi_{r+1},\ldots,\phi_{r+l}),\phi_{r+l+1},\ldots,\phi_k)= 0.
	\end{align*}
	For the first term, $\cT(\phi_1,\ldots,\phi_k)(\phi) \in \Hom_{\A_1}(\U,\A_2 \otimes\A_1 \otimes_{(\Lambda^\oplus)_1} \CF(\bL_1,L))$, and $\delta$ is the differential on $\Hom_{\A_1}(\U,\A_2 \otimes\A_1 \otimes_{(\Lambda^\oplus)_1} \CF(\bL_1,L))$ defined by
	$$ (\delta \rho) := \rho\circ d^\U + (-1)^{|\rho|'}d_{\cF^{(\bL_1,b_1)}(L_k)} \circ \rho.$$
	Thus the first term gives
	\begin{align*}
		\delta (\cT(\phi_1,\ldots,\phi_k)(\phi))(\cdot) =&  (-1)^{|\phi|'+\sum_1^k} \left(\bar{m}_{k+2}^{b_1,b_2,0,\ldots,0}(m_1^{b_1,b_2}(\cdot),\phi,\phi_1,\ldots,\phi_k)\right.\\ 
		&\left.+ m_1^{b_1,0}(\bar{m}_{k+2}^{b_1,b_2,0,\ldots,0}(\cdot,\phi,\phi_1,\ldots,\phi_k))\right). 
	\end{align*}
	We compute the later terms as follows.  First, $\cT$ is in degree $0$, and so $|\cT|'=-1$.
	\begin{align*}
		&(-1)^{\sum_{1}^r} \cF_2(\phi_{r+1},\ldots,\phi_{k}) \circ \cT(\phi_1,\ldots,\phi_r)(\phi)(\cdot)\\
		= & -\cF_2(\phi_{r+1},\ldots,\phi_k)((-1)^{|\cdot|'}\bar{m}^{b_1,b_2,0,\ldots,0}_{r+2}(\cdot,\phi,\phi_1,\ldots,\phi_r))\\
		= &
		(-1)^{|\phi|'+|\cdot|'+\sum_1^k}\cF^\U (m_{k-r+1}^{b_1,0,\ldots,0}(\bar{m}_{r+2}^{b_1,b_2,0,\ldots,0}(\cdot,\phi,\phi_1,\ldots,\phi_r),\phi_{r+1},\ldots,\phi_k) )\\
		= &
		(-1)^{|\phi|'+\sum_1^k}m_{k-r+1}^{b_1,0,\ldots,0}(\bar{m}_{r+2}^{b_1,b_2,0,\ldots,0}(\cdot,\phi,\phi_1,\ldots,\phi_r),\phi_{r+1},\ldots,\phi_k);\\
		& \cT(\phi_{r+1},\ldots,\phi_{k}) \circ \cF_1(\phi_1,\ldots,\phi_r)(\phi)(\cdot)\\
		= & -(-1)^{\sum_1^r+|\phi|'}\cT(\phi_{r+1},\ldots,\phi_k)(m_{r+1}^{b_2,0,\ldots,0}(\phi,\phi_1,\ldots,\phi_r))(\cdot)
		\\
		=&
		(-1)^{\sum_1^k+|\phi|'+|\cdot|'}\bar{m}^{b_1,b_2,0,\ldots,0}_{k-r+2}(\cdot,m^{b_2,0,\ldots,0}_{r+1}(\phi,\phi_1,\ldots,\phi_r),\phi_{r+1},\ldots,\phi_k);\\
		&(-1)^{\sum_1^r} \cT(\phi_1,\phi_2,\ldots,\phi_r,m_l(\phi_{r+1},\cdots,\phi_{r+l}),\ldots,\phi_k)(\phi)(\cdot)\\
		=& -(-1)^{|\cdot|'+\sum_1^r+\sum_1^k}\bar{m}^{b_1,b_2,0,\ldots,0}_{k-l+3}(\cdot,\phi,\phi_1,\ldots,\phi_r,m_l(\phi_{r+1},\ldots,\phi_{r+l}),\phi_{r+l+1},\ldots,\phi_k).
	\end{align*}
	Thus, it reduces to
	\begin{align*}
		&(-1)^{|\phi|'+\sum_1^k} \bar{m}_{k+2}^{b_1,b_2,0,\ldots,0}(m_1^{b_1,b_2}(\cdot),\phi,\phi_{1},\ldots,\phi_k) \\
		&+\sum_{r=0}^{k} (-1)^{|\phi|'+\sum_1^k} m_{k-r+1}^{b_1,0,\ldots,0}(\bar{m}_{r+2}^{b_1,b_2,0,\ldots,0}(\cdot,\phi,\phi_1,\ldots,\phi_r),\phi_{r+1},\ldots,\phi_k) \\
		&+ \sum_{r=0}^k (-1)^{\sum_1^k+|\phi|'+|\cdot|'} \bar{m}^{b_1,b_2,0,\ldots,0}_{k-r+2}(\cdot,m^{b_2,0,\ldots,0}_{r+1}(\phi,\phi_1,\ldots,\phi_r),\phi_{r+1},\ldots,\phi_k) \\
		&+(-1)^{|\cdot|'+\sum_1^r+\sum_1^k} \sum_{r=0}^{k-1}\sum_{l=1}^{k-r}  \bar{m}^{b_1,b_2,0,\ldots,0}_{k-l+3}(\cdot,\phi,\phi_1,\ldots,\phi_r,m_l(\phi_{r+1},\ldots,\phi_{r+l}),\phi_{r+l+1},\ldots,\phi_k) \\
	\end{align*}
	which is the $A_\infty$ equation for $\bar{m}_p^{b_1,b_2,0,\ldots,0}$ in the lemma below, 	
	with the common factor $(-1)^{|\phi|'+\sum_1^k}$.  Thus, $\cT$ is a natural transformation.
\end{proof}

The operations of $m_k^{b_0,\ldots,b_i,0,\ldots,0}$ and $R$ are carefully designed such that the following $A_\infty$ equation is satisfied.

\begin{lemma} \label{lem:mbar}
	The operations $\bar{m}^{b_0,\ldots,b_i,0,\ldots,0}_j= R\circ m^{b_0,\ldots,b_i,0,\ldots,0}_j$ (where $R$ is given in Equation \eqref{eq:R}) satisfies the following $A_\infty$ equation for $$\left((\bL_{i_0},b_{i_0}),\ldots, (\bL_{i_l},b_{i_l}),L_{l+1},\ldots,L_k\right):$$
	\begin{multline} \label{eq:A-infty-R}
		\sum_{s=1}^{l}\sum_{r=1}^{s}(-1)^{\sum_{j=1}^{r-1}|v_j|'} \bar{m}^{b_{i_0},\ldots,b_{i_{r-1}},b_{i_s},\ldots,b_{i_l},0,\ldots,0}_{k-s+r}(v_1,\cdots,v_{r-1},m^{b_{i_{r-1}},\ldots,b_{i_s}}_{s-r+1}(v_r,\cdots,v_s),v_{s+1},\cdots,\\v_k) 
		+ \sum_{s=l+1}^{k}\sum_{r=1}^{s}(-1)^{\sum_{j=1}^{r-1}|v_j|'} \bar{m}^{b_{i_0},\ldots,b_{i_{r-1}},0,\ldots,0}_{k-s+r}(v_1,\cdots,v_{r-1},\bar{m}^{b_{i_{r-1}},\ldots,b_{i_l},0,\ldots,0}_{s-r+1}(v_r,\cdots,v_s),v_{s+1},\cdots,\\v_k)= 0.
	\end{multline}
\end{lemma}

\begin{proof}
	Let $v_j = y_j Q_j x_j^\op$ for $j=1,\ldots,l$ and $v_{l+1}=\phi X_{l+1}$, $v_j=X_j$ for $j=l+2,\ldots,k$, where $y_j \in \A_{i_{j-1}}$, $x_j^\op \in \A^\op_{i_j}$, $\phi \in \A_{i_l}$.  For $s \leq l$, $m^{b_{i_{r-1}},\ldots,b_{i_s}}_{s-r+1}(v_r,\cdots,v_s)$ takes the form
	$$ o(b_s) \otimes y_s o(b_{s-1})\otimes\ldots \otimes y_r o(b_{r-1}) \otimes m(\ldots,Q_r,\ldots,Q_s,\ldots) \otimes  (x_r\otimes \ldots \otimes x_s)^\op $$
	where $o(b_j)$ are certain Novikov series in $b_j$.  For $s>l$, $\bar{m}^{b_{i_{r-1}},\ldots,b_{i_l},0,\ldots,0}_{s-r+1}(v_r,\cdots,v_s)$ takes the form
	$$ (x_r\otimes \ldots \otimes x_l) \phi o(b_l) \otimes y_l o(b_{l-1}) \otimes \ldots \otimes y_r o(b_{r-1}) \otimes m(\ldots,Q_r,\ldots,Q_l,\ldots,X_{l+1},\ldots).$$
	We can check that all the terms in \eqref{eq:A-infty-R} have the general form
	\begin{multline*}
		(x_1\otimes \ldots \otimes x_l) \phi o(b_l) \otimes y_l o(b_{l-1}) \otimes \ldots \otimes y_1 o(b_{0}) \otimes m(\ldots,Q_1,\ldots,\\Q_{r-1},\ldots,m(\ldots,Q_r,\ldots,Q_s,\ldots),\ldots,Q_{s+1},\ldots). 
	\end{multline*} 
	Thus all terms have the same coefficient $(x_1\otimes \ldots \otimes x_l) \phi o(b_l) \otimes y_l o(b_{l-1}) \otimes \ldots \otimes y_1 o(b_{0})$ and the result follows from the usual $A_\infty$ equation without this common coefficient.
\end{proof}

Now we have an $A_\infty$-transformation from $\mathscr{F}^{(\bL_2,b_2)}$ to $\A_2 \otimes (\cF^{\U}\circ \mathscr{F}^{(\bL_1,b_1)})$.  If we fix a representation $G_{12}$ of $\A_2$ over $\A_1$, then the $A_\infty$-transformation can be made to $(\cF^{\U}\circ \mathscr{F}^{(\bL_1,b_1)})$.  Namely, we take the multiplication $\cM^\op_{21}(x^{(2)}\otimes x^{(1)}) = x^{(1)} G_{12}(x^{(2)})$, and take the composition
$$ \cM^\op_{21} \circ R\circ m_{k+2}^{b_1,b_2,0,...,0} $$
in place of $R\circ m_{k+2}^{b_1,b_2,0,...,0}$ in the definition of natural transformation \eqref{eq:nat-trans}.  For instance, in the notation in the proof of Theorem \ref{thm:nat-trans}, 
$$ R\left(m_2^{b_1,b_2,0}(p Q q^\op,\phi)\right) = q\phi_i f_i(b_2)\otimes p g_i(b_1) \,\mathrm{out}_i.$$
Then
$$\cM^\op_{21}\left(R\left(m_2^{b_1,b_2,0}(p Q q^\op,\phi)\right)\right) = p g_i(b_1) G_{12}(q\phi_i f_i(b_2))\,\mathrm{out}_i.$$
The scaling by $c\in \A_1$ left on $p$ or $c \in \A_2$ left on $\phi$ (or right on $q^\op$) enjoys the same nice properties as in the proof of Theorem \ref{thm:nat-trans}.  (If we used $\cM_{21}$ instead, then it would be no longer $\A_1$-linear on $p$.) The $A_\infty$ equation for $(\bL_1,\bL_2,L_1,\ldots,L_k)$ continues to hold.  In this way, we get an $A_\infty$ natural transformation from $\mathscr{F}^{(\bL_2,b_2)}$ to $\cF^{\U}\circ \mathscr{F}^{(\bL_1,b_1)}$.

Similarly, in the reverse direction, if we fix a representation $G_{21}$ of $\A_1$ over $\A_2$, then we have a natural $A_\infty$-transformation from $\mathscr{F}^{(\bL_1,b_1)}$ to $\cF^{\U^*}\circ \mathscr{F}^{(\bL_2,b_2)}$, where $\U^* = \mathscr{F}^{(\bL_2,b_2)}$ $((\bL_1,b_1))$.  Then we can compose the natural transformations
$$\cF^{(\bL_2,b_2)} \to \cF^\U \circ \mathscr{F}^{(\bL_1,b_1)} \to \cF^\U \circ \cF^{\U^*}\circ \mathscr{F}^{(\bL_2,b_2)}$$ 
of functors from $\Fuk(M)$ to $ \mathrm{dg}(\A_2-\mathrm{mod})$. 

Given $\alpha \in \U$ and $\beta \in \U^*$, we have the evaluation natural transformation $\ev_{(\alpha,\beta)}:\cF^\U \circ \cF^{\U^*}\circ \mathscr{F}^{(\bL_2,b_2)} \to \mathscr{F}^{(\bL_2,b_2)}$.  By composing all of these, we get a self natural transformation on $\mathscr{F}^{(\bL_2,b_2)}$.

To go further, we consider a part of the setup in Section \ref{section:modified algebroid}.  Namely, suppose the representations $G_{12}$ and $G_{21}$ satisfy
\begin{equation}
	G_{12} \circ G_{21}(a)=c_{121}(h_a) \cdot a \cdot c^{-1}_{121}(t_a) \textrm{and } G_{21} \circ G_{12}(a)=c_{212}(h_a) \cdot a \cdot  c^{-1}_{212}(t_a) \label{eq:121}
\end{equation}
where $c_{121}(v)\in \left(e_{G_{12}(G_{21}(v))}\cdot\mathbb{A}_{1}\cdot e_{v}\right)^{\times}$ and $c_{212}(v')\in \left(e_{G_{21}(G_{12}(v'))}\cdot\mathbb{A}_{2}\cdot e_{v'}\right)^{\times}$ for every $v \in Q_0^{(1)}$ and $v' \in Q_0^{(2)}$.  Recall that we have defined the multiplication $\cM_{i_k,\ldots ,i_0}^{\op}: \mathbb{A}_{i_k} \otimes \ldots \otimes \mathbb{A}_{i_0} \to \mathbb{A}_{i_0}$ using $G_{12}$ and $G_{21}$ by  \eqref{eq:mult}.  Then define
$$\hat{m}^{b_0,\ldots,b_j}_j = \cM^\op \circ m^{b_0,\ldots,b_j}_j \textrm{ and } \overline{\hat{m}}^{b_0,\ldots,b_i,0,\ldots,0}_j= \cM^\op \circ R \circ m^{b_0,\ldots,b_i,0,\ldots,0}_j. $$ 
Explicitly, they take the form
\begin{multline*}
	\hat{m}^{b_0,\ldots,b_j}_j(p_1Q_1q_1^\op,\ldots,p_jQ_jq_j^\op) \\ = \scriptstyle \cM^\op_{i_j,\ldots,i_1}\left(f_j(b_j)\otimes p_j f_{j-1}(b_{j-1})\otimes \ldots \otimes p_1f_{0}(b_{0})\right)  m(\ldots,Q_1,\ldots,Q_j,\ldots) \left(\cM^\op_{i_1,\ldots,i_j}(q_1 \otimes \ldots \otimes q_j)\right)^\op
\end{multline*}
and
\begin{multline*}
	\overline{\hat{m}}^{b_0,\ldots,b_i,0,\ldots,0}_j(p_1Q_1q_1^\op,\ldots,p_iQ_iq_i^\op,p_{i+1}Q_{i+1},Q_{i+2},\ldots,Q_j) \\= \cM^\op_{i_j,\ldots,i_1}\left(q_1 \otimes \ldots \otimes q_i p_{i+1}f_i(b_i)\otimes p_i f_{i-1}(b_{i-1})\otimes \ldots \otimes p_1f_{0}(b_{0})\right)  m(\ldots,Q_1,\ldots,Q_j,\ldots).
\end{multline*} 

Here $f_i(b_i)$ is a linear combination of paths in $\A_i$ for $i=0, \cdots,j$.

\begin{theorem} \label{thm:A-infty-hat}
	The operations $\hat{m}^{b_0,\ldots,b_j}_j$ and $\overline{\hat{m}}^{b_0,\ldots,b_i,0,\ldots,0}_j$ satisfies the following $A_\infty$ equation for $$\left((\bL_{i_0},b_{i_0}),\ldots, (\bL_{i_l},b_{i_l}),L_{l+1},\ldots,L_k\right):$$
	\begin{multline} \label{eq:A-infty-hatbar}
		\scriptstyle \sum_{s=1}^{l}\sum_{r=1}^{s}(-1)^{\sum_{j=1}^{r-1}|v_j|'} \overline{\hat{m}}^{b_0,\ldots,b_{i_{r-1}},b_{i_s},\ldots,b_{i_l},0,\ldots,0}_{k-s+r}(v_1,\cdots,v_{r-1},\hat{m}^{b_{i_{r-1}},\ldots,b_{i_s}}_{s-r+1}(v_r,\cdots,v_s),v_{s+1},\cdots,v_k) \\\scriptstyle
		+ \sum_{s=l+1}^{k}\sum_{r=1}^{s}(-1)^{\sum_{j=1}^{r-1}|v_j|'} \overline{\hat{m}}^{b_0,\ldots,b_{i_{r-1}},0,\ldots,0}_{k-s+r}(v_1,\cdots,v_{r-1},\overline{\hat{m}}^{b_{i_{r-1}},\ldots,b_{i_l},0,\ldots,0}_{s-r+1}(v_r,\cdots,v_s),v_{s+1},\cdots,v_k)= 0.
	\end{multline}
\end{theorem}

\begin{proof}
	As in the proof of Lemma \ref{lem:mbar}, 
	Let $v_j = y_j Q_j x_j^\op$ for $j=1,\ldots,l$ and $v_{l+1}=\phi X_{l+1}$, $v_j=X_j$ for $j=l+2,\ldots,k$, where $y_j \in \A_{i_{j-1}}$, $x_j^\op \in \A^\op_{i_j}$, $\phi \in \A_{i_l}$.
	The summands in the first term take the form
	\begin{multline*}
		\cM^\op\left(x_1\otimes \ldots \otimes x_{r-1} \otimes \cM^\op(x_r\otimes\ldots\otimes x_s) \otimes x_{s+1} \otimes \ldots \otimes x_l \right.\\
		\left.\cdot \phi o(b_l) \otimes y_l o(b_{l-1}) \otimes \ldots \otimes \cM^\op(y_so(b_{s-1})\otimes \ldots \otimes y_ro(b_{r-1})) \otimes \ldots \otimes y_1 o(b_{0})\right)\\ \otimes m(\ldots,Q_1,\ldots,Q_{r-1},\ldots,m(\ldots,Q_r,\ldots,Q_s,\ldots),\ldots,Q_{s+1},\ldots). 
	\end{multline*}	
	The summands in the second term take the form
	\begin{multline*}
		\cM^\op\left(x_1\otimes \ldots \otimes x_{r-1} \otimes \cM^\op(x_r\otimes\ldots\otimes x_l \phi o(b_l) \otimes y_l o(b_{l-1}) \otimes \ldots \otimes y_ro(b_{r-1}))\right.\\ 
		\left.\otimes y_{r-1}o(b_{r-2}) \otimes \ldots \otimes y_1 o(b_{0})\right) \otimes m(\ldots,Q_1,\ldots,Q_{r-1},\ldots,m(\ldots,Q_r,\ldots,Q_s,\ldots),\ldots,Q_{s+1},\ldots). 
	\end{multline*}	
	By Proposition \ref{prop:comp-rev}, in both cases, all the coefficients equal to
	$$\cM^\op\left(x_1\otimes \ldots \otimes x_l \phi o(b_l) \otimes y_l o(b_{l-1}) \otimes \ldots \otimes y_1 o(b_{0})\right).$$
	Then the result follows from the usual $A_\infty$ equation without this common coefficient.
\end{proof}

Now we go back to the self natural transformation on $\mathscr{F}^{(\bL_2,b_2)}$ by composing the natural transformations
$$\cF^{(\bL_2,b_2)} \to \cF^\U \circ \mathscr{F}^{(\bL_1,b_1)} \to \cF^\U \circ \cF^{\U^*}\circ \mathscr{F}^{(\bL_2,b_2)} \to \mathscr{F}^{(\bL_2,b_2)}$$ 
of functors from $\Fuk(M)$ to $ \mathrm{dg}(\A_2-\mathrm{mod})$.  The last one is by evaluation at $\alpha \in \U$ and $\beta \in \U^*$.  

\begin{theorem} \label{thm:one-side inverse}
	Suppose $\alpha \in \U$ and $\beta \in \U^*$ are of degree 0 satisfying $\hat{m}_1^{b_1,b_2}(\alpha)=0$, $\hat{m}_1^{b_2,b_1}(\beta)=0$, and $\hat{m}_2^{b_2,b_1,b_2}(\beta,\alpha)=1_{\bL_2}$.  Then the natural transformation $\cF^{(\bL_2,b_2)} \to \cF^\U \circ \mathscr{F}^{(\bL_1,b_1)}$ has a left inverse, i.e. $$\cF^{(\bL_2,b_2)} \to \cF^\U \circ \mathscr{F}^{(\bL_1,b_1)} \to \cF^\U \circ \cF^{\U^*}\circ \mathscr{F}^{(\bL_2,b_2)} \to \mathscr{F}^{(\bL_2,b_2)}$$ is homotopic to the identity natural transformation.
\end{theorem}

\begin{proof}
	Under the assumption, there's an isomorphism between $\A_1$ and $\A_2$. Thus, we have $T(\A_1,\A_2)\cong \A_i$, and natural transformations $\cT_{12}:\cF^{(\bL_2,b_2)} \to \cF^{\U}\circ \mathscr{F}^{(\bL_1,b_1)}$, $\cT_{21}: \cF^{(\bL_1,b_1)} \to \cF^{\U^*}\circ \mathscr{F}^{(\bL_2,b_2)}$.
	We want to show that the above composition $$\bar{\cT}:=ev_{\alpha,\beta} \circ \cF^\U(\cT_{21}) \circ \cT_{12},$$
	is homotopic to the identity natural transformation $\cI$ on $\cF^{(\bL_2,b_2)}$.
	
	First, in the object level, we need to show that $\bar{\cT}_L$ for a Lagrangian $L$, which is an endomorphism on $\cF^{(\bL_2,b_2)}(L) = \A_2 \otimes_{(\Lambda^\oplus)_2} \CF(\bL_2,L)$,  equals to the identity up to homotopy.  For $\phi \in \A_2 \otimes_{(\Lambda^\oplus)_2} \CF(\bL_2,L)$,
	\begin{align*}
		\bar{\cT}_L(\phi) =&  \overline{\hat{m}}_2^{b_2,b_1,0}(\beta,\overline{\hat{m}}_2^{b_1,b_2,0}(\alpha,\phi)) \\
		=&  \overline{\hat{m}}_2^{b_2,b_2,0}(\hat{m}_2^{b_2,b_1,b_2}(\beta,\alpha),\phi)+ \overline{\hat{m}}_3^{b_2,b_1,b_2,0}(\beta,\alpha,m_1^{b_2,0}(\phi)) + m_1^{b_2,0}(\overline{\hat{m}}_3^{b_2,b_1,b_2,0}(\beta,\alpha,\phi))\\
		=& \overline{\hat{m}}_2^{b_2,b_2,0}(1_{\bL_2},\phi) + \cH_L\circ d_{\cF^{(\bL_2,b_2)}(L)} (\phi) + (-1)^{|\phi|'}d_{\cF^{(\bL_2,b_2)}(L)} \circ \cH_L (\phi)\\
		=& \phi + \cH_L\circ d_{\cF^{(\bL_2,b_2)}(L)} (\phi) +(-1)^{|\phi|'} d_{\cF^{(\bL_2,b_2)}(L)} \circ \cH_L (\phi).
	\end{align*}
	In the second line, we have used the $A_\infty$ equations by Theorem \ref{thm:A-infty-hat}, with the terms $\hat{m}_1^{b_1,b_2}(\alpha)$ and $\hat{m}_1^{b_2,b_2}(\beta)$ vanish.  We define 
	$$\cH_L := \overline{\hat{m}}_3^{b_2,b_1,b_2,0}(\beta,\alpha,-)$$ 
	as an endomorphism on $\cF^{(\bL_2,b_2)}(L)$, and it is extended as a self pre-natural transformation on $\cF^{(\bL_2,b_2)}$, by defining $\cH(\phi_1,\ldots,\phi_k):\cF^{(\bL_2,b_2)}(L_0) \to \cF^{(\bL_2,b_2)}(L_k)$ for $\phi_1 \otimes \ldots \otimes \phi_k \in \CF(L_0,L_1) \otimes \ldots \otimes \CF(L_{k-1},L_k)$ to be
	$$ \cH(\phi_1,\ldots,\phi_k):=(-1)^{\sum_1^k} \overline{\hat{m}}_{k+3}^{b_2,b_1,b_2,0,\ldots,0}(\beta,\alpha,-,\phi_1,\ldots,\phi_k).$$
	
	Then in the morphism level, for $\phi_1 \otimes \ldots \otimes \phi_k \in \CF(L_0,L_1) \otimes \ldots \otimes \CF(L_{k-1},L_k)$ ($k\geq 1$),
	{\small
		\begin{align*}
			&\bar{\cT}(\phi_1,\ldots,\phi_k)(\phi) = \sum_{r=0}^k (-1)^{\sum_1^k+|\phi|'} \overline{\hat{m}}_{k-r+2}^{b_2,b_1,0,\ldots,0}\left(\beta,\overline{\hat{m}}_{r+2}^{b_1,b_2,0,\ldots,0}(\alpha,\phi,\phi_1,\ldots,\phi_r),\phi_{r+1},\ldots,\phi_k\right) \\
			=& (-1)^{\sum_1^k+|\phi|'} \overline{\hat{m}}_{k+2}^{b_2,b_2,0,\ldots,0}\left(\hat{m}_2^{b_2,b_1,b_2}(\beta,\alpha),\phi,\phi_1,\ldots,\phi_k\right) \\
			&+\sum_{r=0}^{k} (-1)^{\sum_1^k+|\phi|'} \overline{\hat{m}}_{k-r+3}^{b_2,b_1,b_2,0,\ldots,0}\left(\beta,\alpha,m_{r+1}^{b_2,0,\ldots,0}(\phi,\phi_1,\ldots,\phi_r),\phi_{r+1},\ldots,\phi_k \right) \\ 
			&+ \sum_{r=0}^{k}(-1)^{\sum_1^k+|\phi|'} m_{k-r+1}^{b_2,0,\ldots,0}\left(\overline{\hat{m}}_{r+3}^{b_2,b_1,b_2,0,\ldots,0}(\beta,\alpha,\phi,\phi_1,\ldots,\phi_r),\phi_{r+1},\ldots,\phi_k\right)\\
			&+\sum_{r=0}^{k-1}\sum_{l=1}^{k-r} (-1)^{\sum_1^k+|\phi|'}(-1)^{|\phi|'+\sum_1^r} \overline{\hat{m}}^{b_2,b_1,b_2,0,\ldots,0}_{k-l+4} \left(\beta,\alpha,\phi,\phi_1,\ldots,\phi_r,m_{l}(\phi_{r+1},\ldots,\phi_{r+l}),\phi_{r+l+1},\ldots,\phi_k\right)\\
			=& \sum_{r=0}^{k} \left(\cH_L(\phi_{r+1},\ldots,\phi_k)\circ \cF^{(\bL_2,b_2)}(\phi_1,\ldots,\phi_r) (\phi) +(-1)^{\sum_1^r} \cF^{(\bL_2,b_2)}(\phi_{r+1},\ldots,\phi_k)\circ \cH_L(\phi_1,\ldots,\phi_r) (\phi) \right) \\
			&- \sum_{r=0}^{k-1}\sum_{l=1}^{k-r}(-1)^{\sum_1^r} \cH_L(\phi_1,\ldots,\phi_r,m_{l}(\phi_{r+1},\ldots,\phi_{r+l}),\phi_{r+l+1},\ldots,\phi_k)(\phi).
		\end{align*}
	}
	The second equation is the $A_\infty$ equation.  The first term $$\overline{\hat{m}}_{k+2}^{b_2,b_2,0,\ldots,0}\left(\hat{m}_2^{b_2,b_1,b_2}(\beta,\alpha),\phi,\phi_1,\ldots,\phi_k\right)$$ vanishes since $\hat{m}_2^{b_2,b_1,b_2}(\beta,\alpha)=1_{\bL_2}$. 
	
	The last expression above is exactly the differential of the pre-natural transformation $\cH_L$ evaluated on $\phi_1 \otimes \ldots \otimes \phi_k$.  This shows that $\bar{\cT} - \cI$ equals to the differential of $\cH_L$.
\end{proof}

In some ideal cases, $\mathscr{F}^{(\bL_2,b_2)}$ is naturally equivalent to $\cF^\U \circ \mathscr{F}^{(\bL_1,b_1)}$.

\begin{theorem} \label{thm: proj res}
	Assume that $\U$ has cohomology concentrated in the highest degree, that is, $\U$ is a projective resolution.  Then $\mathscr{F}^{(\bL_2,b_2)}(L)$ is quasi-isomorphic to $\cF^\U \circ \mathscr{F}^{(\bL_1,b_1)}(L)$ for each object $L$, and $\mathscr{F}^{(\bL_2,b_2)}(HF(L_0, L_1))$ is quasi-isomorphic to $\cF^\U \circ \mathscr{F}^{(\bL_1,b_1)}(HF(L_0, L_1))$ for all $L_0, L_1$. 
\end{theorem}

\begin{proof}
	Consider the following natural transformation 
	$$ \cF^\U \circ \mathscr{F}^{(\bL_1,b_1)} \to \cF^\U \circ \cF^{\U^*}\circ \mathscr{F}^{(\bL_2,b_2)} \to \mathscr{F}^{(\bL_2,b_2)} \to \cF^\U \circ \mathscr{F}^{(\bL_1,b_1)}$$ 
	Let $\bar{\cT}^{'}:=\cT_{12} \circ ev_{\alpha,\beta} \circ \cF^\U(\cT_{21}).$ The strategy is to show for each object $L$, $\bar{\cT}^{'}_L$, which is an endomorphism on $\cF^\U \circ \mathscr{F}^{(\bL_1,b_1)}(L)$, is a quasi-isomorphism. Combining with the previous theorem, we get the desired result.
	
	Let $(C^{\cdot}, d=(-1)^{|\cdot|}m_1^{b_1,0}(\cdot)):= \mathscr{F}^{(\bL_1,b_1)}(L)= \A_1 \otimes_{(\Lambda^\oplus)_1} \CF(\bL_1,L)$, $\U:= \mathscr{F}^{(\bL_1,b_1)}((\bL_2,b_2))= (A^{\cdot}, d = (-1)^{|\cdot|}m_1^{b_1,b_2}(\cdot))$ be the universal bundle with top degree $n$. Set $\U^{*}$ be its dual, i.e. $\U^{*}:=(A^{\cdot*}, d=(-1)^{|\cdot|}m_1^{b_2,b_1}(\cdot))$.  Then $\cF^\U \circ \mathscr{F}^{(\bL_1,b_1)}(L) = \U^{*} \otimes \A_1 \otimes_{(\Lambda^\oplus)_1} \CF(\bL_1,L)= A^{\cdot*} \otimes C^{\cdot}$ is a double complex with total complex $Tot (A^{\cdot*} \otimes C^{\cdot})$. Since this double complex is bounded, there exists a spectral sequence $E^{p,q}_r$ with $E^{p,q}_1= H^q(A^{\cdot*} \otimes C^{p})$ converges to the total cohomology $H^{p+q}(Tot (A^{\cdot*} \otimes C^{\cdot})).$ 
	
	Since $\U$ is a projective resolution, $E^{p,q}_1= H^q(A^{\cdot*} \otimes C^{p})$ for $q=n$, otherwise $0$. The spectral sequence becomes stable on the second page with $E^{p,q}_2= H^p H^q(A^{\cdot*} \otimes C^{\cdot}).$ In particular, $E^{p,q}_2= H^p H^q(A^{\cdot*} \otimes C^{\cdot})=0$ if $q \neq n$. Hence, $H^{m}(Tot (A^{\cdot*} \otimes C^{\cdot})) \cong E^{m-n,n}_{\infty}$, which is spanned by $A^{0*} \otimes H^{m-n}(C^{\cdot})$. Because $\bar{\cT}^{'}$ is a natural transformation, it suffices to show the cohomology class $[\bar{\cT}^{'}_L(A^{0*} \otimes \phi)]= [A^{0*} \otimes \phi]$ for $\phi \in  \A_1 \otimes_{(\Lambda^\oplus)_1} HF^p(\bL_1,L).$
	\begin{align*}
		\bar{\cT}^{'}_{L}(A^{0*} \otimes \phi)=& \cT_{12} \circ ev_{\alpha,\beta} (A^{0*} \otimes (\Sigma_{P \in \CF(\bL_2,\bL_1)} (-1)^{|P|+|\phi|} P \otimes \overline{\hat{m}}_2^{b_2,b_1,0}(P^{*}, \phi))) \\
		=& \cT_{12} \circ (a_0 \otimes \overline{\hat{m}}_2^{b_2,b_1,0}(\beta, \phi) ) \\
		=& \Sigma_{Q \in \CF(\bL_1,\bL_2)} (-1)^{|Q|+|\phi|} Q^{*} \otimes \overline{\hat{m}}_2^{b_1,b_2,0}(a_0 Q, \overline{\hat{m}}_2^{b_2,b_1,0}(\beta, \phi)), 
	\end{align*}
	where $a_0:= <A^0, \alpha>.$
	
	Note that the cohomology class of $\Sigma_{Q \in \CF(\bL_1,\bL_2)} (-1)^{|Q|+|\phi|} Q^{*} \otimes \overline{\hat{m}}_2^{b_1,b_2,0}(a_0 Q, \overline{\hat{m}}_2^{b_2,b_1,0}(\beta, \phi))$ equals to $[A^{0*} \otimes \overline{\hat{m}}_2^{b_1,b_2,0}(a_0 A^{0}, \overline{\hat{m}}_2^{b_2,b_1,0}(\beta, \phi))]= [A^{0*} \otimes \overline{\hat{m}}_2^{b_1,b_2,0}(\alpha, \overline{\hat{m}}_2^{b_2,b_1,0}(\beta, \phi))]$, by the above discussion.
	
	
	Furthermore, by the $A_\infty$ equations in Theorem \ref{thm:A-infty-hat},
	\begin{align*}
		&[A^{0*} \otimes \overline{\hat{m}}_2^{b_1,b_2,0}(\alpha, \overline{\hat{m}}_2^{b_2,b_1,0}(\beta, \phi))] \\
		=& [A^{0*} \otimes (\overline{\hat{m}}_2^{b_1,b_1,0}(\hat{m}_2^{b_1,b_2,b_1}(\alpha,\beta),\phi)+ \overline{\hat{m}}_3^{b_1,b_2,b_1,0}(\alpha,\beta,m_1^{b_1,0}(\phi)) + m_1^{b_1,0}(\overline{\hat{m}}_3^{b_1,b_2,b_1,0}(\alpha,\beta,\phi)))]\\
		=& [A^{0*} \otimes \overline{\hat{m}}_2^{b_1,b_1,0}(1_{\bL_1},\phi) + A^{0*} \otimes (\cH_L^{'}\circ d_{\cF^{(\bL_1,b_1)}(L)} (\phi) + (-1)^{|\phi|'}d_{\cF^{(\bL_1,b_1)}(L)} \circ \cH_L^{'} (\phi))]\\
		=& [A^{0*} \otimes \phi + A^{0*} \otimes (\cH_L^{'}\circ d_{\cF^{(\bL_1,b_1)}(L)} (\phi) +(-1)^{|\phi|'} d_{\cF^{(\bL_1,b_1)}(L)} \circ \cH_L^{'} (\phi))] \\
		=& [A^{0*} \otimes \phi].
	\end{align*}
	
	In the second line, we have used the $A_\infty$ equations by Theorem \ref{thm:A-infty-hat}, with the terms $\hat{m}_1^{b_1,b_2}(\alpha)$ and $\hat{m}_1^{b_2,b_1}(\beta)$ vanish. And we define 
	$$\cH_L^{'} := \overline{\hat{m}}_3^{b_1,b_2,b_1,0}(\alpha,\beta,-)$$ 
	as an endomorphism on $\cF^{(\bL_1,b_1)}(L).$ Note that $d_{\cF^{(\bL_1,b_1)}(L)} (\phi)=0$, since $\phi$ is closed. Hence, $\bar{\cT}^{'}_L: \cF^\U \circ \mathscr{F}^{(\bL_1,b_1)}(L) \xrightarrow{} \cF^\U \circ \mathscr{F}^{(\bL_1,b_1)}(L)$ is a quasi-isomorphism. With theorem \ref{thm:one-side inverse}, we know $\cT_{12, L}: \mathscr{F}^{(\bL_2,b_2)}(L) \to \cF^\U \circ \mathscr{F}^{(\bL_1,b_1)}(L)$ is a quasi-isomorphism.
	
	Therefore, in the derived $\mathrm{dg}(\A_2-\mathrm{mod})$ category, we have the following commutative diagram:
	$$\begin{tikzcd}
		\mathscr{F}^{(\bL_2,b_2)}(L_0) \arrow{d}{\mathscr{F}^{(\bL_2,b_2)}(\phi)} \arrow{r}{\cT_{12,L_0}}
		& \cF^\U \circ \mathscr{F}^{(\bL_1,b_1)}(L_0) \arrow{d}{\cF^\U \circ \mathscr{F}^{(\bL_1,b_1)}(\phi)} \\
		\mathscr{F}^{(\bL_2,b_2)}(L_1) \arrow{r} {\cT_{12, L_1}}
		& \cF^\U \circ \mathscr{F}^{(\bL_1,b_1)}(L_1)
	\end{tikzcd}$$
	for any objects $L_0, L_1$ in $\Fuk(M)$ and $\phi \in HF(L_0,L_1)$. Since $\cT_{12, L_0}$ and $\cT_{12, L_1}$ are isomorphisms, we get $\mathscr{F}^{(\bL_2,b_2)}(HF(L_0, L_1)) \cong \cF^\U \circ \mathscr{F}^{(\bL_1,b_1)}(HF(L_0, L_1))$ for all $L_0, L_1$.

\end{proof}

The condition in Theorem \ref{thm: proj res} is known to be held in some good cases, for example when $\bL_1$ and $\bL_2$ are the Lagrangian tori or pinched tori.

This also motivates the gluing construction via isomorphisms in the next section.  
In the next section, we will use the Fukaya isomorphisms to glue the nc deformation spaces of a collection of Lagrangian submanifolds, which form a quiver algebroid stack.

\subsection{Mirror algebroid stacks}
\label{section: mirror_algebroid}

In the last section, we enlarged the Fukaya category by two families of noncommutatively deformed Lagrangians, which naturally extend to $n$ families. This provides the foundation for the next section, where we glue the noncommutative deformation spaces and the localized mirror functors. Notably, for the purpose of gluing, we put all the coefficients on the left.

Let $\cL_1,\cdots,\cL_n$ be compact spin oriented immersed Lagrangians. Recall that we have the nc deformation spaces of $\cL_i$ as quiver algebras in Construction \ref{constr:nc-single} and we denote them by $\cA_i= \Lambda Q_i/ R_i,$ where $Q_i$ is the associated quiver of $\cL_i$, $R_i$ is the two-sided ideal generated by coefficients of the obstruction term $m_0^{(\cL_i,b)}$ and $b$ is the deformation variable. We have
$$T(\cA_1,\ldots,\cA_n):=\widehat{\bigoplus_{m\ge 0}} \bigoplus_{|I|=m} (\cA_{i_0}\otimes\cdots\otimes\cA_{i_m})
$$ 
which is understood as a product of the deformation spaces as in Section \ref{section:extended fukaya}.  The space of Floer chains and $A_\infty$ operations have been extended over $T(\cA_1,\ldots,\cA_n)$.  Namely, for two Lagrangians $L_0,L_1$ that are not any of these $\cL_i$'s, the morphism space is $T(\cA_1,\ldots,\cA_n)\otimes\CF(L_0,L_1)$.  The morphism spaces involving $(\cL_i,b_i)$ are extended as $(\cA_j\otimes T(\cA_1,\ldots,\cA_n) \otimes \cA_i) \\ \otimes_{(\Lambda^\oplus)_i \otimes (\Lambda^\oplus)_j} \CF^\bullet(\cL_i,\cL_j),$ $T(\cA_1,\ldots,\cA_n)\otimes \cA_i \otimes_{(\Lambda^\oplus)_i} \CF^\bullet(\cL_i,L)$, and $\cA_i \otimes T(\cA_1,\ldots,\cA_n) \otimes_{(\Lambda^\oplus)_i} \CF^\bullet(L,\cL_i)$.  All coefficients are pulled to the left according to \eqref{eq:mk}.  This is analogous to Definition \ref{def: extended Fuk}. 

In this section, we would like to construct mirror quiver algebroid stacks out of $(\cL_j,b_j)$ for $i=1,\ldots,n$.  Naively, for every $k\not= j$, we want to find $\alpha_{jk} \in (\cA_k\otimes \cA_j) \otimes_{(\Lambda^\oplus)_k\otimes (\Lambda^\oplus)_j} \CF^0(\cL_j,\cL_k)$ that satisfies
\begin{align}
	m_1^{b_j,b_k}(\alpha_{jk})=&0, \label{eq:closed}\\
	m_2^{b_j,b_k,b_l}(\alpha_{jk},\alpha_{kl})=&\alpha_{jl}, \label{eq:cocycle-alpha}\\
	m_p^{b_{i_0},\ldots,b_{i_p}}(\alpha_{i_0i_1},\ldots,\alpha_{i_{p-1}i_p})=& 0 \textrm{ for } p\geq 3.  \label{eq:higher0}
\end{align}
We set $\alpha_{jj} = 1_{\cL_j}$.
Indeed, we can make a version that allows homotopy terms in the second equation, namely, the two sides are allowed to differ by $m_1^{b_j,b_l}(\gamma_{jkl})$ for some $\gamma_{jkl} \in (\cA_l\otimes \cA_j) \otimes_{(\Lambda^\oplus)_l\otimes (\Lambda^\oplus)_j} \CF^{-1}(\cL_j,\cL_l)$.
(Similarly, we can also allow homotopy terms in the third equation.)
Such a system of equations of isomorphisms is a natural generalization of the equations $m_1^{b_j,b_k}(\alpha_{jk})=0$ and 
$m_2^{b_j,b_k,b_j}(\alpha_{jk},\alpha_{kj})=1_{\cL_j}$ raised and studied in \cite{CHL,HKL} in the two-chart case and before noncommutative extensions.

However, solving for $\alpha_{ij}$ inside $(\cA_j\otimes \cA_i) \otimes_{(\Lambda^\oplus)_j\otimes (\Lambda^\oplus)_i} \CF^0(\cL_i,\cL_j)$ is not the right thing to do.  $\cA_j\otimes \cA_i$ plays the role of a product.  On the other hand, we want to find gluing between the charts so that the isomorphism equations hold over the resulting manifold, rather than over the product of the charts.  To do so, we need to extend Fukaya category over an algebroid stack (in a modified version defined in Section \ref{section:modified algebroid}).

To begin with, let's motivate by the case of two charts.  Given a representation $G_{ji}$ of $\cA_i^{\mathrm{loc}}$ over $ \cA_j^{\mathrm{loc}}$ and representation $G_{ij}$ of $\cA_j^{\mathrm{loc}}$ over $\cA_i^{\mathrm{loc}}$ that satisfy \eqref{eq:121}, where $\cA_i^{\mathrm{loc}},\cA_j^{\mathrm{loc}}$ are certain localizations of $\cA_i, \cA_j$ respectively, we can define $m_1^{b_j,b_k}$ with target in $\cA_j^{\mathrm{loc}} \otimes_{(\Lambda^\oplus)_j\otimes (\Lambda^\oplus)_i} \CF^0(\cL_i,\cL_j)$ by using
\begin{equation} \label{eq:mult2}
	\cA_j^{\mathrm{loc}}\otimes \cA_i^{\mathrm{loc}} \to \cA_j^{\mathrm{loc}},\,\,
	a_j \otimes a_i = a_j \cdot G_{ji}(a_i).
\end{equation}
This is how we make sense of Equation \eqref{eq:closed}.  For higher $m_k$ operations, we need to use the multiplication defined by \eqref{eq:mult}.

Let's first state simple and helpful lemmas that follow directly from the definition of extended $m_k$-operations.

\begin{lemma} \label{lem:loop}
	Suppose $\phi \in (\cA_k \cdot e_{i_1}^{Q_k}\otimes e_{i_0}^{Q_j}\cA_j) \otimes_{(\Lambda^\oplus)_k\otimes (\Lambda^\oplus)_j} \CF^\bullet(\cL_j,\cL_k)$, where $e_{i_1}^{Q_k}$ and $e_{i_0}^{Q_j}$ are the trivial paths at the $i_1$-vertex in $Q_k$ and $i_0$-vertex in $Q_j$ respectively.  Then the coefficient of each output $P \in \CF^{\bullet}(\cL_j,\cL_k)$ in $m_1^{b_j,b_k}(\phi)$ belongs to $e_{h(P)} \cdot \cA_k \cdot e_{i_1}^{Q_k} \otimes e_{i_0}^{Q_j} \cdot \cA_j \cdot e_{t(P)}$.
	
	Similarly, let in addition that $\psi \in (\cA_l \cdot e_{i_3}^{Q_l}\otimes e_{i_2}^{Q_k}\cA_k) \otimes_{(\Lambda^\oplus)_l\otimes (\Lambda^\oplus)_k} \CF^\bullet(\cL_k,\cL_l)$.  Then the coefficient of each output $P \in \CF^{\bullet}(\cL_j,\cL_l)$ in $m_2^{b_j,b_k,b_l}(\phi,\psi)$ belongs to $e_{h(P)} \cdot \cA_l \cdot e_{i_3}^{Q_l} \otimes e_{i_2}^{Q_k} \cdot \cA_k \cdot e_{i_1}^{Q_k}\otimes e_{i_0}^{Q_j} \cdot \cA_j\cdot e_{t(P)}$.
\end{lemma}

\begin{lemma} \label{lem:vertex-match}
	The map \eqref{eq:mult2} restricted to $\cA_j^{\mathrm{loc}} e^{Q^{(j)}}_t\otimes e^{Q^{(i)}}_h\cA_i^{\mathrm{loc}}$ is non-zero only if $e^{Q^{(j)}}_t = G_{ji}\left(e^{Q^{(i)}}_h\right)$, where $t$ and $h$ are certain fixed vertices in $Q^{(j)}$ and $Q^{(i)}$ respectively.  In particular, if $Q^{(i)}$ consists of only one vertex, then $G_{ji}$ takes image in the loop algebra of $\cA_j^{\mathrm{loc}}$ at the vertex $t$.
\end{lemma}

Now consider the general case.   Suppose a quiver algebroid stack $\cX$ (in the version of Section \ref{section:modified algebroid}) is given, where the charts $\cA_i$ over $U_i$ are given by the nc deformation spaces of $\cL_i$ and their localizations.  We can simplify by fixing a base vertex $v^{(j)}$ for each $Q^{(j)}$ (although this is not a necessary procedure).  Then we take 
$$\alpha_{jk} \in \left(\cA_k^{\mathrm{loc}} e^{Q^{(k)}}_{v^{(k)}}\otimes e^{Q^{(j)}}_{v^{(j)}}\cA_j^{\mathrm{loc}}\right) \otimes_{(\Lambda^\oplus)_k\otimes (\Lambda^\oplus)_j} \CF^0(\cL_j,\cL_k)$$
and its corresponding image in $\cA_k^{\mathrm{loc}} \otimes_{(\Lambda^\oplus)_k\otimes (\Lambda^\oplus)_j} \CF^0(\cL_j,\cL_k)$ (which is also denoted by $\alpha_{jk}$ by abuse of notation).   ($\Lambda^\oplus)_j$ acts on $\cA_k^{\mathrm{loc}}$ via $G_{kj}$.)
By Lemma \ref{lem:vertex-match}, we should only consider quiver algebroid stacks whose transition maps satisfy $e^{Q^{(k)}}_{v^{(k)}} = G_{kj}\left(e^{Q^{(j)}}_{v^{(j)}}\right)$.  $\alpha_{kj} \in \left(\cA_j^{\mathrm{loc}} e^{Q^{(j)}}_{v^{(j)}}\otimes e^{Q^{(k)}}_{v^{(k)}}\cA_k^{\mathrm{loc}}\right) \otimes_{(\Lambda^\oplus)_j\otimes (\Lambda^\oplus)_k} \CF^0(\cL_k,\cL_j)$ induces an element in $\cA_j^{\mathrm{loc}} \otimes_{(\Lambda^\oplus)_j\otimes (\Lambda^\oplus)_k} \CF^0(\cL_k,\cL_j)$ which is again denoted by $\alpha_{kj}$.

\begin{defn} \label{def:CF_I}
	Define
	\begin{align*}
		\CF(L^{(p)}_0,L^{(q)}_1) :=& \cA_{p}(U_{pq}) \otimes \CF(L_0,L_1),\\
		\CF((\cL_j,b_j),L_1^{(p)}) :=& \cA_{j}(U_{jp}) \otimes_{(\Lambda^\oplus)_j} \CF(\cL_j,L_1),\\
		\CF(L_0^{(p)},(\cL_j,b_j)) :=& \cA_{p}(U_{pj}) \otimes_{(\Lambda^\oplus)_j} \CF(L_0,\cL_j),\\
		\CF((\cL_j,b_j),(\cL_k,b_k)) :=& \cA_{j}(U_{jk}) \otimes_{(\Lambda^\oplus)_k\otimes(\Lambda^\oplus)_j} \CF(\cL_j,\cL_k).	
	\end{align*}
	In above, $L_0,L_1$ denote Lagrangians that are not $(\cL_j,b_j)$ for any $j$.  They are decorated with an index $p$, meaning that they are treated over $\cA_p$.
	In the last line, $(\Lambda^\oplus)_k$ left multiplies on $\cA_{j}|_{U_{jk}}$ via  the representation $G_{jk}$ of $\cA_k|_{U_{jk}}$ by $\cA_{j}|_{U_{jk}}$.  (And similarly for the third line.)
	By restricting the sheaf of algebras over an open subset $U$, we have the notion of $\CF_{U}$ (where $U$ is a subset in the original domain, for instance $U_{pq}$ in the first line).
	
	By pulling all the coefficients to the left according to \eqref{eq:mk} and multiplying using \eqref{eq:mult}, we have the operations 
	$$ m_{k,\cX}^{b_0,\ldots,b_k}: \CF_{U_1}(K_0,K_1) \otimes \ldots \otimes \CF_{U_k}(K_{k-1},K_k) \to \CF_{\left(\bigcap_j U_j\right)} (K_0,K_k)$$
	where $K_l$ can be one of $(\cL_{j_{K_l}},b_{j_{K_l}})$ or other Lagrangians (in which case $b_l=0$ and $K_l$ is decorated with an index of a chart which is denoted as $\cA_l$). More precisely, let $f_jX_j \in \CF_{U_j}(K_{j-1},K_j)$ for $j=1,\ldots, k,$ then $$m_{k,\cX}^{b_0,\ldots,b_k}(f_1X_1, \ldots,f_kX_k)=\mathcal{M}_{k,\ldots ,0}(f_k \otimes \cdots \otimes f_1)m_k(X_1, \ldots, X_k).$$
\end{defn}

\begin{remark}
	Recall that $b_j$ varies in the nc deformation space $\cA_j$. Hence, $(\cL_j,b_j)$ forms a nc family of immersed Lagrangians over $\cA_j$ in the Fukaya category.
\end{remark}


\begin{theorem} \label{thm:A-infty-X}
	$\left\{m_{k,\cX}^{b_{i_0},\ldots,b_{i_k}}: k\geq 0\right\}$ satisfies the $A_\infty$ equations.
\end{theorem}
\begin{proof}
	Recall the $A_\infty$ equations for the original Fukaya category:
	$$
	\sum_{k_1+k_2=n+1}\sum_{l=1}^{k_1}(-1)^{\epsilon_l} m_{k_1}(X_1,\cdots,m_{k_2}(X_l,\cdots,X_{l+k_2-1}),X_{l+k_2},\cdots,X_n) = 0
	$$ 
	where $\epsilon_l = \sum_{j=1}^{l-1}(|X_j|')$.  Over $T(\cA_1(U_{1,\ldots,n}),\ldots,\cA_n(U_{1,\ldots,n}))$, we have
	{\small
		\begin{align*}
			&\sum_{k_1+k_2=n+1}\sum_{l=1}^{k_1}(-1)^{\epsilon_l} m_{k_1}^{b_0,\ldots,b_{l-1},b_{l+k_2-1},\ldots,b_n}\left(y_1\otimes x_0X_1,\cdots,y_{l-1}\otimes x_{l-2}X_{l-1},\right.\\
			&\left.m^{b_{l-1},\ldots,b_{l+k_2-1}}_{k_2}\left(y_{l}\otimes x_{l-1}X_l,\cdots,y_{l+k_2-1} \otimes x_{l+k_2-2}X_{l+k_2-1}\right),y_{l+k_2} \otimes x_{l+k_2-1}X_{l+k_2},\cdots,y_n \otimes x_{n-1}X_n\right) \\
			=& \sum_{p_0,\ldots,p_n} \beta_n^{p_n}y_n \otimes x_{n-1} \beta_{n-1}^{p_{n-1}} y_{n-1} \otimes \ldots \otimes x_1 \beta_1^{p_1} y_1 \otimes x_0 \beta_0^{p_0} \\
			& \sum_{k_2=0}^{n+1}\sum_{l=1}^{n+1-k_2}(-1)^{\epsilon_l} \sum m\left(B_0,\ldots,B_0,X_1,B_1,\ldots,B_1,X_2,\ldots,X_{l-1},B_{l-1},\ldots,B_{l-1},\right.\\
			& m(B_{l-1},\ldots,B_{l-1},X_l,\ldots,X_{l+k_2-1},B_{l+k_2-1},\ldots,B_{l+k_2-1}),
			B_{l+k_2-1},\ldots,B_{l+k_2-1},X_{l+k_2},\ldots,\\& X_n,B_n,\ldots,B_n),
	\end{align*}}
	which vanishes since the last two lines equal to zero.  Here, we write $b = \beta \cdot B$ in basis (understood as a linear combination) where $|B|'=0$.  The last summation above is over all the ways to split $p_{l-1}$ copies of $B_{l-1}$ into two sets, and $p_{l+k_2-1}$ copies of $B_{l+k_2-1}$ into two sets.
	
	Then for the last expression, we multiply the coefficient for each $(p_0,\ldots,p_n)$ using \eqref{eq:mult}, and we still have
	{\small
		\begin{align*}
			0=& \sum_{p_0,\ldots,p_n} \cM_{i_n\ldots i_0}(\beta_n^{p_n}y_n \otimes x_{n-1} \beta_{n-1}^{p_{n-1}}y_{n-1}\otimes \cdots \otimes x_1 \beta_1^{p_1} y_1 \otimes x_0 \beta_0^{p_0})\\	
			& \sum_{k_2=0}^{n+1}\sum_{l=1}^{n+1-k_2}(-1)^{\epsilon_l} \sum m\left(B_0,\ldots,B_0,X_1,B_1,\ldots,B_1,X_2,\ldots,X_{l-1},B_{l-1},\ldots,B_{l-1},\right.\\
			& m(B_{l-1},\ldots,B_{l-1},X_l,\ldots,X_{l+k_2-1},B_{l+k_2-1},\ldots,B_{l+k_2-1}),\\
			&\left. B_{l+k_2-1},\ldots,B_{l+k_2-1},X_{l+k_2},\ldots,X_n,B_n,\ldots,B_n\right).
		\end{align*}
	}
	By Proposition \ref{prop:comp}, the coefficients equal to
	{\small
		\begin{multline*} \mathcal{M}_{i_n,\ldots ,i_{l-1},i_{l+k_2-1},\ldots ,i_0}\left(\beta_n^{p_n}y_n \otimes \ldots \otimes x_{l+k_2-1} \beta_{l+k_2-1}^{r_1}\otimes \right.\\ \left.\mathcal{M}_{i_{l+k_2-1},\ldots ,i_{l-1}}\left(\beta_{l+k_2-1}^{r_2}y_{l+k_2-1} \otimes x_{l+k_2-2}\beta_{l+k_2-2}^{p_{l+k_2-2}}y_{l+k_2-2} \otimes \ldots \otimes  x_{l-1}\beta_{l-1}^{s_1}\right)\beta_{l-1}^{s_2} y_{l-1} \otimes \ldots \otimes x_0 \beta_0^{p_0}\right)
		\end{multline*}
	}
	where $r_1+r_2=p_{l+k_2-1}$ and $s_1+s_2=p_{l-1}$.  By putting back the coefficients into the $m_k$ operations, we obtain the $A_\infty$ equations for $m_{k,\cX}$.
\end{proof}

\begin{remark}
	We need to index the Lagrangians $L_i$ by charts, since the multiplication \eqref{eq:mult} needs this information.  $b_i = 0$ for $L_i$ not being any of $\cL_k$, but we still insert $e^{b_i} = 1_{L_i}$ in the coefficient.
	
	The following situation is particularly important for later use.  Consider the sequence of Lagrangians  $(\cL_{i_0},b_{i_0}), \ldots, (\cL_{i_k},b_{i_k}),L^{(i_k)}_0,\ldots,L^{(i_k)}_p$, for $i \leq p$.  One of the terms in the corresponding $A_\infty$ equation is
	\begin{small}
		$$m_{j+p-l+1,\cX}^{b_{i_0},\ldots,b_{i_j},0,\ldots,0}(\alpha_{i_0i_1},\ldots,\alpha_{i_{j-1}i_j},m_{k-j+l+1,\cX}^{b_{i_j},\ldots,b_{i_k},0,\ldots,0}(\alpha_{i_{j}i_{j+1}},\ldots,\alpha_{i_{k-1}i_{k}},\chi X, Q_1,\ldots,Q_l),Q_{l+1},\ldots,Q_p)$$
	\end{small}
	(where $\chi \in \cA_{i_k}$ is regarded as the input).
	Let $$m_{k-j+l+1,\cX}^{b_{i_j},\ldots,b_{i_k},0,\ldots,0}(\alpha_{i_{j}i_{j+1}},\ldots,\alpha_{i_{k-1}i_{k}},\chi X, Q_1,\ldots,Q_l) = \psi(\chi) \cdot \mathrm{out}'$$ 
	for $\psi(\chi) \in \cA_{i_j}$ with $h_{\psi(\chi)} = G_{i_ji_k}(h_\chi)$, and $$m_{j+p-l+1,\cX}^{b_{i_0},\ldots,b_{i_j},0,\ldots,0}(\alpha_{i_0i_1},\ldots,\alpha_{i_{j-1}i_j}, \mathrm{out}',Q_{l+1},\ldots,Q_p) = a_{i_0} \cdot \mathrm{out}$$
	for $a_{i_0} \in \cA_{i_0}$.  
	Then the above takes the form
	\begin{align*}
		\cM_{i_k,i_j,i_0}(e_{h(\chi)} \otimes \psi(\chi) \otimes a_{i_0}) \cdot  \mathrm{out}=& G_{i_0 i_k}(e_{h(\chi)}) c^{-1}_{i_0i_ji_k}({h(\chi)})G_{i_0i_j}(\psi(\chi))a_{i_0}\\
		=& c^{-1}_{i_0i_ji_k}({h(\chi)})G_{i_0i_j}(\psi(\chi))a_{i_0} = \phi \cup \psi (\chi)
	\end{align*}
	where $\phi(-):=G_{i_0i_j}(-)a_{i_0} = m_{j+p-l+1,\cX}^{b_{i_0},\ldots,b_{i_j},0,\ldots,0}(\alpha_{i_0i_1},\ldots,\alpha_{i_{j-1}i_j}, (-)\cdot\mathrm{out}',Q_{l+1},\ldots,Q_p)$, and $\cup$ is defined by \eqref{eq:cup-gen}.  This is the key ingredient in the proof of Theorem \ref{thm:A-infty-X} later.
	(Note that we cannot get this if we take $\cM_{i_j,i_0}(\psi(\chi) \otimes a_{i_0})$ instead of $\cM_{i_k,i_j,i_0}(e_{h(\chi)} \otimes \psi(\chi) \otimes a_{i_0})$.)
\end{remark}

Then Equation \eqref{eq:closed} and \eqref{eq:cocycle-alpha} are defined using $m_{1,\cX}^{b_j,b_k}$ and $m_{2,\cX}^{b_j,b_k,b_l}$.  We can also use $m_{k,\cX}^{b_{i_0},\ldots,b_{i_j},0,\ldots,0}$ to define an $A_\infty$ functor from the Fukaya category to the dg category of twisted complexes over the algebroid stack.

We summarize our noncommutative gluing construction as follows.

\begin{construction}
	\begin{enumerate}
		\item Fix a collection of spin oriented Lagrangian immersions $\cL_1,\\ \cL_2,\ldots, \cL_N$.
		\item Take their corresponding quivers $Q^{(j)}$ of degree one endomorphisms, and algebras of weakly unobstructed deformations $\cA_j = \Lambda Q^{(j)}/R^{(j)}$.
		\item Fix a topological space $B$ and an open cover with $N$ open sets.  Moreover, fix a sheaf of algebras over each open set $U_j$ which is given by localizations of $\cA_j$.
		For each $j=1,\ldots,N$, fix a vertex $v^{(j)} \in Q^{(j)}$.  Moreover, we fix $\alpha_{jk} \in \CF_{U_{jk}}^0((\cL_j,b_j),(\cL_k,b_k))$.
		\item Solve for gluing maps $G_{kj}:\cA_{j}|_{U_{jk}} \to \cA_{k}|_{U_{jk}}$ and gerbe terms $c_{jkl}(v)$ that define an algebroid stack $\cX$ over $B$, such that the collection of $\alpha_{jk}$ satisfies \eqref{eq:closed} and \eqref{eq:cocycle-alpha} using $m_{1,\cX}^{b_j,b_k}$ and $m_{2,\cX}^{b_j,b_k,b_l}$.
	\end{enumerate}
\end{construction}

\subsection{Gluing noncommutative mirror functors}\label{section:fuk_tw_functor}


In this section, we construct the $A_\infty$ functor 
$$\cF^\cL: \Fuk(M) \to \Tw(\cX)$$
in object and morphism level, using the $A_\infty$-operations $m_{k,\cX}^{b_{i_0},\ldots,b_{i_j},0,\ldots,0}$ defined in the last section.  The quiver algebroid stack $\cX$ is constructed by gluing the deformation spaces of a collection of Lagrangian immersions $\cL = \{\cL_1,\ldots,\cL_N\}$. 

First, let's consider the object level.  Given an object $L$ in $\Fuk(M)$, we define the corresponding twisted complex $\phi = \cF^\cL(L)$ on $\cX$ as follows.  Over each chart $U_i$, we take the complex $\left(\CF((\cL_i,b_i),L), \phi_i = (-1)^{|-|}m_{1,\cX}^{b_i,0}(-)\right)$.  Then the transition maps are defined by $\phi_{ij}(-) := m_{2,\cX}^{b_i,b_j,0}(\alpha_{ij},-): \CF_{ij}((\cL_j,b_j),L) \to \CF_{ij}((\cL_i,b_i),L)$.  Similarly, the higher maps $\phi_{i_0\ldots i_k}: \CF_{i_0\ldots i_k}((\cL_{i_k},b_{i_k}),L) \to \CF_{i_0\ldots i_k}((\cL_{i_0},b_{i_0}),L)$ for the twisted complex are defined by 
$$\phi_{i_0\ldots i_k}(-):=(-1)^{(k-1)|-|'} m_{k+1,\cX}^{b_{i_0},\ldots,b_{i_k},0}(\alpha_{i_0i_1},\ldots,\alpha_{i_{k-1}i_k},-).$$

\begin{lemma}
	$\phi$ above defines a twisted complex over $\cX$, namely, $\phi$ is intertwining and it satisfies the Maurer-Cartan equation \eqref{eq:MC}.
\end{lemma}
\begin{proof}
	Since the coefficient of the input for $\phi_{i_0\ldots i_k}$ will be pulled out to the leftmost, and by the definition of $\cM_{i_0\ldots i_k}$ \eqref{eq:mult}, $\phi_{i_0\ldots i_k}$ is intertwining.
	The Maurer-Cartan quation for $\phi$ follows from $A_\infty$-equations (Theorem \ref{thm:A-infty-X}) for $(\alpha_{i_0i_1},\ldots,\alpha_{i_{k-1}i_k},X)$.  Namely,
	\begin{multline*}
		-(-1)^{k|X|'}(-1)^{p-1}m_{k,\cX}^{b_{i_0},\ldots,\hat{b}_{i_{p}},\ldots,b_{i_k},0}(\alpha_{i_0i_1},\ldots, m_{2,\cX}^{b_{i_{p-1}},b_{i_p},b_{i_{p+1}}}(\alpha_{i_{p-1}i_{p}},\alpha_{i_pi_{p+1}}), \ldots,\alpha_{i_{k-1}i_k},X) \\= (-1)^{k|X|'}(-1)^p m_{k,\cX}^{b_{i_0},\ldots,\hat{b}_{i_{p}},\ldots,b_{i_k},0}(\alpha_{i_0i_1},\ldots, \alpha_{i_{p-1}i_{p+1}}, \ldots,\alpha_{i_{k-1}i_k},X) = (-1)^p\phi_{i_0\ldots \hat{i_p}\ldots i_k}
	\end{multline*}
	and
	\begin{align*}
		& -(-1)^{k|X|'}(-1)^{p}m_{p+1,\cX}^{b_{i_0},\ldots,b_{i_k},0}(\alpha_{i_0i_1},\ldots, \alpha_{i_{p-1}i_{p}}, m_{k-p+1,\cX}^{b_{i_p},\ldots,b_{i_k},0}(\alpha_{i_pi_{p+1}}, \ldots,\alpha_{i_{k-1}i_k},X))
		\\=&-(-1)^{k|X|'}(-1)^pm_{p+1,\cX}^{b_{i_0},\ldots,b_{i_k},0}(\alpha_{i_0i_1},\ldots, \alpha_{i_{p-1}i_{p}}, (-1)^{(k-p-1)|X|'}\phi_{i_0\ldots i_k}(X))
		\\=& -(-1)^{k|X|'}(-1)^p(-1)^{(k-p-1)|X|'}(-1)^{(p-1)(k-p+|X|'+1)}\phi_{i_0\ldots i_p} \cup \phi_{i_p\ldots i_k}(X)
		\\= &-(-1)^{k|X|'}(-1)(-1)^{k|X|'}(-1)^{(p-1)(k-p)}\phi_{i_0\ldots i_p} \cup \phi_{i_p\ldots i_k}(X)
		\\ = &(-1)^{(p-1)(k-p)}\phi_{i_0\ldots i_p} \cup \phi_{i_p\ldots i_k}(X) = \phi_{i_0\ldots i_p} \cdot \phi_{i_p\ldots i_k}(X).
	\end{align*}
	Moreover, $m_{k,\cX}^{b_{i_{1}},\ldots,b_{i_{k}}}(\alpha_{i_{1}i_{2}},\ldots,\alpha_{i_{k-1}i_{k}}) = 0$ for $k\not= 2$ by \eqref{eq:closed} and \eqref{eq:higher0}.  The RHS of the above equations add up to the Maurer-Cartan equation for $\phi$, while the LHS add up to zero by the $A_\infty$ equation.
\end{proof}

Next, let's consider the morphism level.  
For $L,L'$ in $\Fuk(M)$ and $Q\in CF^\bullet(L,L')$, we want to define a morphism $u = \cF^\cL(Q): \cF^\cL(L) \to \cF^\cL(L')$.  Over the charts $U_i$, we define
$u_i(-):= m_{2,\cX}^{b_i,0,0}(\cdot, Q).$  Over $U_{i_0\ldots i_k}$,
$ u_{i_0\ldots i_k}(-) := m_{k+2,\cX}^{b_{i_0},\ldots,b_{i_k},0,0}(\alpha_{i_0i_1},\ldots,\alpha_{i_{k-1}i_k},\cdot, Q).$  Similarly, given $L_0,\ldots,L_p$ and morphisms $Q_j \in \CF(L_{j-1},L_j)$, we define the higher morphism $u = \cF^\cL(Q_1,\ldots,Q_p)$ by
$$ u_{i_0\ldots i_k}(-) := (-1)^{k(|-|'+S_p)+|-|'} m_{k+p+1,\cX}^{b_{i_0},\ldots,b_{i_k},0,\ldots,0}(\alpha_{i_0i_1},\ldots,\alpha_{i_{k-1}i_k},-, Q_1,\ldots,Q_p),$$ where $S_p = \sum_{i=1}^p |Q_i|'.$

In the following computation, we denote $|X|'$ by $x$.

\begin{theorem} \label{thm: alg- functor}
	The above defines an $A_\infty$ functor $\cF^\cL:\Fuk(M) \to \Tw(\cX)$.
\end{theorem}
To prove our main theorem, let's first recall the definition and notation of $A_\infty$-functor. Let $\cC$ be an $A_\infty-$ categories. Take $A, B \in Ob(\cC)$. We put $$B_k\cC[1](A,B)= \bigoplus_{A=A_0, A_1,\cdots A_{k-1},A_k=B} \cC[1](A_0,A_1) \otimes \cdots \otimes \cC[1](A_{k-1},A_k).$$ 
$$B\cC[1](A,B)= \bigoplus_{k=1}^{\infty} B_k\cC[1](A,B), B\cC[1]= \bigoplus_{A,B} B\cC[1](A,B).$$
The $A_\infty$ operation $m_k$ induces coderivation $\hat{d}_k$ on $B\cC[1]$. The system of $A_\infty$ equations can be written as a single equation: $\hat{d} \circ \hat{d}=0.$

\begin{defn}
	Let $\cC_1$ and $\cC_2$ be two $A_\infty-$ categories. An $A_\infty-$functor $\cF: \cC_1 \to \cC_2$ is a collection $\cF_k, k \in \Z_{\geq 0}$ such that $\cF_0: Ob(\cC_1) \to Ob(\cC_2)$ is a map between objects, and for $A_1, A_2 \in Ob(\cC_1)$, $\cF_k(A_1,A_2): B_k\cC_1(A_1,A_2) \to \cC_2[1](\cF_0(A_1),\cF_0(A_2))$ is a homomorphism of degree 0. The induced coalgebra $\hat{\cF}_k:B\cC_1[1] \to B\cC_2[1]$ is required to be a chain map with respect to $\hat{d}$ where $\hat{\cF}_k(x_1 \otimes \cdots \otimes x_k)$ is given by $$\sum_m \sum_{0=l_1 < l_2 <\cdots <l_m=k} \cF_{l_2-l_1}(x_{l_1+1} \otimes \cdots \otimes x_{l_2}) \otimes \cdots \otimes \cF_{l_m-l_{m-1}}(x_{l_{m-1}+1}\otimes \cdots \otimes x_{l_m}).$$
\end{defn} 
\begin{proof}
	
	Consider the $A_\infty$ equation for $(\alpha_{i_0i_1},\ldots,\alpha_{i_{k-1}i_k},X, Q_1,\ldots,Q_p)$.  It consists of terms
	{\small
		\begin{align*}
			& (-1)^{k+x+S_{r-1}}m_{k+1-(s-r),\cX}^{b_{i_0},\ldots,b_{i_k},0,\ldots,0}(\alpha_{i_0i_1},\ldots,\alpha_{i_{k-1}i_k},X, Q_1,\ldots,Q_{r-1},m_{s-r+1}(Q_r,\ldots,Q_s),Q_{s+1},\ldots,Q_p) \\
			\\
			=&(-1)^{1+k(x+S_p)}(-(-1)^{S_{r-1}}) \cF_{i_0\ldots i_k}^\cL(Q_1,\ldots,Q_{r-1},m_{s-r+1}(Q_r,\ldots,Q_s),Q_{s+1},\ldots,Q_p)(X),
		\end{align*}
	Similarly,
		\begin{align*}
			&(-1)^jm_{j+p-l+1,\cX}^{b_{i_0},\ldots,b_{i_j},0,\ldots,0}(\alpha_{i_0i_1},\ldots,\alpha_{i_{j-1}i_j},m_{k-j+l+1,\cX}^{b_{i_j},\ldots,b_{i_k},0,\ldots,0}(\alpha_{i_{j}i_{j+1}},\ldots,\alpha_{i_{k-1}i_{k}},X, Q_1,\ldots,Q_l),Q_{l+1},\ldots,Q_p) \\
			\\
			= &(-1)^{(kx+(k+1+j)S_l+jS_p+jk+k+1)+(S_p - S_l + j + 1) (k - j)}\cF_{i_0\ldots i_j}^\cL(Q_{l+1},\ldots,Q_p)\cdot \cF_{i_{j}\ldots i_k}^\cL(Q_1,\ldots,Q_l) (X)
			\\
			\\=&(-1)^{kS_p+kx+1}m_2(\cF_{i_0\ldots i_j}^\cL(Q_{l+1},\ldots,Q_p),\cF_{i_{j}\ldots i_k}^\cL(Q_1,\ldots,Q_l))(X),
		\end{align*}
		\begin{align*}
			&\sum_{l=1}^{k-1} (-1)^{l-2}m_{k+p,\cX}^{b_{i_0},\ldots,\hat{b}_{i_l},\ldots,b_{i_k},0,\ldots,0}(\alpha_{i_0i_1},\ldots,m_{2,\cX}^{b_{i_{l-1}},b_{i_l},b_{i_{l+1}}}(\alpha_{i_{l-1}i_l},\alpha_{i_li_{l+1}}),\ldots,\alpha_{i_{k-1}i_k},X, Q_1,\ldots,Q_p) \nonumber\\
			+&\sum_{j=0}^{k} (-1)^jm_{j+p+1,\cX}^{b_{i_0},\ldots,b_{i_j},0,\ldots,0}(\alpha_{i_0i_1},\ldots,\alpha_{i_{j-1}i_j},m_{k-j+1,\cX}^{b_{i_j},\ldots,b_{i_k},0,\ldots,0}(\alpha_{i_{j}i_{j+1}},\ldots,\alpha_{i_{k-1}i_{k}},X), Q_1,\ldots,Q_p) \nonumber\\
			+& \sum_{j=0}^{k}(-1)^j m_{j+1,\cX}^{b_{i_0},\ldots,b_{i_j},0,\ldots,0}(\alpha_{i_0i_1},\ldots,\alpha_{i_{j-1}i_j},m_{k-j+1+p,\cX}^{b_{i_j},\ldots,b_{i_k},0,\ldots,0}(\alpha_{i_{j}i_{j+1}},\ldots,\alpha_{i_{k-1}i_{k}},X, Q_1,\ldots,Q_p)) 
			\nonumber\\
			=&(-1)^{kS_p+kx+1}(-1)^l \sum_{j=0}^{k-1}\cF_{i_0\ldots \hat{i}_l\ldots i_k}^\cL(Q_1,\ldots,Q_p)(X) \nonumber\\
			+& (-1)^{kS_p+kx+1}\sum_{j=0}^{k}\left(-(-1)^{S_p+1}\cF_{i_{0}\ldots i_j}^\cL(Q_1,\ldots,Q_p) \cdot \cF_{i_j\ldots i_k}^\cL(L)(X) + \cF_{i_{0}\ldots i_j}^\cL(L) \cdot \cF_{i_j\ldots i_k}^\cL(Q_1,\ldots,Q_p)(X)
			\right) \nonumber\\
			=&(-1)^{kS_p+kx+1} (d \cF^\cL(Q_1,\ldots,Q_p))_{i_0\ldots i_k}(X).
		\end{align*}
	}
	Moreover, $m_{k,\cX}^{b_{i_{1}},\ldots,b_{i_{k}}}(\alpha_{i_{1}i_{2}},\ldots,\alpha_{i_{k-1}i_{k}}) = 0$ for $k\not= 2$ by \eqref{eq:closed} and \eqref{eq:higher0}.  With the common factor $(-1)^{kS_p+kx+1}$, the right hand sides of the above equations add up to the equation for being an $A_\infty$ functor (keeping in mind that $\Tw(\cX)$ is a dg category with no higher multiplication), while the LHS add up to zero by the $A_\infty$ equation.
\end{proof}

The following proposition shows that our functor is injective on a certain class of Hom spaces related to the collections of reference Lagrangians $\cL:=\{\cL_k\}_{k\in I}$.

\begin{prop}\label{prop:inj}
	If the $A_\infty$-category is unital, then the mirror $A_\infty$ functor $\cF^\cL$ is injective on $\mathrm{HF}^\bullet((\cL',b_0),L)$ (and also on $\CF((\cL',b_0),L)$) for any Lagrangian $L$ and any constant elements $b_0$ in the deformation space of $\cL'$, where $\cL'$ is a subset of $\cL$. 
\end{prop} 

\begin{proof}
	Our strategy is writing down a right inverse $$\Psi: \Hom_{\cX}(\cF^{\cL}(\cL',b_0),\cF^\cL(L))\to \CF((\cL',b_0),L)$$ to the mirror functor $\cF^{\cL}$, which implies the injectivity. It suffices to consider $\cL'$ consists of a single Lagrangian immersion $\cL_k$ by definition.
	
	Recall that over the open subset $U_i$, $$\cF^\cL(\cL_k,b_0)=(\cA_i \otimes_{\Lambda_0}\CF^\bullet((\cL_i,b_i),(\cL_k,b_0)),m_{1}^{b_i,b_0}),$$ and on the overlap, we have the transition maps up to gerbe terms. 
	
	Let $\phi$ be a morphism in $\Hom_{\cX}(\cF^\cL(\cL_k,b_0),\cF^{\cL}(L)).$ We define $\Psi(\phi)$ as $$\Psi(\phi):=(\phi_k(\one_{\cL_k})\mid_{b_k=b_0}) \in \CF^\bullet ((\cL_k,b_0),L),$$ where $\phi_k$ is the morphism over $U_k$. In other words, it only makes use of the morphism over $U_k$ and set others to be zero.
	
	We first show $\Psi$ defines a chain map:
	\begin{align*}
		\Psi(d_\cX (\phi))&= \Psi(\check{\partial} \phi)+ \Psi(m_1^{b_k,0}\circ \phi)- (-1)^{\mid\phi\mid}\Psi(\phi \circ m_1^{b_k,b_0})\\
		&= \check{\partial} \phi_k(\one_{\cL_k})|_{b_k=b_0}+ m_1^{b_k,0} (\phi_k(\one_{\cL_k})|_{b_k=b_0})-(-1)^{\mid\phi\mid}(\phi(m_1^{b_k,b_0}(\one_{\cL_k}))\mid_{b_k=b_0}).
	\end{align*}
	Notice that $\check{\partial} \phi_k =0$ and $m_1^{b_k,b_0}(\one_{\cL_k})=b_k-b_0$. Hence, 
	\begin{align*}
		\Psi(d_\cX (\phi))&= m_1^{b_k,0} (\phi_k(\one_{\cL_k})|_{b_k=b_0}) = m_1^{b_k,0}(\Psi(\phi)),
	\end{align*} which shows $\Psi$ is a chain map.
	
	Next, we show that $\Psi$ is the right inverse to $\cF^\cL$:
	$$(\Psi \circ \cF^\cL)(p)=(\cF^\cL(p)_k(\one_{\cL_k}))\mid_{b_k=b_0}=(m_2^{b_k,b_0,0}(\one_{\cL_k},p))\mid_{b_k=b_0}=p.$$
	
	Using the same strategy, one can show that the mirror functor $\cF^\cL$ has the same properties for the union of Lagrangian immersions in $\cL'$. 
\end{proof}

\begin{remark}
	For Lagrangians $L_1$ and $L_2$ intersecting transversally, it happens that $L_1$ intersects with $\cL$, while $L_2$ does not. This implies that $\CF(L_1,L_2)\neq 0$. However, $\Hom_\cX(\cF^{\cL}(L_1),\cF^{\cL}(L_2))=0.$ Therefore, one won't expect faithfulness holds in general.
\end{remark}

\subsection{Fourier-Mukai transform from an algebroid stack to an algebra}

Given a Lagrangian immersion $\bL$, \cite{CHL-nc} constructed an $A_\infty$-functor
$$ \Fuk(M) \to \mathrm{dg-mod}(\A)$$
where $\A$ is the quiver algebra associated to $\bL$.  (As in the last section, we have assumed that $W=0$ for simplicity).  On the other hand, for a collection of Lagrangian immersions $\cL_1,\ldots,\cL_N$, we solve for a quiver algebroid stack $\cX$ and $\alpha_{ij} \in \CF((\cL_i,b_i),(\cL_j,b_j))$ that satisfy \eqref{eq:closed}, \eqref{eq:cocycle-alpha} and \eqref{eq:higher0}.  In this setting, we have constructed an $A_\infty$-functor
$$ \Fuk(M) \to \Tw(\cX)$$
in the last section.
We would like to compare these two functors.  This is a natural extension of Section \ref{section:extended fukaya} for a transformation between two algebras.

We shall consider bimodules as in Section \ref{section:extended fukaya}.  Below is a combination of Definition \ref{def: extended Fuk} and Definition \ref{def:CF_I}.

\begin{defn}
	The enlarged Fukaya category bi-extended over $T := T(\A,\cA_1,\ldots,\cA_N)$ has objects in $\Fuk(M)$ or $(\bL,b),(\cL_1,b_1),\ldots,(\cL_N,b_N)$, and morphism spaces between any two objects $L,L'$ are defined as follows.
	
	\begin{align*}
		\CF_i(L_0,L_1) :=& T(\cA_{i},\A) \otimes \CF(L_0,L_1) \otimes (T(\cA_{i},\A))^\op;\\
		\CF_i((\bL,b),L_1) :=& T(\cA_{i},\A) \otimes \A \otimes_{\Lambda^\oplus_\A} \CF(\bL,L_1) \otimes (T(\cA_{i},\A))^\op;\\
		\CF_i(L_0,(\bL,b)) :=& T(\cA_{i},\A) \otimes \CF(L_0,\bL) \otimes_{\Lambda^\oplus_\A} (T(\cA_{i},\A) \otimes \A)^\op;\\
		\CF_i((\bL,b),(\bL,b)) :=& T(\cA_{i},\A)\otimes \A  \otimes_{\Lambda^\oplus_\A} \CF(\bL,\bL) \otimes_{\Lambda^\oplus_\A} (T(\cA_{i},\A) \otimes \A)^\op;\\	
		\CF_j((\cL_j,b_j),L_1) :=& T(\cA_{j},\A) \otimes \cA_{j} \otimes_{(\Lambda^\oplus)_j} \CF(\cL_j,L_1) \otimes (T(\cA_{j},\A))^\op;\\
	\end{align*}
	\begin{align*}
		\CF_j(L_0,(\cL_j,b_j)) :=& T(\cA_{j},\A) \otimes \CF(L_0,\cL_j) \otimes_{(\Lambda^\oplus)_j} (T(\cA_{j},\A) \otimes \cA_{j})^\op;\\
		\CF_{jk}((\cL_j,b_j),(\cL_k,b_k)) :=& T(\cA_{j}(U_{jk}),\A) \otimes \cA_{j}(U_{jk}) \otimes_{(\Lambda^\oplus)_j} \CF(\cL_j,\cL_k)\\ &\otimes_{(\Lambda^\oplus)_k} (T(\cA_{k}(U_{jk}),\A) \otimes \cA_{k}(U_{jk}))^\op;\\
		\CF_j((\cL_j,b_j),(\bL,b)) :=& T(\cA_{j},\A) \otimes \cA_{j} \otimes_{(\Lambda^\oplus)_j} \CF(\cL_j,\bL) \otimes_{\Lambda^\oplus_\bL} (T(\cA_{j},\A) \otimes \A)^\op;\\
		\CF_j((\bL,b),(\cL_j,b_j)) :=& T(\cA_{j},\A) \otimes \A \otimes_{\Lambda^\oplus_\bL} \CF(\bL,\cL_j) \otimes_{(\Lambda^\oplus)_j} (T(\cA_{j},\A) \otimes \cA_{j})^\op.\\
	\end{align*}
	By pulling the coefficients to the left and right according to \eqref{eq:mk-bimod} and multiplying among $\cA_{j}$ using $\cM_{i_0\ldots i_k}$ \eqref{eq:mult}, we have the operations 
	$$ m_{k,\cX,\A}^{b_0,\ldots,b_k}: \CF_{U_1}(K_0,K_1) \otimes \ldots \otimes \CF_{U_k}(K_{k-1},K_k) \to \CF_{\left(\bigcap_j U_j\right) \cap \left(\bigcap_l U_{j_{K_l}}\right)} (K_0,K_k)$$
	where $K_l$ can be one of $(\cL_{j_{K_l}},b_{j_{K_l}})$, $(\bL,b_{j_{K_l}})$ (in which case we set $j_{K_l}=0$) or other Lagrangian (in which case $b_l=0$ and $j_{K_l}=\emptyset$). For brevity, we will denote $T(\A,\cA_1, \ldots, \cA_N)$ by $T(\A,\cX).$
\end{defn}

Similar to Theorem \ref{thm:A-infty-X}, $m_{k,\cX,\A}^{b_0,\ldots,b_k}$ satisfy $A_\infty$ equations.

\begin{defn}
	The universal sheaf $\U$ is defined as $\cF^\cL((\bL,b))$, which is a twisted complex of right $\A$-modules over $\cX$.  Namely, over each chart $U_i$, 
	$$\U_i = \A \otimes \cA_{i} \otimes_{(\Lambda^\oplus)_i} \CF(\cL_i,\bL) \otimes_{\Lambda^\oplus_\bL} \A^\op, \phi_i^\U=(-1)^{|-|}m_{1,\cX,\A}^{b_i,b}(-)).$$  
	The transition maps of $\U$ are defined by $\phi^\U_{ij}(-) := m_{2,\cX,\A}^{b_i,b_j,b}(\alpha_{ij},-): \U_j(U_{ij}) \to \U_i(U_{ij})$.  Similarly, we have the higher maps $\phi^\U_{i_0\ldots i_k}: \U_{i_k}(U_{i_0\ldots i_k}) \to \U_{i_0}(U_{i_0\ldots i_k})$ given by 
	$$\phi^\U_{i_0\ldots i_k}(-):=(-1)^{(k-1)|-|'} m_{k+1,\cX}^{b_{i_0},\ldots,b_{i_k},0}(\alpha_{i_0i_1},\ldots,\alpha_{i_{k-1}i_k},-).$$
\end{defn}

Then we have the dg functor 
\begin{equation}\cF^{\U}:=\Hom_{\cX}(\U,-): \Tw(\cX) \to \mathrm{dg}(\A-\mathrm{mod}).\label{defn:F^U}\end{equation}
We modify the signs as follows.  Given $\phi \in \Hom_{\cX}(\U,E)$,
its differential is given by 
$$(d_{\cF^\U(E)}\phi)= (-1)^{|\phi|} d_\cX(\phi)$$
where $d_\cX$ is defined by \eqref{eq:d-cA}.
Given $C,D\in \mathrm{dg}(\cX-\mathrm{mod})$, $f\in \Hom_{\cX}(C,D)$ and $\phi \in \Hom_{\cX}(\U,C)$, $$\cF^\U(f)(\phi)(-) = f \cdot \phi(-).$$

\begin{theorem} \label{thm:nat-trans-X-A}
	There exists a natural $A_\infty$-transformation $\cT$ from $\cF_1 = 
	\mathscr{F}^{(\bL,b)}$ to $\cF_2 = 
	\A \otimes (\cF^{\U}\circ \mathscr{F}^{\cL})$.
\end{theorem}
\begin{proof}
	First consider object level.  Given an object $L$ of $\Fuk(M)$, we define the following morphism (of objects in $\mathrm{dg} (\A-\textrm{mod})$)
	$$\mathscr{F}^{(\bL,b)}(L) = 
	\A \otimes_{\Lambda^\oplus_\A} \CF(\bL,L) \to 
	\A \otimes \cF^{\U} \left(\mathscr{F}^{\cL}(L)\right) = 
	\Hom_{\cX}(\U,\A \otimes\cF^{\cL}(L)).$$ 
	Over each chart $U_i$, for $\phi\in \mathscr{F}^{(\bL,b)}(L)$,
	$$\cT^L_{i}(\phi) := (-1)^{|\phi|'+|-|'}R\left(m_{2,\cX,\A}^{b_i,b,0}(-,\phi)\right)$$
	where $R$ is the operator that moves $\A^\op$ on the rightmost to $\A$ on the leftmost, see \eqref{eq:R}.  Over an intersection $U_{i_0\ldots i_k}$, 
	$$\cT^L_{i_0\ldots i_k}(\phi) := (-1)^{k(|\phi|'+|-|')+|\phi|'+|-|'}R\left(m_{k+2,\cX,\A}^{b_{i_0},\ldots,b_{i_k},b,0}(\alpha_{i_0i_1},\ldots,\alpha_{i_{k-1}i_k},-,\phi)\right). $$
	In the above expression, all coefficients of $\alpha_{i_{j-1}i_j}$ and $\phi$ appear on the left (with coefficient on the right being $1$); the only entry that can have non-trivial right-coefficients is the input $(-)$.  
	As in the proof of Theorem \ref{thm:nat-trans}, we denote 
	$$\bar{m}_{k+2,\cX,\A}^{b_{i_0},\ldots,b_{i_k},b,0} := R \circ m_{k+2,\cX,\A}^{b_{i_0},\ldots,b_{i_k},b,0}. $$
	It satisfies an analogous $A_\infty$ equation as \eqref{eq:A-infty-R}.  Thus $\cT^L_{i_0\ldots i_k}$ is a chain map:
	\begin{small}
		\begin{align*}
			& \sum_{j=1}^k (-1)^{j-1}\bar{m}_{k+1,\cX,\A}^{b_{i_0},\ldots,\hat{b}_{i_j}\ldots,b_{i_k},b,0}(\alpha_{i_0i_1},\ldots,m^{b_{i_{j-1}i_ji_{j+1}}}_{2,\cX}(\alpha_{i_{j-1}i_j},\alpha_{i_{j}i_{j+1}})\ldots,\alpha_{i_{k-1}i_k},-,\phi)\\
			&+
			\sum_{j=0}^{k} (-1)^j\bar{m}_{j+2,\cX,\A}^{b_{i_0},\ldots,b_{i_j},b,0}(\alpha_{i_0i_1},\ldots,\alpha_{i_{j-1}i_j},m^{b_{i_j},\ldots,b_{i_k},b}_{k-j+1,\cX,\A}(\alpha_{i_{j}i_{j+1}},\ldots,\alpha_{i_{k-1}i_k},-),\phi)\\
			&+
			\sum_{j=0}^{k}(-1)^{j} m_{j+1,\cX,\A}^{b_{i_0},\ldots,b_{i_j},0}(\alpha_{i_0i_1},\ldots,\alpha_{i_{j-1}i_j},\bar{m}^{b_{i_j},\ldots,b_{i_k},b,0}_{k-j+2,\cX,\A}(\alpha_{i_{j}i_{j+1}},\ldots,\alpha_{i_{k-1}i_k},-,\phi))\\
			&+(-1)^{k+|-|'} \bar{m}_{k+2,\cX,\A}^{b_{i_0},\ldots,\hat{b}_{i_k},0}(\alpha_{i_0i_1},\ldots,\alpha_{i_{k-1}i_k},-,m^{b,0}_{1}(\phi))  \\
			=&(-1)^{1+k(|-|'+|\phi|')} (\check{\partial}\cT^L(\phi))_{i_0\ldots i_k} -(-1)^{|\phi|'+1} (\cT^L(\phi)\cdot \U )_{i_0\ldots i_k} + (\cF^\cL\cdot \cT^L(\phi))_{i_0\ldots i_k} \\&+ (-1)^{|\phi|'}\cT^L_{i_0\ldots i_k}(d_{\cF^{(\bL,b)}(L)}\phi))\\
			=&(-1)^{1+k(|-|'+|\phi|')}(d_{\Hom_\cX(\U,\cF^\cL(L))} \circ \cT_{i_0\ldots i_k}^L +(-1)^{|\phi|'} \cT_{i_0\ldots i_k}^L \circ d_{\cF^{(\bL,b)}(L)})(\phi).
		\end{align*}
	\end{small}
	
	For morphisms and higher morphisms, let $L_0,\ldots,L_p$ be objects of $\Fuk(M)$ and $\phi_1 \otimes \ldots \otimes \phi_p \in \CF(L_0,L_1) \otimes \ldots \otimes \CF(L_{p-1},L_p)$.  Then we define a corresponding morphism
	\begin{small} 
		\begin{align*}
			\cT(\phi_1,\cdot,\phi_p): \cF^{(\bL,b)}(L_0) &\to \Hom_{\cX}(\U,\A \otimes \cF^\cL(L_p)),\\
			\left(\cT(\phi_1,\cdot,\phi_p)(\phi)\right)_{i_0\ldots i_k}(\cdot)&:= (-1)^{k(|\cdot|'+\sum_1^p+|\phi|')+|\cdot|'+|\phi|'}\bar{m}_{p+k+1,\cX,\A}^{b_{i_0},\ldots,b_{i_p},b,0,...,0}(\alpha_{i_0i_1},\ldots,\alpha_{i_{k-1}i_k},\cdot,\phi,\phi_1,\cdots,\phi_p). \label{eq:nat-trans}
		\end{align*}
	\end{small}
	(Recall that $\sum_1^r= \sum_{i=1}^r |\phi_i|'$ in \eqref{eq:Sigma}.)   
	
	Now we show that it satisfies the equations for the $A_\infty$-natural transformation $\cT$:
	\begin{align*}
		&(-1)^{1+\sum_1^p}d_{\Hom_{\cX}(\U,\A\otimes \cF^\cL(L_k))}\circ \cT(\phi_1,\ldots,\phi_p) +\sum_{r=0}^{p-1}  (-1)^{|\cT|'\sum_1^r}\cF_2(\phi_{r+1},\ldots,\phi_{p}) \circ \cT(\phi_1,\ldots,\phi_r)\\
		+&\sum_{r=1}^{p}  \cT(\phi_{r+1},\ldots,\phi_{p}) \circ \cF_1(\phi_1,\ldots,\phi_r)- \sum_{r=0}^{p-1}\sum_{l=1}^{p-r} (-1)^{\sum_{i=1}^r |\phi_i|'} \cT(\phi_1,\ldots,\phi_{r},m_l(\phi_{r+1},\ldots,\phi_{r+l}),\phi_{r+l+1},\ldots,\phi_p)= 0.
	\end{align*}
	
	The first term gives
	\begin{small}
		\begin{align*}
			& (-1)^{1+\sum_1^p}(d_{\Hom_{\cX}(\U,\A\otimes \cF^\cL(L_p))}( \cT(\phi_1,\ldots,\phi_p)(\phi)))_{i_0\ldots i_k} \\
			=&(-1)^{1+\sum_1^p} (\check{\partial}\cT(\phi_1,\ldots,\phi_p)(\phi))_{i_0\ldots i_k} + (-1)^{|\phi|'+\sum_1^p} (\cT(\phi_1,\ldots,\phi_p)(\phi)\cdot \U)_{i_0\ldots i_k}\\& + (\cF^\cL(L_p)\cdot \cT(\phi_1,\ldots,\phi_p)(\phi))_{i_0\ldots i_k}\\
			= (-1)^A&\sum_{j=1}^k (-1)^{j-1}\bar{m}_{p+k+1,\cX,\A}^{b_{i_0},\ldots,\hat{b}_{i_j}\ldots,b_{i_k},b,0,\ldots,0}(\alpha_{i_0i_1},\ldots,m^{b_{i_{j-1}i_ji_{j+1}}}_{2,\cX}(\alpha_{i_{j-1}i_j},\alpha_{i_{j}i_{j+1}})\ldots,\alpha_{i_{k-1}i_k},-,\phi,\phi_1,\ldots,\phi_p)\\
			&+(-1)^A
			\sum_{j=0}^{k} (-1)^j\bar{m}_{j+p+2,\cX,\A}^{b_{i_0},\ldots,b_{i_j},b,0,\ldots,0}(\alpha_{i_0i_1},\ldots,\alpha_{i_{j-1}i_j},m^{b_{i_j},\ldots,b_{i_k},b}_{k-j+1,\cX,\A}(\alpha_{i_{j}i_{j+1}},\ldots,\alpha_{i_{k-1}i_k},-),\phi,\phi_1,\ldots,\phi_p)\\
			&+(-1)^A
			\sum_{j=0}^{k} (-1)^j m_{j+1,\cX,\A}^{b_{i_0},\ldots,b_{i_j},0}(\alpha_{i_0i_1},\ldots,\alpha_{i_{j-1}i_j},\bar{m}^{b_{i_j},\ldots,b_{i_p},b,0,\ldots,0}_{p+k-j+2,\cX,\A}(\alpha_{i_{j}i_{j+1}},\ldots,\alpha_{i_{k-1}i_k},-,\phi,\phi_1,\ldots,\phi_p)),
		\end{align*}
	\end{small}
	where $A = p(|-|'+|\phi|'+\sum_1^p).$
	
	We compute the later terms as follows.
	\begin{small}
		\begin{align*}
			&(-1)^{|\cT|'\sum_1^r}(\cF^{\cL}(\phi_{r+1},\ldots,\phi_{p}) \cdot \cT(\phi_1,\ldots,\phi_r)(\phi))_{i_0\ldots i_k}\\
			= &
			\sum_{l=0}^{p} (-1)^A(-1)^{l}m_{l+p-r+1,\cX,\A}^{b_{i_0},\ldots,b_{i_l},0,\ldots,0}(\alpha_{i_0i_1},\ldots,\alpha_{i_{l-1}i_l},\bar{m}_{k-l+r+2,\cX,\A}^{b_{i_l},\ldots,b_{i_k},b,0,\ldots,0}(\alpha_{i_{l}i_{l+1}}\ldots,\alpha_{i_{k-1}i_k},\cdot,\phi,\\&\phi_1,\ldots,\phi_r),
			\phi_{r+1},\ldots,\phi_p);\\
			& (\cT(\phi_{r+1},\ldots,\phi_{p}) ( \cF_1(\phi_1,\ldots,\phi_r)(\phi))_{i_0\ldots i_k}\\
			=& (-1)^{A}(-1)^{k}\bar{m}_{p+k-r+2,\cX,\A}^{b_{i_0},\ldots,b_{i_k},b,0,\ldots,0}(\alpha_{i_0i_1},\ldots,\alpha_{i_{k-1}i_k},\cdot,m_{r+1}^{b,0,\ldots,0}(\phi,\phi_1,\ldots,\phi_r), \phi_{r+1},\ldots,\phi_p);\\
			&-(-1)^{\sum_1^r} (\cT(\phi_1,\phi_2,\ldots,\phi_r,m_l(\phi_{r+1},\cdots,\phi_{r+l}),\ldots,\phi_p)(\phi))_{i_0\ldots i_k}\\
			=& (-1)^{A}(-1)^{k+|\cdot|'+|\phi|'+\sum_1^r}\bar{m}^{b_{i_0},\ldots,b_{i_k},b,0,\ldots,0}_{p+3+k-l,\cX,\A}(\alpha_{i_0i_1},\ldots,\alpha_{i_{k-1}i_k},\cdot,\phi,\phi_1,\ldots,\phi_r,\\& m_l(\phi_{r+1},\ldots,\phi_{r+l}),\phi_{r+l+1},\ldots,\phi_p).
		\end{align*}
	\end{small}
	Result follows from $A_\infty$ equations for $\bar{m}_{k,\cX,\A}$.
\end{proof}

Similar to Theorem \ref{thm:one-side inverse}, the $A_\infty$-transformation $\cF_1 \xrightarrow{} \cF_2$ has a left inverse up to homotopy.

\begin{theorem} \label{thm:loc inj}
	Assume that there exist isomorphism pairs $\alpha_{0i} \in \cF^{\cL_i}(\bL), \alpha_{i0} \in \cF^{\bL}(\cL_{i})$ for some i. Then the natural transformation $\cT:\cF^{(\bL,b)} \to \A \otimes (\cF^{\U}\circ \mathscr{F}^{\cL})$ has a left inverse. Namely, $$\mathscr{F}^{(\bL,b)} \xrightarrow{} \A \otimes (\cF^{\U}\circ \mathscr{F}^{\cL}) \xrightarrow{} \A \otimes (\mathscr{F}^{\U} \circ \mathscr{F}^{\U^*} \circ \mathscr{F}^{(\bL,b)}) \xrightarrow{} \mathscr{F}^{(\bL,b)}$$ is homotopic to the identity natural transformation.
\end{theorem}
\begin{proof}
	By the previous theorem, we have natural transformations $\cT:\cF^{(\bL,b)} \to \A \otimes (\cF^{\U}\circ \mathscr{F}^{\cL})$ and $\cF^\U(\cT^{'}):\A \otimes (\cF^{\U}\circ \mathscr{F}^{\cL}) \xrightarrow{} \A \otimes (\mathscr{F}^{\U} \circ \mathscr{F}^{\U^*} \circ \mathscr{F}^{(\bL,b)})$.
	Define the last arrow above by $ev_{\alpha_{i0},\alpha_{0i}}$.  We get $$\bar{\cT}:=ev_{\alpha_{i0},\alpha_{0i}} \circ \cF^\U(\cT^{'}) \circ \cT: \cF^{(\bL,b)} \xrightarrow{} \cF^{(\bL,b)}.$$ 
	We want to show that it is homotopic to the identity natural transformation $\cI$ on $\cF^{(\bL,b)}$.
	
	For a Lagrangian $L$, we need to show that $\bar{\cT}_L$, which is an endomorphism on $\cF^{(\bL,b)}(L)$,  equals to the identity up to homotopy. 
	
	Over an intersection $U_{i_0\ldots i_k}$, for $\phi \in \cF^{(\bL,b)}(L)$, 
	$$\cT^L_{i_0\ldots i_k}(\phi) := (-1)^{k(|\phi|'+|-|')+|\phi|'+|-|'}\bar{m}_{k+2,\cX,\A}^{b_{i_0},\ldots,b_{i_k},b,0} (\alpha_{i_0i_1},\ldots,\alpha_{i_{k-1}i_k},-,\phi) $$ as in the theorem \ref{thm:nat-trans-X-A} .
	
	Note that $\cF^\U(\cT^{'L}) \circ \cT^L$ is a morphism of twisting complexes. Over an intersection $U_{i_0\ldots i_k} \cap U_{j_{0}\ldots j_{l}}$ with $j_l= i_0,$ up to sign we have
	$$\cF^\U(\cT^{'L}_{j_{0}\ldots j_{l}}) \circ \cT^L_{i_0\ldots i_k}(\phi):=\bar{m}_{l+2,\A,\cX}^{b,b_{j_{0}},\ldots,b_{j_l},0}(-, \alpha_{j_{0}j_{1}}, \ldots, \alpha_{j_{l-1}j_l},\bar{m}_{k+2,\cX,\A}^{b_{i_0},\ldots,b_{i_k},b,0} (\alpha_{i_0i_1},\ldots,\alpha_{i_{k-1}i_k},-,\phi)) $$
	
	If we further evaluate at $\alpha_{0i}, \alpha_{i0}$, by definition only 
	$\cF^\U(\cT^{'L}_{i}) \circ \cT^L_{i}(\phi)=\bar{m}_{2,\A,\cX}^{b,b_{i},0}(-,\bar{m}_{2,\cX,\A}^{b_i,b,0}(-, \phi))$ remains. Namely, 
	\begin{align*}
		\bar{\cT}^{L}(\phi) =&  \bar{m}_{2,\A,\cX}^{b,b_{i},0}(\alpha_{0i},\bar{m}_{2,\cX,\A}^{b_i,b,0}(\alpha_{i0}, \phi)) \\
		=&  \bar{m}_{2,\A,\cX}^{b,b_{i},0}(\bar{m}_{2,\cX,\A}^{b_i,b,0}(\alpha_{0i},\alpha_{i0}),\phi)+ \bar{m}_{3,\A,\cX}^{b,b_i,b,0}(\alpha_{0i},\alpha_{i0},m_1^{b,0}(\phi)) + m_1^{b,0}(\bar{m}_{3,\A,\cX}^{b,b_i,b,0}(\alpha_{0i},\alpha_{i0},\phi))\\
		=& \bar{m}_{2,\A,\cX}^{b,b_{i},0}(1_{\bL},\phi) + \cH_L\circ d_{\cF^{(\bL,b)}(L)} (\phi) + (-1)^{|\phi|'}d_{\cF^{(\bL,b)}(L)} \circ \cH_L (\phi)\\
		=& \phi + \cH_L\circ d_{\cF^{(\bL,b)}(L)} (\phi) +(-1)^{|\phi|'} d_{\cF^{(\bL,b)}(L)} \circ \cH_L (\phi).
	\end{align*}
	In the second line, we have used the $A_\infty$ equations, with the terms $\bar{m}_{1,\A,\cX}^{b,b_{i}}(\alpha_{0i})$ and $\bar{m}_{1,\cX,\A}^{b_i,b}(\alpha_{i0})$ vanish. And we define 
	$\cH_L := \bar{m}_{3,\A,\cX}^{b,b_i,b,0}(\alpha_{0i},\alpha_{i0},-)$
	as an endomorphism on $\cF^{(\bL,b)}(L)$ and the self pre-natural transformation as in theorem \ref{thm:one-side inverse} . Hence, $\bar{\cT}_L:\cF^{(\bL,b)}(L) \xrightarrow{} \cF^{(\bL,b)}(L)$ equals to identity up to homotopy in the object level.
	
	Then in the morphism level, for $\phi_1 \otimes \ldots \otimes \phi_k \in \CF(L_0,L_1) \otimes \ldots \otimes \CF(L_{k-1},L_k)$ ($k\geq 1$), 
	$$\bar{\cT}(\phi_1,\ldots,\phi_k)(\phi) = \sum_{r=0}^k (-1)^{\sum_1^k+|\phi|'} \bar{m}_{k-r+2,\A,\cX}^{b,b_i,0,\ldots,0}\left(\alpha_{0i},\bar{m}_{r+2,\cX,\A}^{b_i,b,0,\ldots,0}(\alpha_{i0},\phi,\phi_1,\ldots,\phi_r),\phi_{r+1},\ldots,\phi_k\right)$$
	Similar to theorem \ref{thm:one-side inverse} ,  $\bar{\cT} - \cI$ equals to the differential of $\cH_L$.

	Hence, the $A_\infty$-transformation $\cF_1 \xrightarrow{} \cF_2$ has a left inverse up to homotopy.
\end{proof}

In practical situations, we have $\alpha_{0i}$ and $\alpha_{i0}$ defined over certain localization $\A_{loc,i}$. Then theorem \ref{thm:loc inj} implies $\mathscr{F}^{(\bL,b)}|_{U_{i}}:= \A_{loc,i} \otimes_{\A} \mathscr{F}^{(\bL,b)} \xrightarrow{} \A_{loc,i} \otimes (\cF^{\U}\circ \mathscr{F}^{\cL})$ is injective. 

Assuming that there are enough charts of $\A$ such that $\alpha_{0i}, \alpha_{i0}$ are defined over certain localizations for all $i$, and any object $M^{\cdot}$ in dg ($\A$-mod) satisfies $M^{\cdot} \xrightarrow{} \prod_i \A_{loc,i} \otimes_{\A} M^{\cdot}$ is injective in the derived category of dg ($\A$-mod). We attain the injectivity of $\mathscr{F}^{(\bL,b)} \xrightarrow{} \A \otimes (\cF^{\U}\circ\mathscr{F}^{\cL})$.

\begin{remark} \label{rmk:surj}
	If $\U_i$ is a projective resolution for all $i$ and $\A \otimes (\cF^{\U}\circ \mathscr{F}^{\cL})|_{U_i} \cong \A_{loc,i} \otimes (\cF^{\U_i}\circ \mathscr{F}^{\cL_i})$, with Theorem \ref{thm: proj res}, we know $\mathscr{F}^{(\bL,b)}|_{U_{i}} \xrightarrow{} \A \otimes (\cF^{\U}\circ \mathscr{F}^{\cL})|_{U_i}$ is a quasi-isomoprhism. Besides, these quasi-isomorphisms agree on the overlap. Suppose any object $M^{\cdot}$ in dg ($\A$-mod) satisfies that
	\begin{equation} \label{eq:sheaf}
		\begin{tikzcd}
			M^{\cdot} \arrow[r] &
			\prod_i \A_{loc,i} \otimes_{\A} M^{\cdot} \arrow[r, shift left]
			\arrow[r, shift right] &
			\prod_{i,j} \A_{loc,ij} \otimes_{\A} M^{\cdot}
		\end{tikzcd}
	\end{equation} 
	is an equalizer in the derived category of dg ($\A$-mod). For any object $L$, the following diagram commutes in the derived category of dg ($\A$-mod)
	$$\begin{tikzcd}
		\mathscr{F}^{(\bL,b)}(L) \arrow[r] \arrow[d,dotted] &
		\prod_i \mathscr{F}^{(\bL,b)}(L)|_{U_{i}} \arrow[r, shift left]
		\arrow[r, shift right] \arrow[d]&
		\prod_{i,j} \mathscr{F}^{(\bL,b)}(L)|_{U_{ij}} \arrow[d] \\
		\A \otimes (\cF^{\U}\circ \mathscr{F}^{\cL})(L) \arrow[r] &
		\prod_i \A \otimes (\cF^{\U}\circ \mathscr{F}^{\cL})(L)|_{U_{i}} \arrow[r, shift left]
		\arrow[r, shift right] &
		\prod_{i,j} \A \otimes (\cF^{\U}\circ \mathscr{F}^{\cL})(L)|_{U_{ij}} 
	\end{tikzcd} $$
	where the two vertical arrows are isomorphisms and the dotted arrow comes from the universal property of the equalizer. By the universal property, $\mathscr{F}^{(\bL,b)}(L)$ is quasi-isomorphic to $\A \otimes (\cF^{\U}\circ \mathscr{F}^{\cL})(L)$ for any object $L$.
\end{remark}




\section{NC Local Projective Plane}
\label{chapter:KP2}

In this section, we apply the method introduced in the previous section to construct a quiver stack as the mirror space of a three-punctured elliptic curve $M$.  The resulting quiver stack (extended over $\Lambda$) consists of two parts.  One is a quiver algebra $\A$ with relations (see the right of Figure \ref{fig:C3Q3}), which is the (noncommutatively deformed) quiver resolution of $\C^3/\Z_3$ in the sense of Van den Bergh \cite{vdBergh}. Another part is an algebroid stack $\cY$, which is nc deformed $K_{\bP^2}$ as a manifold (see Figure \ref{fig:ncKP2}).

As a result, we construct two $A_\infty$ functors $\cF^{\bL}: \Fuk(M) \to \mathrm{dg-mod}(\A)$ and $\cF^{\cL}: \Fuk(M) \to \mathrm{Tw}(\cY)$.  Moreover, we construct the universal sheaf $\U = \cF^{\cL}(\bL)$ that induces a dg-functor $\cF^\U: \mathrm{Tw}(\cY) \to \mathrm{dg-mod}(\A)$.  This realizes the commutative diagram \eqref{eq:diagram}.  All these can be explicitly calculated from the ($\Z$-graded) Lagrangian Floer theory on the punctured elliptic curve.

The key step is to find isomorphisms between the local Seidel Lagrangians $\cL_i$ and the Lagrangian skeleton $\bL$ of $M$.  For instance, the isomorphism pair we have found between $\cL_3$ and $\bL$ is 
$$(\alpha_3, \beta_3) = \left(-{Q}^{2,3},(T^{-W}1\otimes b_3^{-1}b_1^{-1})\overline{P^{3,3}}\right)$$
where ${Q}^{2,3}$ and $\overline{P^{3,3}}$ are intersection points shown in Figure \ref{fig:KP2_m2}.

\subsection{Non-Archimedean quiver algebroid stacks}\label{sec: narch}
In the previous sections, we focus on algebraic gluing and do not specifically work on the Novikov field $\Lambda$. On the other hand, it is necessary to consider non-Archimedean norms and completions for Lagrangian Floer theory and mirror symmetry, since the generating functions of pseudo-holomorphic polygons are generally infinite series and enjoy convergence with respect to certain valuations. In this subsection, we extend the notion of non-Archimedean norms to noncommutative algebras.

First, we generalize the definition of a valuation for a noncommutative ring $R$. 
\begin{defn} \label{def:val}
	Let $R$ be a ring. A valuation on $R$ is a function $\val: R \to \R \cup \{\infty\}$ that satisfies the following. For all $a,b \in R$,
	\begin{enumerate}
		\item $\val(ab) \geq \val(a) + \val(b)$;
		\item $\val(a+b) \geq \min(\val(a),\val(b))$;
		\item $\val(a)=\infty$ if and only if $a=0$.
	\end{enumerate}
\end{defn}
The only modification we have made is the first one: we change the equality $\val(ab) = \val(a) + \val(b)$ for valuation on a commutative ring to the above inequality.

Define $\norm{a} := e^{-\val(a)}$. Then the above definition translates to the definition of a non-Archimedean norm.
\begin{defn}
	Let $R$ be a ring. A non-Archimedean norm on $R$ is a function $\norm{\cdot}:R \to \R_{\geq 0}$ that satisfies the following. For all $a,b \in R$,
	\begin{enumerate}
		\item $\norm{ab} \leq \norm{a}\norm{b}$;
		\item $\norm{a+b} \leq \max\{\norm{a},\norm{b}\}$;
		\item $\norm{a} = 0$ if and only if $a=0$.
	\end{enumerate}
\end{defn}
The first inequality is a common condition for norms on matrix algebras. Equality $\norm{ab} = \norm{a}\norm{b}$ is satisfied for scalars but not for matrices. This is the main motivating reason we change this to the above inequality. Moreover, for quiver algebra, two paths $a$ and $b$ may not concatenate which gives $ab=0$. Then this inequality is automatically satisfied.

\begin{example}
	Consider the algebra of $\Lambda$-valued $n$-by-$n$ matrices. Define
	$$ \val (A) := \frac{1}{2} \val_\Lambda (\tr (AA^*)) $$
	where $A^*$ denotes the conjugate transpose of $A$.
	Explicitly, writing each non-zero matrix elements as $a_{ij} = T^{E_{ij}} c_{ij} (1 + o(T))$ where $E_{ij} \in \R$, $c_{ij} \in \C^\times$, and $o(T) \in \Lambda_+$, we have
	\begin{align*}
		\val_\Lambda (\tr (AA^*)) =& \val_\Lambda \sum_{i,j} a_{ij} \bar{a_{ij}} \\
		=& \val_\Lambda \sum_{i,j: a_{ij} \not=0} T^{2 E_{ij}} |c_{ij}|^2 (1 + o(T))(1 + \bar{o}(T)) \\
		=& 2 \min_{i,j} \val_\Lambda (a_{ij})
	\end{align*}
	where the last equality holds because $|c_{ij}|^2 > 0$. 
	In particular, $\val(A)=+\infty$ if and only if $A=0$.
	Thus $\val A = \min_{i,j} \val_\Lambda a_{ij}$.
	In other words, $\val A$ is the maximal number such that $A / T^{\val A}$ has every entry in $\Lambda_{\geq 0}$.
	
	It is obvious that $\val(A)=\infty$ if and only if $A=0$.
	Now we check the conditions $\val(AB) \geq \val(A) + \val(B)$ and $\val(A+B) \geq \min(\val(A),\val(B))$. They are obvious if one of the matrices is zero, so let's assume $A\not=0$ and $B\not=0$. Let's write $A = T^{\val A} A_0$ and $B = T^{\val B} B_0$ where $A_0$ and $B_0$ have every entry in $\Lambda_{\geq 0}$ and at least one entry in each matrix has valuation zero. Then $AB = T^{\val A + \val B} (A_0 B_0)$ and $A_0 B_0$ has every entry in $\Lambda_{\geq 0}$. Thus $\val (AB) \geq \val A + \val B$.
	
	For the second condition,
	\begin{align*}
		\val(A+B) =& \min_{i,j} \val_\Lambda(a_{ij} + b_{ij}) \\
		\geq & \min_{i,j} \min(\val_\Lambda (a_{ij}), \val_\Lambda (b_{ij})) \\
		= & \min (\min_{i,j} \val_\Lambda (a_{ij}), \min_{i,j} \val_\Lambda (b_{ij})) \\
		= & \min (\val A, \val B).
	\end{align*}
\end{example}

In the following subsections, we will work with the three-dimensional noncommutative non-Archimedean Euclidean space. Fixing the valuation of each variable, we equip it with a valuation given as follows.

\begin{example} \label{ex:val_ncC3}
	Let $\cA^\hbar = \Lambda \langle w,y,x\rangle/\partial (yxw-T^{-3\hbar}xyw)$ be the noncommutative algebra given in Proposition \ref{prop:cA_1}. We have the relations 
	$$yx = T^{-3\hbar} xy, xw = T^{-3\hbar} wx, \textrm{ and } wy = T^{-3\hbar} yw.$$ 
	Given $v = (v_x,v_y,v_w) \in (\R \cup \{\infty\})^3$, we define a valuation $\val_v$ on $\cA^\hbar$ as follows. For simplicity we write $\val = \val_v$ for a fixed $v$. First we set
	\begin{align*}
		\val(y) &= v_y, \\
		\val(x) &= v_x, \\
		\val(w) &= v_w.
	\end{align*}
	Moreover, we set $\val(y^kx^l) = k v_y + l v_x$, and similarly for $\val(x^kw^l)$ and $\val(w^ky^l)$.
	Then using the relations, we have $\val(x^{l_0}y^{k_1}x^{l_1}\ldots y^{k_m}x^{l_m}) \geq k v_y + l v_x$ where $\sum_{i=0}^m l_i = l$ and $\sum_{i=1}^m k_i = k$.
	For a general monomial with $k_y, k_x, k_w$ numbers of $y,x,w$ respectively, we consider
	$y^{k_y}x^{k_x}w^{k_w}$ if $k_x$ is maximal among $k_y, k_x, k_w$, $x^{k_x}w^{k_w}y^{k_y}$ if $k_w$ is maximal, and $w^{k_w}y^{k_y}x^{k_x}$ if $k_y$ is maximal. 
	We can check that such monomials have the minimal valuation among their permutations with the given relations. Then we define
	$$ \val(T^A y^{k_y}x^{k_x}w^{k_w}) = A + k_y v_y + k_x v_x + k_w v_w $$
	and similarly for $\val(T^A x^{k_x}w^{k_w}y^{k_y})$ and $\val(T^A w^{k_w}y^{k_y}x^{k_x})$.
	By this definition, the condition $\val(ab) \geq \val(a) + \val(b)$ holds for monomials $a,b$: let $a_0$ and $b_0$ be the monomials obtained from permutation of factors of $a$ and $b$ respectively such that $a_0$ and $b_0$ have the minimal valuation among all the permutations. 
	Then $\val (a) = \val(a_0) + 3k\hbar$ and $\val (b) = \val(b_0) + 3l\hbar$ for some non-negative integers $k,l$.
	Combining the two permutations, we have $\val(ab) = \val(a_0 b_0) + 3k\hbar + 3l\hbar$. 
	We can further permute $a_0b_0$ to achieve a monomial that has the minimal valuation which equals $\val(a_0) + \val(b_0)$.
	Thus $\val(a_0 b_0) \geq \val(a_0) + \val(b_0)$.
	Combining, we get
	$$ \val(ab) = \val(a_0 b_0) + 3k\hbar + 3l\hbar \geq \val(a_0) + \val(b_0) + 3k\hbar + 3l\hbar = \val(a) + \val(b). $$
	
	For a polynomial in $\cA^\hbar$, we define its valuation being the minimum valuation among all of its monomial terms. Then a polynomial can be written as $P = P_0 + P_1$, where $P_0$ consists of all the terms with valuation being $\val(P)$ and $\val P_1 > \val P$. For two polynomials $P, Q$, we write
	$$ PQ = (P_0+P_1)(Q_0+Q_1) = P_0 Q_0 + P_0 Q_1 + P_1 Q_0 + P_1 Q_1 $$
	and so $\val(PQ) = \val(P_0Q_0)$. Every term in the polynomial expansion of $P_0 Q_0$ has valuation $\geq \val(P_0) + \val(Q_0) = \val(P) + \val(Q)$. Thus $\val(PQ) = \val(P_0Q_0) \geq \val(P) + \val(Q)$.
	
	The other two conditions, namely $\val(a+b) \geq \min(\val(a),\val(b))$ and $\val(a)=\infty$ if and only if $a=0$, are standard and easy to check.
	
	Given $v \in (\R\cup \{\infty\})^3$, we have the non-Archimedean norm $\norm{a}_v := e^{-\val_v(a)}$ on $\cA^\hbar$. 
	Then we define $\overline{\cA^\hbar}^v$ to be the subalgebra of formal power series in $\cA^\hbar$ which are convergent with respect to this norm. The fact that this is a subalgebra easily follows from Properties (1) and (2) of the norm.
	For an open subset $U \subseteq (\R\cup \{\infty\})^3$, we define the completion
	\begin{equation} \label{eq:completion}
		\overline{\cA^\hbar}(U) := \bigcap_{v \in U} \overline{\cA^\hbar}^v.
	\end{equation} 
	By definition $\overline{\cA^\hbar}(U) \subset \overline{\cA^\hbar}(V)$ if $V \subset U$, and this gives a sheaf over $(\R\cup \{\infty\})^3$ which we denote by $\overline{\cA^\hbar}$.
\end{example}

Generally, given a family of non-Archimedean norms on a quiver algebra $\cA = \Lambda Q / R$ parametrized by a topological space $B$, we define the sheaf of convergent series $\overline{\cA}$ over $B$ as in \eqref{eq:completion}. 
Below we define non-Archimedean norms on a noncommutative resolution of $\C^3/\Z_3$, which will be the main example in the following sections.

\begin{example} \label{ex:val_ncKP2}
	Consider the quiver algebra $\A^\hbar = \Lambda Q/\partial \Phi$, where $Q$ is the quiver in Figure \ref{fig:C3Q3} and $\Phi = -T^{\hbar}(b_1c_3a_2+a_1b_3c_2+c_1a_3b_2)+(c_1b_3a_2+b_1a_3c_2+a_1c_3b_2)$. For instance, one of the relations is $c_1b_3 = T^\hbar b_1c_3$.
	
	Given $v = (v_a,v_b,v_c) \in (\R \cup \{\infty\})^3$, we define a valuation $\val = \val_v$ on $\A^\hbar$ as follows. The valuation of idempotents $e_i$ at the three vertices $i=1,2,3$ are defined to be $0$. We set $\val(a_i) = v_a$, $\val(b_i) = v_b$, $\val(c_i) = v_c$ for all $i=1,2,3$.
	For a monomial starting with vertex $i$ with $k_a, k_b, k_c$ numbers of $a,b,c$ respectively (where $a_1,a_2,a_3$ are considered to be $a$, and similar for $b$ and $c$), we consider $a_{i+k_c+k_b+k_a-1} \ldots a_{i+k_c+k_b} b_{i+k_c+k_b-1}\ldots b_{i+k_c}c_{i+k_c-1}\ldots c_i$
	if $k_b$ is maximal among $k_a, k_b, k_c$, and similarly for the remaining two cases by cyclic permuting $a,b,c$.
	We can check that such monomials have the minimal valuation among their permutations with the given relations. Then we define
	$$ \val(T^A a_{i+k_c+k_b+k_a-1} \ldots a_{i+k_c+k_b} b_{i+k_c+k_b-1}\ldots b_{i+k_c}c_{i+k_c-1}\ldots c_i) = A + k_c v_c + k_b v_b + k_a v_a $$
	and similarly for the other two cases.
	As in Example \ref{ex:val_ncC3}, we can check that this defines a valuation on $\A^\hbar$.
	We have the sheaf of convergent series $\overline{\A^\hbar}$ over $(\R\cup \{\infty\})^3$.
\end{example}

Next, we would like to construct a local ring from $\overline{\A^\hbar}$. Let's first recall the definition of a local ring.
\begin{defn}
	A ring $R$ is said to be local if for every $x\in R$, at least one of $x$ or $1-x$ is invertible.
	
	Let $e\in R$ be an idempotent of $R$. $e$ is called a local idempotent if $e R e$ is a local ring.
\end{defn}

A quiver algebra with more than one vertices has idempotents and hence cannot be local. Instead, we consider if $e_i \A e_i$ are local rings for all vertices $i$. Note that $e_i$ serves as the identity in $e_i \A e_i$.

We take $\overline{\A^\hbar}_{\geq 0}$, which is defined as the subring of elements in $\overline{\A^\hbar}$ that has non-negative valuation (norm less than or equal to 1) with respect to every $v \in \R_{>0}^3$.
Note that $\overline{\A^\hbar}_{\geq 0}$ is no longer an algebra over $\Lambda$ and is a module over $\Lambda_0$.

For convergence in Floer theory, we need to restrict the valuation of each arrow that corresponds to an immersed sector of a Lagrangian to be a positive real number. On the other hand, if we just concern about gluing of the space itself (without Floer theory), this may not be necessary and we may take the valuation of each arrow to be an arbitrary real number.

\begin{prop}
	Let $\A$ be the quiver algebra given in Construction \ref{constr:nc-single} for a compact Lagrangian immersion $\bL$. For any valuation $\val$ on $\A$ such that $\val(a) > 0$ for every arrow $a \in \A$, the $A_\infty$-operations $m_k^b$ for the family $(\bL,b)$ over $\A$ have coefficients lying in $\overline{\A^\hbar}_{\geq 0}$.
\end{prop}
\begin{proof}
	By Gromov compactness, for each $K>0$, there are only finitely many polygons with energy $< K$. Since $\val(a)>0$ for every arrow $a$ and by (1) of Definition \ref{def:val} that valuation of a path $\gamma$ is greater than or equal to the sum of that for the individual arrows, the valuation of each non-trivial path is positive. Thus there are just finitely many terms $T^A \gamma$ in $m_0^b$ that has valuation $< K$. Thus $m_0^b$ is convergent under such a valuation. Moreover, each term  $T^A \gamma$ has non-negative valuation (and has zero valuation if and only if $A=0$ and $\gamma$ is a trivial path, in which case the corresponding polygon must be constant). Thus the coefficients of $m_0^b$ lie in $\overline{\A^\hbar}_{\geq 0}$.
\end{proof}

\begin{prop}
	For every vertex $i$ of the quiver $Q$ in Example \ref{ex:val_ncKP2}, $e_i$ is a local idempotent of $\overline{\A^\hbar}_{\geq 0}$.
\end{prop}
\begin{proof}
	In this case, the non-invertible elements $x \in e_i \overline{\A^\hbar}_{\geq 0} e_i$ are those series that have every term with path length at least 1. Since the valuation of the variables can be arbitrarily closed to $0$, the coefficient of each minimal monomial must have valuation $\geq 0$. Thus $x$ has positive valuation. Then $e_i/(e_i-x) = e_i + \sum_{k=1}^\infty x^k$ is the inverse of $e_i - x$. This shows that $e_i \overline{\A^\hbar}_{\geq 0} e_i$ is a local ring.
\end{proof}

We glue these rings into a non-Archimedean quiver algebroid stack which is defined as follows.

\begin{defn}
	A non-Archimedean quiver algebroid stack is a quiver algebroid stack $\cA$ over a topological space $B$ whose stalks $\cA_b$ are rings equipped with non-Archimedean norms $\norm{\cdot}_b$ such that $\cA_b$ are complete with respect to $\norm{\cdot}_b$ for all $b\in B$. More concretely, for each multi-index $I$ and $i\in I$, we have a family of non-Archimedean norms $\norm{\cdot}_b$ for $b \in U_I$ on $\cA_i(U_I)$ such that $\cA_i(U_I)$ is complete with respect to $\norm{\cdot}_b$ for all $b \in U_I$. Moreover, for $i,j \in I$, the transition map $\cA_i(U_I) \to \cA_j(U_I)$ is an isometry with respect to the non-Archimedean norms $\norm{\cdot}_b$ on both sides.
\end{defn}

\begin{example}
	Consider the polynomial algebra $\Lambda[x]$ and $B = [0,1)$. For each $b \in B$, we assign $\val(x) := -\log b \in (0,+\infty]$ which gives a valuation $\val_b$ on $\Lambda[x]$. Then we take the completed local ring $\overline{\Lambda[x]}^B_{\geq 0}$ which consists of all series that are convergent and valued in $\Lambda_{\geq 0}$ with respect to $\val_b$ for all $b \in B$.
	
	Now consider two copies $\Lambda[x]$ and $\Lambda[z]$ with the transition map $x \mapsto T^B z^{-1}$ where $B>0$ is fixed. They are both over the interval $[0,1)$. The transition map gives $\val (x) = B - \val (z)$. Since $\val (x) > 0$, $\val (z) < B$. In other words, we glue the two intervals by the transition map $b_x = e^{-B} b_z^{-1}$ where the overlapping region is $(e^{-B},1)$ in each of the two intervals. They glue to a closed interval.
	
	By construction $\norm{f(x)}_{b_x} = \norm{f(T^B z^{-1})}_{e^{-B} b_z^{-1}}$ for any polynomial $f$. In the overlapping region, we take the completed local ring $\overline{\Lambda[x,x^{-1}]}^{(e^{-B},1)}_{\geq 0} \cong \overline{\Lambda[z,z^{-1}]}^{(e^{-B},1)}_{\geq 0}$. We get a non-Archimedean algebroid stack (namely a projective line) over the closed interval.
\end{example}

In the above basic example, we glue the base according to the valuation of the transition maps for the algebroid stack. We do the same for the quiver algebroid stack that we construct in the following subsections and hence obtain a non-Archimedean quiver algebroid stack.

\begin{example}
	Consider the noncommutative $K_{\bP^2}$ glued from three affine charts as given by Equation \eqref{eq:relation_between_Seidel}.
	Here, the valuations for the variables (such as $v_{x_1}, v_{y_1}, v_{w_1}$) are taken in $\R \cup \{+\infty\}$.
	Taking $e^{-v}$, the corresponding base is glued by three copies of $\R_{\geq 0}^3$ via the equations
	\begin{equation}\label{eq:gluing_ncKP2}
		\begin{cases}
			X_1 \mapsto e^{B+\hbar}Z_2^{-1}\\
			Y_1 \mapsto e^{\frac{B}{2}+2 \hbar}Y_2 Z_2^{-1}\\
			W_1 \mapsto e^{-\frac{3B}{2}-9\hbar}W_2 Z_2^{3}
		\end{cases}
		\begin{cases}
			Y_2 \mapsto e^{B+\hbar}X_3^{-1}\\
			Z_2 \mapsto e^{\frac{B}{2}+2\hbar}Z_3 X_3^{-1}\\
			W_2 \mapsto e^{-\frac{3B}{2}-9\hbar}W_3 X_3^{3}
		\end{cases}
		\begin{cases}
			Z_3 \mapsto e^{B+\hbar}Y_1^{-1}\\
			X_3 \mapsto e^{\frac{B}{2}+2 \hbar}X_1 Y_1^{-1}\\
			W_3 \mapsto e^{-\frac{3B}{2}-9\hbar}W_1 Y_1^{3}.
		\end{cases}
	\end{equation}
	where $X_1 = e^{-v_{x_1}}$ and so on. (Note that these are now commutative coordinates of $\R_{\geq 0}^3$.)
	The base is homeomorphic to a toric polytope of $K_{\bP^2}$.
	The geometric charts coming from Floer theory to be considered in the next subsection restrict $X_1 < 1$ (and similarly for other variables) which give three disjoint subsets $[0,1)^3$ in the base.
	
	We have another chart given by the quiver algebra (the nc resolution) in Example \ref{ex:val_ncKP2}, which is glued to the above three charts via Equation \eqref{eq:Seidel_Lag}. 
	Let $\alpha = e^{-v_a}, \beta = e^{-v_b}, \gamma = e^{-v_c}$. Here we take $(\alpha,\beta,\gamma) \in \R_{\geq 0}^3 - \{(0,0,0)\}$. This is homeomorphic to the toric cone of $\C^3/\Z_3$ with the origin removed.
	Then the gluing for the base is given by
	\begin{equation}\label{eq:Seidel_Lag}
		\begin{cases}
			X_1 \mapsto e^{\frac{B}{2}-\hbar}\alpha\gamma^{-1} \\
			Y_1 \mapsto e^{\frac{B}{2}+\hbar}\alpha\beta^{-1} \\
			W_1 \mapsto e^{-B}\alpha^3
		\end{cases}
		\begin{cases}
			Y_2 \mapsto e^{\frac{B}{2}-\hbar}\gamma\beta^{-1}\\
			Z_2 \mapsto e^{\frac{B}{2}+\hbar}\gamma\alpha^{-1}\\
			W_2 \mapsto e^{-B}\gamma^3
		\end{cases}
		\begin{cases}
			Z_3 \mapsto e^{\frac{B}{2}-\hbar}\beta\alpha^{-1}\\
			X_3 \mapsto e^{\frac{B}{2}+\hbar}\beta\gamma^{-1}\\
			W_3 \mapsto e^{-B}\beta^3.
		\end{cases}
	\end{equation}
	Note that we need to remove the origin in order for the above gluing to be well-defined.
	At least one of $\alpha,\beta,\gamma$ is nonzero, say $\alpha \neq 0$.
	Then the first equation of the above is a homeomorphism in the overlapping region.
	
	This gives a non-Archimedean quiver algebroid stack (namely the nc $K_{\bP^2}$) over the polytope base of $K_{\bP^2}$.
\end{example}

\subsection{Construction of the Algebroid Stack}\label{sec:example}
In \cite{CHL-nc}, the quiver resolution of $\C^3/\Z_3$ was constructed as the mirror space using a (normalized) Lagrangian skeleton $\bL$ of the three-punctured elliptic curve $M$.  $\bL$ is a union of three circles, $\bL=L_1\cup L_2\cup L_3$, see Figure \ref{fig:KP2_all}.  $M$ can be constructed as a 3-to-1 cover of the pair-of-pants $\bP^1 - \{\textrm{three points}\}$, and $\bL$ is the lifting of a Seidel Lagrangian in the pair-of-pants \cite{Seidel-g2}.  Alternatively, $\bL$ can also be understood as vanishing cycles of the LG mirror $z_1 + z_2 + \frac{1}{z_1z_2}$ of $\bP^2$, by identifying $M$ with $\{z_1 + z_2 + \frac{1}{z_1z_2}=0\} \subset (\C^\times)^2$.  $\bL$ can also be constructed from a dimer model, see for instance \cite{FHKV}, \cite{Ishii-Ueda-McKay}.  Note that $\bL$ has a ramified 2-to-1 cover to a Lagrangian skeleton of $M$.  $\bL$ is an immersed Lagrangian, while the Lagrangian skeleton is too singular for defining Lagrangian Floer theory analytically.

On the other hand, to produce a geometric resolution of $\C^3/\Z_3$, we can decompose $M$ into three pair-of-pants and consider Seidel Lagrangians $S_1,S_2,S_3$ as their normalized Lagrangian skeletons.  See Figure \ref{fig:KP2_all}.  Note that these Seidel Lagrangians do not intersect with each other, so their deformation spaces (over $\Lambda$) are disjoint and do not directly glue into a (connected) manifold.  In \cite{CHL3}, deformed copies of Seidel Lagrangians were added in order to produce a connected space.  However, homotopies and gradings are rather complicated in this approach for constructing a threefold.  We proceed in another method as we shall see below.

We fix non-trivial spin structures on $\bL$ and $S_i$, whose connections act as $(-1)$ at the points marked by stars in the figure.  We also fix a perfect Morse function on each Lagrangian, whose maximum point (representing the fundamental class) are marked by circles.  Moreover, we denote by $Q_0^{i,j}, Q_1^{i,j},Q_2^{i,j}$ and $P_1^{i,j},P_2^{i,j},P_3^{i,j}$ the even and odd degree generators in $\CF(L_i,S_j)$ respectively.  We simply write $Q^{i,j}=Q_0^{i,j}$ and $P^{i,j}=P_3^{i,j}$.  See Figure \ref{fig:KP2_area} for notations of areas $A_i,A_i'$ for $i=1,\ldots,5$.  (We will use the notation $A_{i_0\ldots i_k} = A_{i_0} + \ldots + A_{i_k}$.)  We shall make the simplifying assumption on the areas: $A_2=A_2'=A_4=A_4'=A_3=0$, and $A_5=A_5'$.  Then we can express all area terms in terms of $$B = A_{1123455'} \textrm{ and } \hbar = A_1-A_1'.$$

\begin{figure}[htb!]
	\centering
	\includegraphics[scale=0.9,trim={15.5cm 8.5cm 17.5cm 10cm},clip]{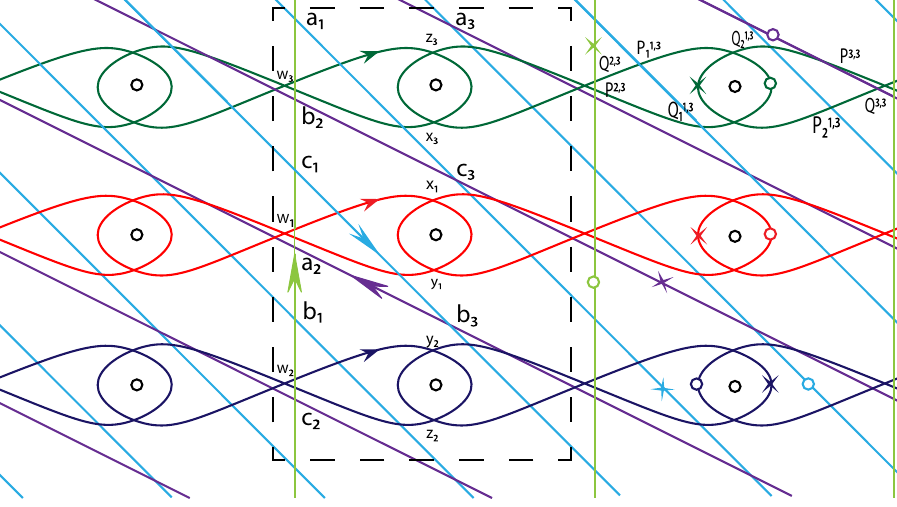}
	\caption{Lagrangians in $M$.}
	\label{fig:KP2_all}
\end{figure}

The variables are named such that they obey the following cyclic symmetry:
\begin{equation}\label{transformation_rule}
	\begin{cases}
		x_3\leftrightarrow z_2\leftrightarrow y_1\\
		z_3 \leftrightarrow y_2 \leftrightarrow x_1\\
		w_3 \leftrightarrow w_2 \leftrightarrow w_1;
	\end{cases}
	\begin{cases}
		a_1\leftrightarrow b_1\leftrightarrow c_1\\
		b_2 \leftrightarrow c_2 \leftrightarrow a_2\\
		c_3 \leftrightarrow a_3 \leftrightarrow b_3.
	\end{cases}
\end{equation}

We recall the following proposition for $\bL$ from \cite{CHL-nc}.

\begin{prop}[Lemma 10.13 in \cite{CHL-nc}]
	Consider the formal nc deformation parameter $\bb = \sum_{i=1}^3 a_i A_i+b_iB_i+c_iC_i$ of $\bL$, where $A_i, B_i, C_i$ are generators of $\CF^1(\bL)$ and $a_i,b_i,c_i$ are the corresponding quiver arrows.  The nc unobstructed deformation space is $\A^\hbar = \Lambda Q/\langle\partial \Phi \rangle$, where $Q$ is the quiver in Figure \ref{fig:C3Q3}, $\Phi = -T^{\hbar}(b_1c_3a_2+a_1b_3c_2+c_1a_3b_2)+(c_1b_3a_2+b_1a_3c_2+a_1c_3b_2)$ and $\partial$ denotes the cyclic derivative.
\end{prop}

\begin{remark}
	Indeed, we are applying the mirror construction to a $\Z$-graded $A_\infty$ category of Lagrangians, rather than the $\Z_2$-graded Fukaya category of Lagrangians in Riemann surfaces.  Below, we give a $\Z$-grading to the collection of immersed Lagrangians $\{\bL,S_1,S_2,S_3\}$.  In this paper, we simply check by hand that the resulting objects obtained from mirror transform are well-defined.  In a forthcoming work, we will prove that the grading gives an $A_\infty$ category.
	
	We may also use $\Z_2$-grading.  Then we have Landau-Ginzburg superpotentials on the mirror quiver algebra $\A$ and the mirror stack $\cY$.  Moreover, the universal bundle in the next subsection will become glued matrix factorizations rather than twisted complexes.
\end{remark}

The grading on $\bL$ and $S_i$ individually are straight-forward: the odd and even immersed generators are equipped with degree $1$ and $2$ respectively; the degrees of point class and fundamental class are assigned to be $0$ and $3$.  For $\CF(L_i,S_j)$, $Q^{i,j}$ is assigned with degree $0$, $P_1^{i,j},P_2^{i,j}$ are of degree $1$, $Q_1^{i,j},Q_2^{i,j}$ are of degree $2$, and $P^{i,j}$ has degree $3$.  Their complementary generators in $\CF(S_j,L_i)$ have degree $3-d$.

We denote the local deformation space of each Seidel Lagrangian $S_i$ by $\cA_i^\hbar$.  As we shall see, they serve as affine charts of $\A^\hbar$.  The deformation space for the Seidel Lagrangian was computed in \cite{CHL}.

\begin{prop}[\cite{CHL}] \label{prop:cA_1}
	Consider the Seidel Lagrangian $S_1$ with the given orientation, fundamental class and spin structure in Figure \ref{fig:KP2_all}.  Consider the formal nc deformations $\bb_1 = w_1 W_1 + y_1 Y_1 + x_1 X_1$ of $S_1$. The noncommutative deformation space of $S_1$ is $\cA_1^\hbar = \Lambda \langle w_1,y_1,x_1\rangle/\langle\partial \Phi_1 \rangle$, where 
	$$\Phi_1 = y_1x_1w_1-T^{-3\hbar}x_1y_1w_1.$$
\end{prop}

\begin{proof}
	The main step is computing NC Maurer-Cartan relations. Namely, by quotient out the coefficients $P_f$ of the degree 2 generators $X_f$ of $\CF(S_1,S_1)$ in $m_0^{\bb_1}= m(e^{\bb_1})= \sum_{f}P_f X_f$, we obtain the nc deformation space $\cA_1^\hbar$. The explicit computation can be found in proposition \ref{proof: mirror_Seidel}.
\end{proof}

Similarly, the noncommutative deformation space of $S_2$ is $\cA_2^\hbar = \Lambda \langle w_2,z_2,y_2\rangle/\langle\partial \Phi_2\rangle$, where $\Phi_2 = z_2y_2w_2-T^{-3\hbar}y_2z_2w_2$, and that of $S_3$ is $\cA_3^\hbar = \Lambda \langle w_3,x_3,z_3 \rangle/\langle \partial \Phi_3 \rangle$, where $\Phi_3 = x_3z_3w_3-T^{-3\hbar}z_3x_3w_3$.  Note that the noncommutative deformation parameter for $S_i$ is $T^{-3\hbar}$ rather than $T^{-\hbar}$.

We would like to construct an algebroid stack with charts being $\cA_i^\hbar$'s using Floer theory.  However, the three Seidel Lagrangians do not intersect with each other, and there is simply no isomorphism between them!

Here is the key idea.  We also include the nc deformation space $\A^\hbar$ of $\mathbb{L}$ as a chart and denote it by $\A_0^\hbar$.  (In actual computation of the mirror functor, we take $\mathbb{L}_0$ to be a Hamiltonian deformation of $\mathbb{L}$ by a Morse function.)  $\mathbb{L}_0$ serves as a `middle agent' that intersects with all the three Seidel Lagrangians $S_i$.  Note that $\A_0^\hbar$ is a quiver algebra with three vertices, while $\cA_i^\hbar, i=1,2,3$ are quiver algebras with a single vertex.  To glue them together, we need to employ the concept of a quiver stack defined in Section \ref{section:modified algebroid}. 

We take the collection of Lagrangians $\cL:=\{\bL_0,S_1,S_2,S_3\}$.
Then we solve for isomorphisms between $(\mathbb{L}_0,\bb_0)$ and $(S_i,\bb_i)$.  Solutions exist once we make suitable localizations for the deformation space $\A_0^\hbar$ of $\mathbb{L}_0$.

\begin{theorem}\label{prop:isomorphism_KP2}
	There exist preisomorphism pairs between $(\bL_0,\bb_0)$ and $(S_i,\bb_i),i=1,2,3$:
	$$\alpha_i \in CF_{\A_0^\hbar(U_{0i})\otimes \cA_i^\hbar}((\bL_0,\bb_0),(S_i,\bb_i)),\beta_i \in CF_{\cA_i^\hbar \otimes \A_0^\hbar(U_{0i})}((S_i,\bb_i),(\bL_0,\bb_0))$$
	and a quiver stack $\hat{\cY}$, whose charts are $\A_0^\hbar$ and $\cA_i^\hbar,i=1,2,3$, that solves the isomorphism equations for $(\alpha_i,\beta_i)$ over the Novikov field $\Lambda$:
	\begin{align*}
		m_{1,\hat{\cY}}^{\bb_0,\bb_i}(\alpha_i) =& 0, m_{1,\hat{\cY}}^{\bb_i,\bb_0}(\beta_i) = 0;\\
		m_{2,\hat{\cY}}^{\bb_0,\bb_i,\bb_0}(\alpha_i,\beta_i) =& 1_\bL, m_{2,\hat{\cY}}^{\bb_i,\bb_0,\bb_i}(\beta_i,\alpha_i) = 1_{S_i}.
	\end{align*}
	In above, $\A_0^\hbar(U_{0i})$ is the localization of $\A_0^\hbar$ at the set of arrows
	$\{a_1,a_3\},\{c_1,c_3\},\{b_1,b_3\}$ for $i=1,2,3$ respectively.  Moreover, $\bb_i$ is restricted to the subset 
	$$\{\val(w_i)>B\} \subset \Lambda^3$$ 
	for $i=1,2,3$ and $\bb_0$ is restricted to the subset 
	$$\left\{\val (b_1) > \val (a_1) + \frac{B}{2} + \hbar, \val (c_1) > \val (a_1) + \frac{B}{2}\right\}$$ 
	in order to define $G_{03}$ and $G_{30}$.  The cases for $G_{0i}$ and $G_{i0}$, $i=1,2$, are obtained by cyclic permutation.
\end{theorem}

\begin{proof}
	$$\alpha_3 = -{Q}^{2,3},\beta_3 =(T^{-B}1\otimes b_3^{-1}b_1^{-1})\overline{P^{3,3}},$$
	where $B= A_{112345(5)'}$. The notation for the area term can be found in Appendix \ref{subsection:KP2_area}.
	
	Similarly, we define preisomorphism pairs
	$$
	\begin{cases*}
		(\alpha_2,\beta_2) = \left(-{Q}^{2,2},(T^{-B}1\otimes c_3^{-1}c_1^{-1})\overline{P^{3,2}}\right)\\  (\alpha_1,\beta_1) = \left(-{Q}^{2,1},(T^{-B}1\otimes a_3^{-1}a_1^{-1})\overline{P^{3,1}}\right) .\\ 
	\end{cases*} 
	$$	
	
	The quiver stack $\hat{\cY}$ obtained as a solution is explicitly defined by the following data:
	\begin{enumerate}
		\item The underlying topological space is the polyhedral set $P$ of $K_{\bP^2}$, see Figure \ref{fig:ncKP2} for the projection of $P$ onto a plane.  The open sets $\emptyset$, $U_0=P$, $U_i$ for $i=1,2,3$, which are the complements of the $i$-th facet corresponding to the extremal rays of the fan, form a base of its topology.  Here $U_1$ corresponds to the facet on the left in Figure \ref{fig:ncKP2}, and the remaining open sets $U_2, U_3$ are labeled in the clockwise order.
		\item $\hat{\cY}$ associates $U_0=P$ to a presheaf of quiver algebras $\A_0^\hbar$ and $U_i$ to $\cA_i^\hbar$ for $i=1,2,3$ as in Section \ref{section:modified algebroid}. More precisely, $\A_0^\hbar(U_i)$ is the localization of $\A_0^\hbar$ at the set of variables $\{a_1,a_3\}, \{c_1,c_3\}, \{b_1,b_3\}$ for $i=1,2,3$ respectively.  $\A_0^\hbar(U_{ij})$ ($i\not=j$) and $\A_0^\hbar(U_{123})$ are the localizations of the union of corresponding sets of variables.  $\cA_1^\hbar(U_{12})=\cA_1^\hbar[x_1^{-1}], \cA_1^\hbar(U_{13})=\cA_1^\hbar[y_1^{-1}], \cA_1^\hbar(U_{123})=\cA_1^\hbar[x_1^{-1},y_1^{-1}]$.  Similarly, the sheaves over $U_2$ and $U_3$ are defined by the cyclic permutation on $(1,2,3)$ and \eqref{transformation_rule}. 
		
		Indeed, one can check that the presheaves are sheaf of quiver algebras. We will postpone the proof to Lemma \ref{lem:sheaf}. 
		\item  The transition representations $G_{0i}:\cA^\hbar_{i,0i}\rightarrow \A_{0,0i}^\hbar$ for $i=1,2,3$ are defined by
		\begin{equation}\label{eq:Seidel_Lag}
			\begin{cases}
				x_1 \mapsto T^{-\frac{B}{2}+\hbar}a_2^{-1}c_2 \\
				y_1 \mapsto T^{-\frac{B}{2}-\hbar}b_1a_1^{-1} \\
				w_1 \mapsto T^{B}a_1a_3a_2
			\end{cases}
			\begin{cases}
				y_2 \mapsto T^{-\frac{B}{2}+\hbar}c_2^{-1}b_2\\
				z_2 \mapsto T^{-\frac{B}{2}-\hbar}a_1c^{-1}_1\\
				w_2 \mapsto T^{B}c_1c_3c_2
			\end{cases}
			\begin{cases}
				z_3 \mapsto T^{-\frac{B}{2}+\hbar}b_2^{-1}a_2\\
				x_3 \mapsto T^{-\frac{B}{2}-\hbar}c_1b^{-1}_1\\
				w_3 \mapsto T^{B}b_1b_3b_2.
			\end{cases},
		\end{equation}
		\item The transition representation $G_{30}: \cA_{0,03}^\hbar \xrightarrow{} \cA^\hbar_{3,03}$ is defined by 
		\begin{equation}
			\begin{cases}
				e_1 \mapsto 1\\
				a_1 \mapsto T^{\frac{B}{2}}z_3 \\
				b_1^{-1} \mapsto 1\\
				b_1 \mapsto 1\\
				c_1 \mapsto T^{\frac{B}{2}+ \hbar}x_3
			\end{cases}
			\begin{cases}
				e_2 \mapsto 1\\
				a_2 \mapsto T^{-\hbar- \frac{B}{2}}w_3z_3\\
				b_2 \mapsto T^{-B}w_3\\
				c_2 \mapsto T^{2\hbar-\frac{B}{2}}w_3x_3
			\end{cases}
			\begin{cases}
				e_3 \mapsto 1\\
				a_3 \mapsto T^{\frac{B}{2}+\hbar}z_3\\
				b_3^{-1} \mapsto  1\\
				b_3 \mapsto 1\\
				c_3 \mapsto T^{\frac{B}{2}}x_3
			\end{cases}.
		\end{equation}
		$G_{i0}$ for $i=1,2$ are defined similarly using the cyclic symmetry Equation \ref{transformation_rule}.
		\item The gerbe terms at vertices of $Q_0$ are defined as follows.  $c_{0i0}(v_2)  = e_2$ for all $i=1,2,3$; $c_{030}(v_3)=b_1 b_3,c_{030}(v_1) = b_1, c_{020}(v_3)=c_1c_3,c_{020}(v_1) = c_1,c_{010}(v_3)=a_1a_3,$ $c_{010}(v_1) = a_1$.  The gerbe terms for $Q_i$, $i=1,2,3$ are trivial.
	\end{enumerate}
	
	The cocycle condition $G_{0i}\circ G_{i0}(a)=c_{0i0}(h_a) \cdot G_{00}(a) \cdot c_{0i0}^{-1}(t_a)$ and $c_{ijk}(G_{kl}(v)) c_{ikl}(v) = G_{ij}(c_{jkl}(v))c_{ijl}(v)$ can be verified explicitly for any $i,j,k,l$ and paths $a$. For example, $G_{03}\circ G_{30}(a_1)= G_{03}(T^{\frac{B}{2}}z_3)= a_1 b_1^{-1},$ while $c_{030}(h_{a_1}) \cdot G_{00}(a_1) \cdot c_{030}^{-1}(t_{a_1})= c_{030}(v_2)\cdot a_1 \cdot c_{030}^{-1}(v_1)= a_1 b_1^{-1}=G_{03}\circ G_{30}(a_1).$ Similarly, we obtain the cocycle conditions for the remaining $i,j,k,l$ and paths $a$ by explicit computations. 
	
	Furthermore, we can solve the isomorphism equations for  $(\alpha_i,\beta_i)$ over the quiver stack explicitly. More precisely, we get
	\begin{align*}
		m_{2,\hat{\cY}}^{\bb_0,\bb_3,\bb_0}(\alpha_3,\beta_3)
		=&(b_3b_3^{-1}b_1^{-1}\cdot c_{030}(v_2)\cdot b_1)1_{L_1}+(b_1b_3b_3^{-1}b_1^{-1}\cdot c_{030}(v_2))1_{L_2}+(b_3^{-1}b_1^{-1}\cdot c_{030}(v_2) \cdot b_1b_3)1_{L_3}\\
		= & \sum_{i=1}^3 e_i 1_{L_i}= (e_1+e_2+e_3)1_\bL = 1_\bL\\
	\end{align*} Besides, we obtain
	\begin{align*}
		m_{1,\hat{\cY}}^{\bb_0,\bb_3}(\alpha_3)=& (w_3\otimes 1\otimes e_2 -T^B 1\otimes e_2\otimes b_1\otimes b_3\otimes b_2) P^{2,3}+ (-1\otimes e_2 \otimes a_1+T^{\frac{B}{2}}z_3\otimes 1\otimes e_2 \otimes b_1)P^{1,3}_1\\ &+(- 1\otimes e_2\otimes c_1 +T^{\frac{B}{2}+\hbar}x_3\otimes 1\otimes e_2\otimes b_1 )P^{1,3}_2
	\end{align*}
	Using the transition representations \ref{eq:Seidel_Lag}, 
	\begin{align*}
		m_{1,\hat{\cY}}^{\bb_0,\bb_3}(\alpha_3)=& (G_{03}(w_3)e_2- T^B e_2b_1b_3b_2)P^{2,3}+ (-e_2a_1+ T^{\frac{B}{2}}G_{03}(z_3)e_2b_1)P^{1,3}_1\\
		&+(-e_2c_1+ T^{\frac{B}{2}+\hbar}G_{03}(x_3)e_2b_1)P^{1,3}_2= 0.
	\end{align*}
	
	The computations of the remaining isomorphism equations are similar. The details of computations can be found in Appendix \ref{appendix:KP2_iso}.
\end{proof}

\begin{lemma}\label{lem:sheaf}
	The presheaf $\A_0^\hbar$ (resp. $\cA_i^\hbar$) is a sheaf of quiver algebra over $P$ (resp. $U_i$).
\end{lemma}
\begin{proof}
	One can check that this is a sheaf following the idea in Remark \ref{rem:sheaf}. Here we check the sheaf condition by explicit calculations.
	
	First, we show $\cA_i^\hbar$ is a sheaf of quiver algebra over $U_i$. This is because the localized set doesn't contain any zero divisors, and if the local sections agree on the overlap, using the commutative relations, one may notice that each term should have positive degree. Hence, they come from the global section. 
	
	One can check that $\A_0^\hbar$ is also a sheaf by direct calculations. For example, let's look at the following complex:
	$$0 \to \A_0^\hbar(U_1 \cup U_2)=\A_0^\hbar \to \A_0^\hbar(U_1)\oplus \A_0^\hbar(U_2)= \A_0^\hbar(\{a_1,a_3\}^{-1}) \oplus \A_0^\hbar(\{c_1,c_3\}^{-1}) \to \A_0^\hbar(U_{12}). $$
	The first map is injective, since $a_1, c_1$ (resp. $a_3,c_3$) has no common torsion elements in $e_1 \cdot  \A_0^\hbar$ or $ \A_0^\hbar \cdot e_2$ (resp. $e_3 \cdot  \A_0^\hbar$ or $ \A_0^\hbar \cdot e_1$). 
	
	Let $(x,y)$ be elements in $\A_0^\hbar(\{a_1,a_3\}^{-1}) \oplus \A_0^\hbar(\{c_1,c_3\}^{-1})$ such that $x-y=0$ in $\A_0^\hbar(U_{12})$. Using the commutative relations, $x$ can be written as $x=f_1+ f_2 a_1^{-1}+ f_3 a_3^{-1}$ and $y=g_1+g_2c_1^{-1}+g_3c_3^{-1}$ for some $f_1$, $g_1\in \A_0^\hbar$, $f_2,f_3 \in \A_0^\hbar(\{a_1,a_3\}^{-1})$ and $g_2,g_3 \in \A_0^\hbar(\{c_1,c_3\}^{-1})$. 
	
	According to the idempotent (vertex) of the quiver algebra, we have $f_1'+ f_2 a_1^{-1}= g_1'+ g_2c_1^{-1}$ and $f_1''+ f_3 a_3^{-1}= g_1''+g_3c_3^{-1}, $ where $f_1=f_1'+f_1''$ and $g_1=g_1'+g_1''$. Therefore, $f_1'a_1+f_2-g_1'a_1=g_2c_1^{-1}a_1$. However, the LHS $f_1'a_1+f_2-g_1'a_1$ doesn't contain the factor $c_1^{-1}$ or $c_3^{-1}.$ Thus, $g_2c_1^{-1}a_1$ can be simplified and it's an element in $\cA_0^\hbar$. Thus, $g_2c_1^{-1} \in \cA_0^\hbar$. Similarly for $g_3c_3^{-1}.$ Hence, $y=g_1+g_2c_1^{-1}+g_3c_3^{-1}$ is an element in $\A_0^\hbar = \A_0^\hbar(U_1\cup U_2).$ Use the same method, one can check that $\A_0^\hbar$ is a sheaf.
\end{proof}	

The relations among $\cA_i^\hbar$ for $i=1,2,3$ can be found by extending the charts and the transitions by allowing the variables to be in $\Lambda$ (instead of $\Lambda_+$). If we make such extensions of charts, we can drop the chart $\cA_{0}^\hbar$ and still have a connected algebroid stack $\cY$.

\begin{corollary} \label{thm: mirror KP2}
	There exists an algebroid stack $\cY$ over $\Lambda$ consisting of the following:
	\begin{enumerate}
		\item An open cover $\{U_i\}$ of polyhedral set $P$ of  $K_{\bP^2}$ for $i= 1, 2, 3$.
		\item The collection of nc deformation spaces of Seidel Lagrangians $S_i$, $\cA_i^\hbar$ over $U_i$ with coefficients $\Lambda$. 
		\item Sheaves of representations $G_{ij}:\cA_{j}^\hbar|_{U_{ij}}\xrightarrow{} \cA_{i}^\hbar|_{U_{ij}}$ satisfying the cocycle condition with trivial gerbe terms $c_{ijk}=1$ for $i,j,k \in \{1,2,3\}$.
	\end{enumerate}
\end{corollary}

\begin{proof}
	We have the charts $\cA_{i}^\hbar$ for $i=1,2,3$ from Theorem \ref{prop:isomorphism_KP2}, and they are now extended over $\Lambda$.   We simply define $G_{ij}$ by the composition $G_{i0}\circ G_{0j}$.  The localized variables are 	$S_{0,0ij} = S_{0,0i}\cup S_{0,0j}, S_{1,012}=\{x_1\},S_{2,012} = \{z_2\},S_{1,013} =\{y_1\}, S_{3,013} = \{z_3\}, S_{2,023} = \{y_2\},S_{3,023} = \{x_3\}$ and $S_{i_0,i_0\cdots i_p} = \cup_{k\not=0} S_{i_0,i_0i_k}$ for $i_0,\cdots,i_p\in \{123\}$.
	
	We check that $\textrm{Im} \, G_{0j,0ij} = \textrm{Im} \, G_{0i,0ij}$ for $i,j = 1,2,3$ after we have extended to $\Lambda$.  
	We only show the case for $(i,j)=(1,2)$ and other cases are similar.  By direct computations,
	\begin{align*}
		G_{01}(x_1)=&T^{-B-\hbar}G_{02}(z_2^{-1}), \\
		G_{01}(y_1) =& T^{-\frac{B}{2}-\hbar}b_1a^{-1} = T^{-\frac{B}{2}-\hbar}b_1c_1^{-1}c_1 a_1^{-1} = T^{-\frac{B}{2}-2\hbar}G_{02}(y_2z_2^{-1}), \\
		G_{01}(w_1x_1^3) 
		= &T^{-\frac{B}{2}}a_1a_3a_2(c_1a^{-1}_1)^3
		= T^{-\frac{B}{2}-3\hbar}c_1a_3a_2(c_1a^{-1}_1)^2\\
		= & T^{-\frac{B}{2}-6\hbar}c_1c_3c_2 = T^{-\frac{B}{2}-6\hbar}G_{02}(w_2).
	\end{align*}
	Result follows.  (We remark that the statement is not true over $\Lambda_+$.)
	
	Explicitly, $G_{21}, G_{32}, G_{13}$ are given by
	\begin{equation}\label{eq:relation_between_Seidel}
		\begin{cases}
			x_1 \mapsto T^{-B-\hbar}z_2^{-1}\\
			y_1 \mapsto T^{-\frac{B}{2}-2 \hbar}y_2 z_2^{-1}\\
			w_1 \mapsto T^{\frac{3B}{2}+9\hbar}w_2 z_2^{3}
		\end{cases}
		\begin{cases}
			y_2 \mapsto T^{-B-\hbar}x_3^{-1}\\
			z_2 \mapsto T^{-\frac{B}{2}-2\hbar}z_3 x_3^{-1}\\
			w_2 \mapsto T^{\frac{3B}{2}+9\hbar} w_3 x_3^{3}
		\end{cases}
		\begin{cases}
			z_3 \mapsto T^{-B- \hbar}y_1^{-1}\\
			x_3 \mapsto T^{-\frac{B}{2}-2 \hbar}x_1 y_1^{-1}\\
			w_3 \mapsto T^{\frac{3B}{2}+9\hbar} w_1 y_1^{3}.
		\end{cases}
	\end{equation} 
	The cocycle conditions $G_{ij}\circ G_{jk}= G_{ik}$ for trivial gerbe terms $c_{ijk}=1$ can be directly verified.
\end{proof}

The above gluing equations \eqref{eq:relation_between_Seidel} involve $T^{-B-\hbar} \not\in \Lambda_+$, which manifests the fact that the Seidel Lagrangians $S_i$ do not intersect with each other.  

We can also obtain an algebroid stack $\cY(\C)$ over $\C$ by changing the charts to $\C^3 \subset \Lambda^3$, and specifying the formal parameter $T$ to be $e \in \C$.  Since the transitions in \eqref{eq:relation_between_Seidel} only involve monomials, there is no convergence issue over $\C$. Hence, the transitions define a noncommutative $K_{\bP^2}$ over $\C.$

\begin{remark}
	
	An interesting degenerate phenomenon occurs if we restrict to the zero section $\bP^2_\hbar$. To be more precise, we set $w_i=0$ to obtain noncommutative $\bP^2_\hbar$. Let $\tilde{z}_2:= T^{\frac{B}{4}}z_2,$ $\tilde{x}_3:= T^{\frac{B}{4}}x_3,$ $\tilde{y}_1:= T^{\frac{B}{4}}y_1.$ With these new variables, we have the following transition maps:
	\begin{equation}
		\begin{cases}
			x_1 \mapsto T^{-\frac{3B}{4}-\hbar}\tilde{z}_2^{-1}\\
			\tilde{y}_1 \mapsto T^{-2 \hbar}y_2 \tilde{z}_2^{-1}\\
		\end{cases}
		\begin{cases}
			y_2 \mapsto T^{-\frac{3B}{4}-\hbar}\tilde{x}_3^{-1}\\
			\tilde{z}_2 \mapsto T^{-2\hbar}z_3 \tilde{x}_3^{-1}\\
		\end{cases}
		\begin{cases}
			z_3 \mapsto T^{-\frac{3B}{4}- \hbar}\tilde{y}_1^{-1}\\
			\tilde{x}_3 \mapsto T^{-2 \hbar}x_1 \tilde{y}_1^{-1}.\\
		\end{cases}
	\end{equation} 
	If we set $B \to +\infty$ and fix $\hbar$ (that is, the cylinder area of two adjacent Seidel Lagrangians tends to infinity, see Figure \ref{fig:KP2_area}), the first row vanishes. The noncommutative $\bP^2_\hbar$ degenerates to the union of three noncommutative $\mathbb{F}_1^\hbar$. See Figure \ref{fig:P2_deg}.
\end{remark}

\begin{figure}[htb!]
	\centering
	\includegraphics[scale=0.5]{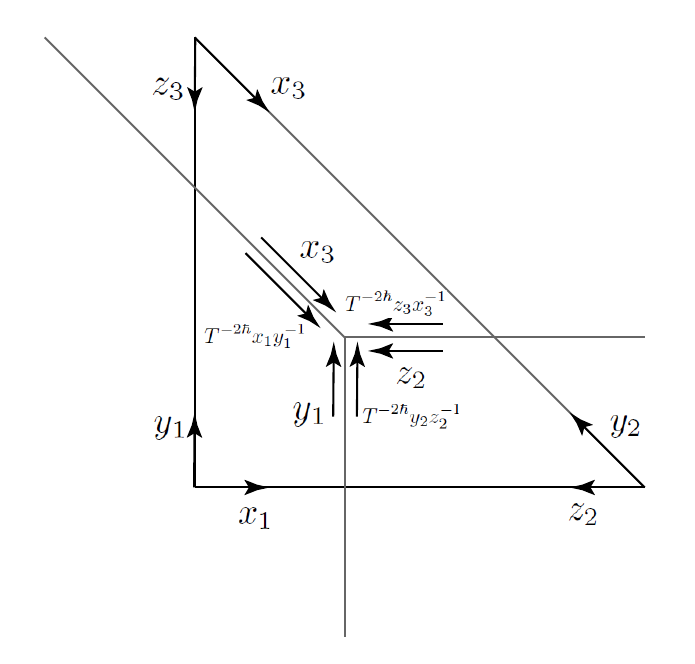}
	\caption{Degeneration of $\bP^2_\hbar$.}
	\label{fig:P2_deg}
\end{figure}

From the general theory in the previous section, we have an $A_\infty$ functor $\Fuk(M) \to \Tw(\hat{\cY})$.  Given an object $L \in \Fuk(M)$, if the corresponding twisted complex $\cF^\cL(L)$ over $\Lambda_+$ still converges over $\Lambda$, then we have a corresponding object $\cF_\Lambda^\cL(L)$ in $\Tw(\cY)$.  Furthermore, if the transition maps in  $\cF_\Lambda^\cL(L)$ converge when we specify $T=e$, then there is a corresponding object $\cF_\C^\cL(L)$ in $\Tw(\cY(\C))$.

The above consideration also holds if we replace a single object $L$ by a collection of objects $\{L_0,\ldots,L_k\}$ and impose the convergence assumption on the morphisms for the corresponding twisted complexes.  In such a situation, we obtain an $A_\infty$ functor from the subcategory generated by  $\{L_0,\ldots,L_k\}$ to $\Tw(\cY(\C))$.

\subsection{Construction of the Universal Bundle}

Recall that we have the collection of Lagrangians $\cL=\{\bL_0,S_1,S_2,S_3\}$ and $\bL$, where $\bL=L_1'\cup L_2' \cup L_3'$ and $\bL_0=L_1\cup L_2 \cup L_3$ just differ by a Hamiltonian deformation.  The nc deformation space of $\bL$ is $\A^\hbar$ whose elements are denoted by $\bb'$.  The intersection points between $L_i,L_j'$ are denoted by $\hat{P}^{i,j},\hat{Q}^{i,j}$.

\begin{theorem}
	The twisted complex $\U :=\cF^\cL((\bL,\bb'))$ converges over $\C$ and defines an object $\U_{\cY(\C)}$ in $\Tw(\cY(\C))$.  Similarly, $\cF^\cL(L_k')$ defines an object in $\Tw(\cY(\C))$ for $k=1,2,3$, and they are denoted by $\cF_{\cY(\C)}^\cL(L_k')$.   Furthermore, the functor $\cF^{\U_{\cY(\C)}^*}:=\Hom_{\A^\hbar}(\U_{\cY(\C)}^*,-):\mathrm{dg-mod}(\A^\hbar) \to \Tw(\cY(\C))$ sends $\cF^\bL(L_k')$ to $\cF_{\cY(\C)}^\cL(L_k')$ for $k=1,2,3$, where $\cF_{\cY(\C)}^{\cL}(L_2')\cong\mathcal{O}_{\mathbb{P}^2}$, $\cF_{\cY(\C)}^\cL(L_3')\cong\mathcal{O}_{\mathbb{P}^2}(-1)$. 
\end{theorem}

\begin{figure}[htb!]
	\centering
	\includegraphics[scale = 0.47, trim={18cm 8cm 21cm 0cm},clip]{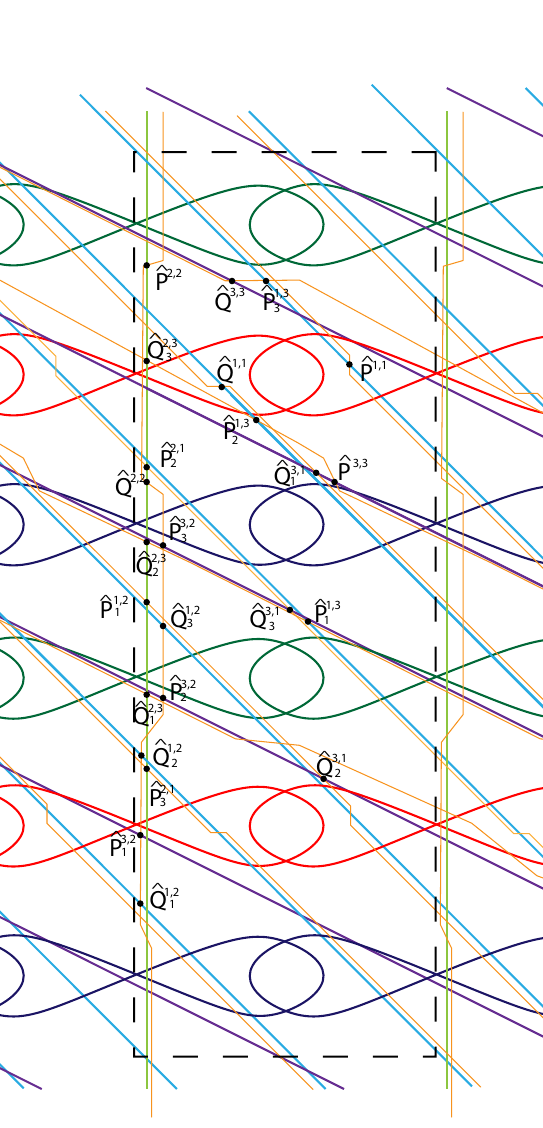}
	\caption{Deformed Lagrangian $\bL$}
	\label{fig:KP2_deformed_L}
\end{figure}

We compute $\U$ over each chart as follows.
Over the chart $U_i$, we have a complex 
$$\U_i:= (E_i, a_i):= (\tilde{\A}^\hbar \otimes \cA_i^\hbar \otimes \CF((S_i,\bb_i),(\bL,\bb')),(-1)^{\deg(\cdot)} m_{1,\cY}^{\bb_i,\bb'}(\cdot)).$$
For $i=1,2,3$,
$$	\centerline{
	\xymatrix{
		0 \ar[r] & Q^{2,i}  \ar[r]^-{a_i^0} & P^{2,i}\oplus P^{1,i}_1\oplus P^{1,i}_2  \ar[r]^-{a_i^1} & Q^{3,i}\oplus Q^{1,i}_2\oplus Q^{1,i}_1 \ar[r]^-{a_i^2} & P^{3,i}  \ar[r] & 0
	},
}
$$
where the horizontal arrows are defined in Appendix \ref{def:horizontal_KP2}. 
We also have the complex $\U_0$, which takes the form
$$ \centerline{\xymatrix{0 \ar[r] & \bigoplus_{j=1,2,3} \hat{Q}^{j,j} \ar[r] & \bigoplus_{j,k=1,2,3}\hat{P}_k^{j+1,j} \ar[r] & \bigoplus_{j,k=1,2,3}\hat{Q}^{j-1,j}_k \ar[r] & \bigoplus_{j=1,2,3}\hat{P}^{j,j} \ar[r] & 0}}. $$



The transitions over $U_{0i}$ are chain maps between $\cat{F}^{\bL_0}(\bL)$ and $\cat{F}^{S_i}(\bL)$. This gives us the following commutative diagram where the vertical arrows are defined over $\cA^\hbar_{0,0i}$:
$$
\centerline{
	\xymatrix{
		0 \ar[r] & Q^{2,i} \ar[d] \ar[r] & P^{2,i}\oplus P^{1,i}_1\oplus P^{1,i}_2 \ar[d] \ar[r] & Q^{3,i}\oplus Q^{1,i}_2\oplus Q^{1,i}_1 \ar[d] \ar[r] & P^{3,i} \ar[d]\ar[r] & 0\\
		0 \ar[r] & \bigoplus_{j=1,2,3} \hat{Q}^{j,j} \ar[r]\ar[u] & \bigoplus_{j,k=1,2,3}\hat{P}_k^{j+1,j} \ar[u] \ar[r] & \bigoplus_{j,k=1,2,3}\hat{Q}^{j-1,j}_k \ar[u]  \ar[r]& \bigoplus_{j=1,2,3} \hat{P}^{j,j} \ar[u] \ar[r] & 0\\
	}
}$$
The vertical arrows are defined by $a_{0i}^{1,0}=m_{2,\hat{\cY}}^{\bb_0,\bb_i,\bb'}(\alpha_i,\cdot),a_{i0}^{1,0}=m_{2,\hat{\cY}}^{\bb_i,\bb_0,\bb'}(\beta_i, \cdot)$. Moreover, we have the non-trivial homotopy terms $a_{0j0}^{2,-1}=m_3^{\bb_0,\bb_j,\bb_0,\bb'}(\alpha_j,\beta_j,-)$ for $j=1,2,3$.

\begin{proof}
	We would like to extend from $\Lambda_+$ to $\Lambda$ and eliminate the middle chart $\cA_0$, so that we obtain a twisted complex over $\cY$ (instead of $\hat{\cY}$).  Furthermore, we restrict to $\C^3 \subset \Lambda_3$ and specify $T=e$ to obtain an object over $\cY(\C)$.
	
	The key point of extension is convergence.  In Section \ref{appendix:KP2_ai} and \ref{appendix:KP2_aijk}, we have found all the polygons that contribute to $a_{j0}^{1,0}$, $a_{0j}^{1,0}$ and $a_{0j0}^{2,-1}$ for $j=1,2,3$.  Since there are just finitely many of them, these expressions are Laurent polynomials and have no convergence issue.
	
	After we have extended over $\Lambda$, the charts $\cA_i(\Lambda)$ for $i=1,2,3$ have common intersections and the transition maps are given by
	$$a_{ij}^{1,0} = m_{2,\hat{\cY}}^{\bb_i,\bb_0,\bb'}(\beta_i, m_{2,\hat{\cY}}^{\bb_0,\bb_j,\bb'}(\alpha_j,\cdot)):E_{j,ij}\rightarrow E_{i,ij}$$ 
	for $i\not= j$ and $a_{ii}^{1,0} = \Id_i:E_{i}\rightarrow E_{i}. $ 
	They take the form
	$$
	\centerline{
		\xymatrix{
			0 \ar[r] & Q^{2,j} \ar[d]_-{a_{ij}} \ar[r]^-{a_j^0} & P^{2,j}\oplus P^{1,j}_1\oplus P^{1,j}_2 \ar[d]_-{a_{ij}} \ar[r]^-{a_j^1} & Q^{3,j}\oplus Q^{1,j}_2\oplus Q^{1,j}_1 \ar[d]_-{a_{ij}} \ar[r]^-{a_j^2} & P^{3,j} \ar[d]_-{a_{ij}} \ar[r] & 0\\
			0 \ar[r] & Q^{2,i} \ar[r]^-{a_i^0}\ar[u]<-4pt>_-{a_{ji}} & P^{2,i}\oplus P^{1,i}_1\oplus P^{1,i}_2 \ar[u]<-4pt>_-{a_{ji}} \ar[r]^-{a_i^1} & Q^{3,i}\oplus Q^{1,i}_2\oplus Q^{1,i}_1 \ar[u]<-4pt>_-{a_{ji}}  \ar[r]^-{a_2^2} & P^{3,i} \ar[u]<-4pt>_-{a_{ji}} \ar[r] & 0,\\
		}
	}
	$$
	where $Q^{k,i}, P^{k,i}$ are generators in $\CF(S_i, L'_k)$.  $a_{23}^{1,0},a_{32}^{1,0}$ are given in Appendix \ref{def:vertical_KP2}. Other $a_{ij}^{1,0}$ can be obtained via the transformation rule \ref{transformation_rule}. 
	
	Besides, we have the homotopy terms $$a_{ijk}^{2,-1}:= m_{2,\hat{\cY}}^{\bb_i,\bb_0,\bb'}(\beta_i,m_3^{\bb_0,\bb_j,\bb_0,\bb'}(\alpha_j,\beta_j,m_{2,\hat{\cY}}^{\bb_0,\bb_k,\bb'}(\alpha_k,\cdot))):E_{k,ijk} \xrightarrow{} E_{i,ijk}$$ 
	for $i,j,k\in\{1,2,3\}$.  The computations of  $a_{321}^{2,-1}$ is given in Appendix \ref{def:homotopy_KP2}. Other $a_{ijk}^{2,-1}$ can be obtained similarly.  This defines a twisted complex over $\cY(\C)$.
	
	By direct computations, $\cF^\bL(L_k')$ equals to the Koszul resolution of the simple module at vertex $k$.  Then $\Hom_{\A^{\hbar}_0}(\U_{\cY(\C)}^*, \cF^\bL(L_k'))$ is obtained from $\U_{\cY(\C)}$ by dropping all the generators except those at vertex $k$.  That is, $\Hom_{\A^{\hbar}_0}(\U_{\cY(\C)}^*, \cF^\bL(L_k'))$ equals to the twisted complexes
	$$
	\xymatrix{
		Q^{2,j} \ar[d]_-{a_{ij}} \ar[r]^-{a_{j}^0} & P^{2,j} \ar[d]_-{a_{ij}} \\
		Q^{2,i} \ar[r]^-{a_{i}^0}\ar[u]<-4pt>_-{a_{ji}} & P^{2,i} \ar[u]<-4pt>_-{a_{ji}},
	}
	\xymatrix{
		P^{1,j}_1\oplus P^{1,j}_2 \ar[d]_-{a_{ij}} \ar[r]^-{a_j^1} & Q^{1,j}_2\oplus Q^{1,j}_1 \ar[d]_-{a_{ij}} \\
		P^{1,i}_1\oplus P^{1,i}_2 \ar[u]<-4pt>_-{a_{ji}} \ar[r]^-{a_i^1} & Q^{1,i}_2\oplus Q^{1,i}_1 \ar[u]<-4pt>_-{a_{ji}},\\
	}
	\xymatrix{
		Q^{3,j} \ar[d]_-{a_{ij}} \ar[r]^-{a_j^2} & P^{3,j} \ar[d]_-{a_{ij}}\\
		Q^{3,i} \ar[u]<-4pt>_-{a_{ji}}  \ar[r]^-{a_2^2} & P^{3,i} \ar[u]<-4pt>_-{a_{ji}}\\
	}
	$$
	for $k=2,1,3$ respectively, which are exactly $\cF^\cL_{\cY(\C)}(L_k)$.  They are explicitly computed in the appendix.  The first and third two-term complexes are resolutions of $\cO_{\bP^2}$ and $\cO_{\bP^2}(-1)$ respectively.

	
\end{proof}

\begin{appendices}
\section{Computations and Figures for mirror symmetry for nc local projective plane} \label{appendix}

\begin{subsection}{Notation of Area Terms}\label{subsection:KP2_area}
	
	The assignment of area is labeled in Figure \ref{fig:KP2_area}, where the green triangle is labelled by $A_1'$, the pink triangle is labeled by $A_1$ and the red one is labeled as $A_2$. In particular, we set $\hbar:= A_1-A_1'$. Then, \textcircled{1},\textcircled{2} are $A_1'-A_2'-A_4'$, $A_1-A_2-A_4$. 
	To simplify, we can set $A_2=A_2'=A_4=A_4'=A_3=0$, and $A_5=A_5'$.
	The area of any other non-labeled polygons can be obtained by symmetry of vertical translations.

	\begin{figure}[htb!]
		\centering
		\includegraphics[scale=0.9,trim={15.5cm 18cm 17.5cm 2cm},clip]{Formal_Figures/KP2_everything_2.pdf}
		\caption{The assignment of area of polygons}
		\label{fig:KP2_area}
	\end{figure}
	
	To shorten the expression of entries, we use the following abbreviation of area terms:
	\begin{itemize}
		\item $A_{j'} = A_j'$
		\item $A_I =\sum_j A_{i_j}$
		\item $A_{I'} = \sum_j A'_{i_j}$
		\item $A_{I(J)'} = \sum_k A_{i_k} + \sum_k A_{j_k}'$
	\end{itemize}
	In particular, to avoid counting, we prefer to denote $A_{i...i}$ by $k A_i$ for k repeated indices.
	
	A crucial thing is that solving the isomorphism equation will give $A_{112345(5)'}=A_{5(112345)'}$, that is, $2 A_1 = 2 A_1' + A_3'$ in the simplified setting.  Thus $A_3' = 2 \hbar$.

	Furthermore, we can simplify the expression by using the following variables: 
	\begin{itemize}
		\item $B = B_1= A_{112345(5)'}= 2A_1 + A_5+ A_5'$
		\item $B_2 = -A_{1345}+A_4'= -A_1 - A_5= -\frac{B}{2}$
		\item $B_3 = -A_{1345}'+A_4= -A_1'- A_3'- A_5'= -\frac{B}{2}- \hbar$
		\item $\Delta_i = A_i-A_i'$
	\end{itemize}
	Note that $B+4\hbar$ is the cylinder area bounded by two Seidel Lagrangians (See the right region in Figure \ref{fig:KP2_area}).
	
\end{subsection}

\begin{subsection}{Computation of Isomorphisms}\label{appendix:KP2_iso}
	
	In the following proof, we will show the proposition holds for $\alpha_3,\beta_3$. For other Seidel Lagrangians, $S_1$ and $S_2$, the computation is similar. 
	
	\begin{figure}[htb!]
		\centering
		\includegraphics[scale=0.8,trim={12.5cm 10cm 19.5cm 5cm},clip]{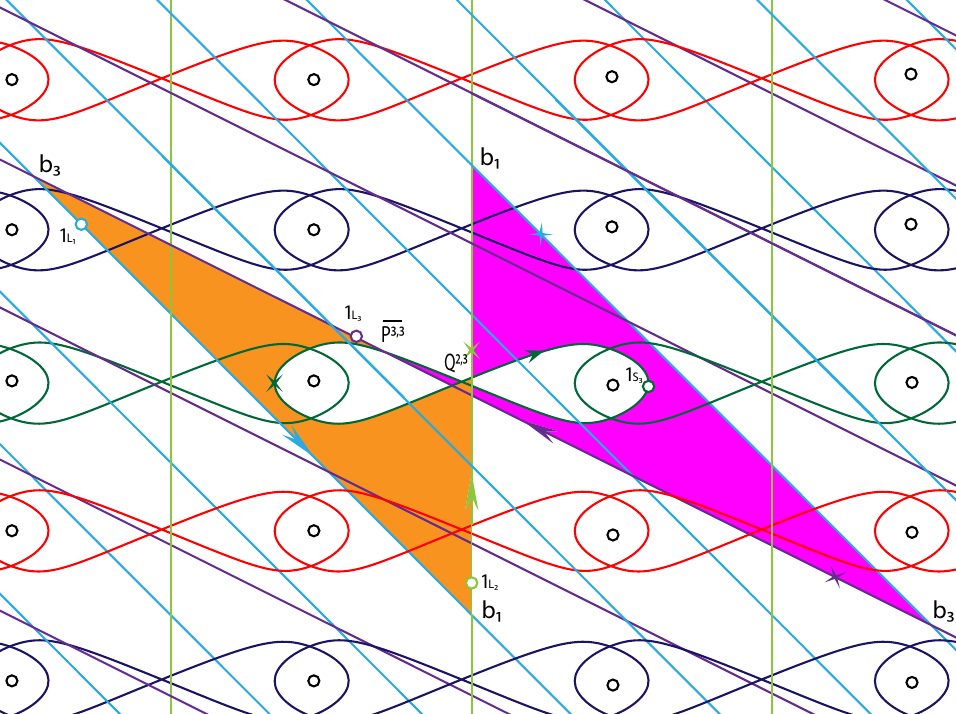}
		\caption{$m_{2,\hat{\cY}}^{\bb_3,\bb_0,\bb_3}(\alpha_3,\beta_3)$ and $m_{2,\hat{\cY}}^{\bb_0,\bb_3,\bb_0}(\beta_3,\alpha_3)$}
		\label{fig:KP2_m2}
	\end{figure}
	
	\begin{figure}[htb!]
		\centering
		\includegraphics[scale=0.8,trim={8.5cm 11cm 24cm 2cm},clip]{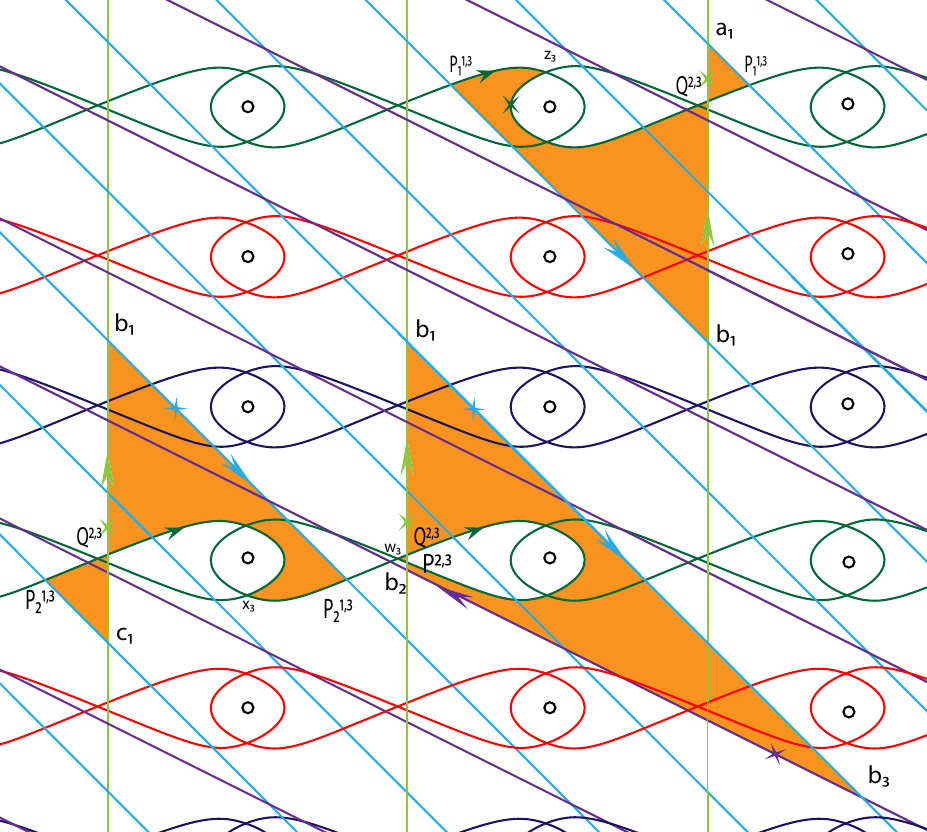}
		\caption{Polygons in $m_{1,\hat{\cY}}^{\bb,\bb_3}(Q^{2,3})$}
		\label{fig:KP2_m1beta}
	\end{figure}
	\begin{proof}[Proof of Theorem \ref{prop:isomorphism_KP2}]
		According to Figure \ref{fig:KP2_m2},
		$$m_{2,\hat{\cY}}^{\bb_3,\bb_0,\bb_3}(\beta_3,\alpha_3) = m_{4,\hat{\cY}}(T^{-B}b_1^{-1}b_3^{-1}\overline{P^{3,3}},b_3B_3,b_1B_1,-{Q^{2,3}}) = T^B (T^{-B}b_1b_3b_3^{-1}b_1^{-1})1_{S_3} = 1_{S_3},$$ where the reversed orientation along $\widearc{b_3b_1},\widearc{b_1{Q^{2,3}}}$ contributes $(-1)^2$ and spin structures along the boundary contribute $(-1)^3$ in the pink polygon. 
		
		In the orange polygon, the only clockwise edge is from $\overline{P^{3,3}}$ to ${Q^{2,3}}$, whose degrees are even. So, the only $(-1)$ comes from the spin structure on this edge. Together with the negative sign in $\alpha_3$, we have 
		\begin{align*}
			m_{2,\hat{\cY}}^{\bb_0,\bb_3,\bb_0}(\alpha_3,\beta_3)
			=&(b_3b_3^{-1}b_1^{-1}\cdot c_{030}(v_2)\cdot b_1)1_{L_1}+(b_1b_3b_3^{-1}b_1^{-1}\cdot c_{030}(v_2))1_{L_2}+(b_3^{-1}b_1^{-1}\cdot c_{030}(v_2) \cdot b_1b_3)1_{L_3}\\
			= & \sum_{i=1}^3 e_i 1_{L_i}= (e_1+e_2+e_3)1_\bL = 1_\bL\\
		\end{align*}
		where $c_{030}(v_2)=e_2$.
		
		Now, we need to check $	m_{1,\hat{\cY}}^{\bb_0,\bb_3}(\alpha_3) = 0$. In Figure \ref{fig:KP2_m1beta}, there are three pairs of polygons from $Q^{2,3}$ to $P^{2,3},P^{1,3}_1,P^{1,3}_2$. The leftmost one contributes to $m_{2,\hat{\cY}}(c_1 C_1, 1\otimes e_2 Q^{2,3}) = - 1\otimes e_2\otimes c_1 P_2^{1,3}$ and $m_{4,\hat{\cY}}(b_1B_1, 1\otimes e_2 Q^{2,3},x_3 X_3) = T^{-\frac{B}{2}-\hbar}x_3\otimes 1\otimes e_2\otimes b_1 P_2^{1,3}$. Similarly, we can compute other pairs of polygons.
		Their coefficients in $m_{1,\hat{\cY}}^{\bb_0,\bb_3}(\alpha_3)$ are
		\begin{equation*}
			\begin{cases*}
				(w_3\otimes 1\otimes e_2 -T^B 1\otimes e_2\otimes b_1\otimes b_3\otimes b_2)\\
				(-1\otimes e_2 \otimes a_1+T^{-\frac{B}{2}}z_3\otimes 1\otimes e_2 \otimes b_1)\\ 
				(- 1\otimes e_2\otimes c_1 +T^{-\frac{B}{2}-\hbar}x_3\otimes 1\otimes e_2\otimes b_1 )
			\end{cases*}
		\end{equation*}
		
		With the relation \ref{eq:Seidel_Lag}, they all vanish after we apply $\cM$ defined by Equation \ref{eq:mult}. For instance, the first sum above corresponds to
		$$G_{03}(w_3)c_{033}^{-1}(v_1)e_2-T^{B}G_{30}(1)c_{033}^{-1}(v_1)b_1b_3b_2= T^{B}b_1b_3b_2-T^{B}b_1b_3b_2 = 0.$$
		
		The computation of $m_{1,\hat{\cY}}^{\bb_3,\bb_0}(\beta_3) = 0$ is similar. We show all polygons involved in Figure \ref{fig:KP2_m1alpha}.
		
		\begin{figure}[htb!]
			\centering
			\includegraphics[scale=0.65,trim={8.5cm 11cm 19cm 7cm},clip]{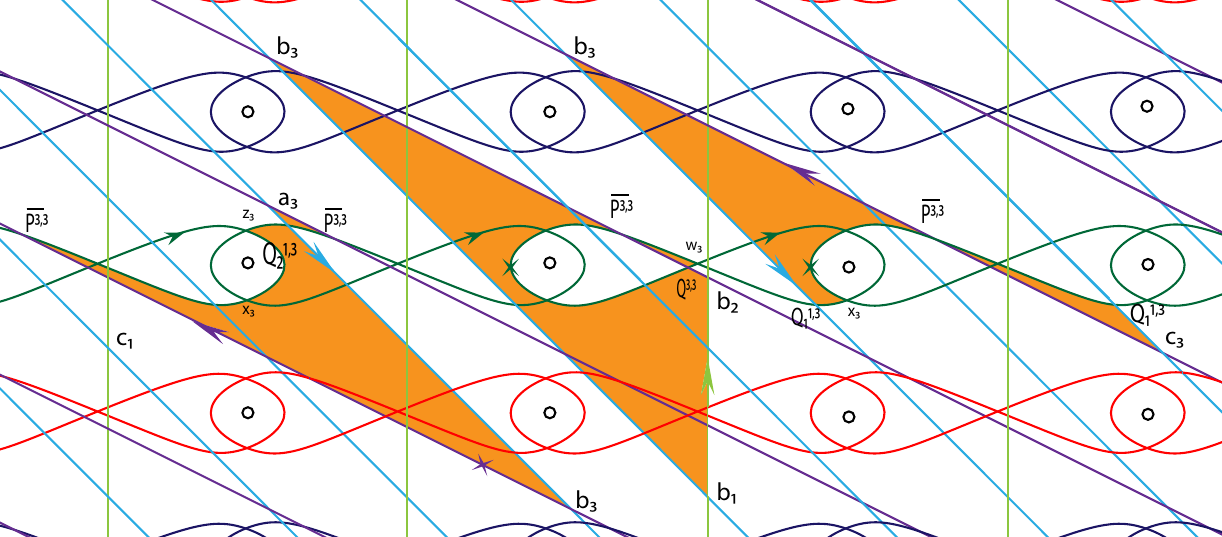}
			\caption{Polygons in $m_1^{b_3,b}(\overline{P^{3,3}})$}
			\label{fig:KP2_m1alpha}
		\end{figure}
	\end{proof}
	
\end{subsection}

\begin{figure}
	\caption{$S_1$ with spin structure and orientation, where the area of orange triangle is $A_3+A_1'-A_2'-A_4'=A_1'$ and the area of pink triangle is $ A_3'+A_1-A_2-A_4= A_1 + A_3'$}.
	\centering
	\includegraphics[scale=0.75,page=1,trim={14cm 13cm 18cm 8cm},clip]{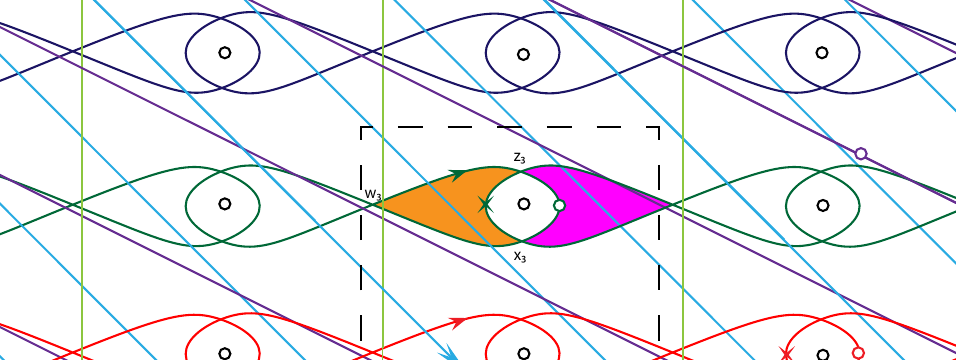}
	\label{fig_S1}
\end{figure}

\begin{prop}\label{proof: mirror_Seidel}
	Consider the reference Lagrangian $S_3$ with the given orientation, fundamental class and spin structure in Figure\ref{fig_S1}. With the space of odd-degree weakly unobstructed formal deformations $\bb_3 = w_3 W_3 + x_3 X_3+ z_3 Z_3$ of $S_3$,  noncommutative deformation space $\cA_3^\hbar = \Lambda_+ <w_3,x_3,z_3>/\partial \Phi$, where $\Phi = w_3x_3z_3-T^{-3\hbar}x_3w_3z_3$
\end{prop}

\begin{proof}
	
	Let $\bb_3 = w_3 W_3 + x_3 X_3+ z_3 Z_3$. There are only two polygons bounded by $S_1$, the shaded and unshaded polygons. (Notice that any unshaded region outside $S_1$ is not a polygon because there are other punctures outside this picture.) Hence, all non-zero terms in $m(e^b)$ comes from those two polygons. $m_2(x_3X_3,w_3W_3)$, $m_2(w_3W_3,z_3Z_3)$, and $m_2(z_3Z_3, x_3 X_3)$ correspond to the pink triangle, and $m_2(w_3W_3,x_3X_3), m_2(x_3X_3,z_3Z_3)$, and $m_2(z_3Z_3, w_3 W_3)$ correspond to the orange one. 
	
	Then, the coefficient of $\overline{Z_3}$ in $m_2(x_3X_3,w_3W_3)$ is $-T^{A_1+ A_3'} w_3x_3$ and the coefficient of $\bar{Z_3}$ in $m_2(w_3W_3,x_3X_3)$ is $T^{A_1'} x_3w_3$. Overall, the coefficient of $\overline{Z_3}$ in $m_2(b,b)$ is
	
	$-T^{A_1+ A_3'} w_3x_3+T^{A_1'} x_3w_3 =  -T^{A_1'}(T^{A_1 + A_3'-A_1'}w_3x_3-x_3w_3)= - T^A(T^{3\hbar}w_3x_3-x_3w_3)$, 
	since $A_1 -A_1'+ A_3' = \hbar + 2\hbar= 3\hbar$, where $A_3'=2\hbar$. 
	
	Similarly, we can obtain the coefficients of $\bar{W_3}$ and $\bar{X_3}$. Then, $\Phi' = \sum \frac{1}{2+1} <m_2(b,b),b> = T^A(T^{3\hbar}w_3x_3-x_3w_3)z_3 \in \cA_3/[\cA_3,\cA_3]$. After rescaling the spacetime superpotential, we have $\Phi = (w_3x_3-T^{-3\hbar}x_3w_3)z_3$
\end{proof}

\begin{subsection}{Polygons in $a_i$, $a_{0i},a_{i0}$}\label{appendix:KP2_ai}
	In this section, we show the polygons involved for $a_i$, $a_{0i}$ and $a_{i0}$. We first show the polygons involved for $a_3$. Other $a_i$'s are similar.
	
	\begin{figure}[H]
		\centering
		\includegraphics[scale=0.65,trim={14cm 11cm 12cm 7cm},clip]{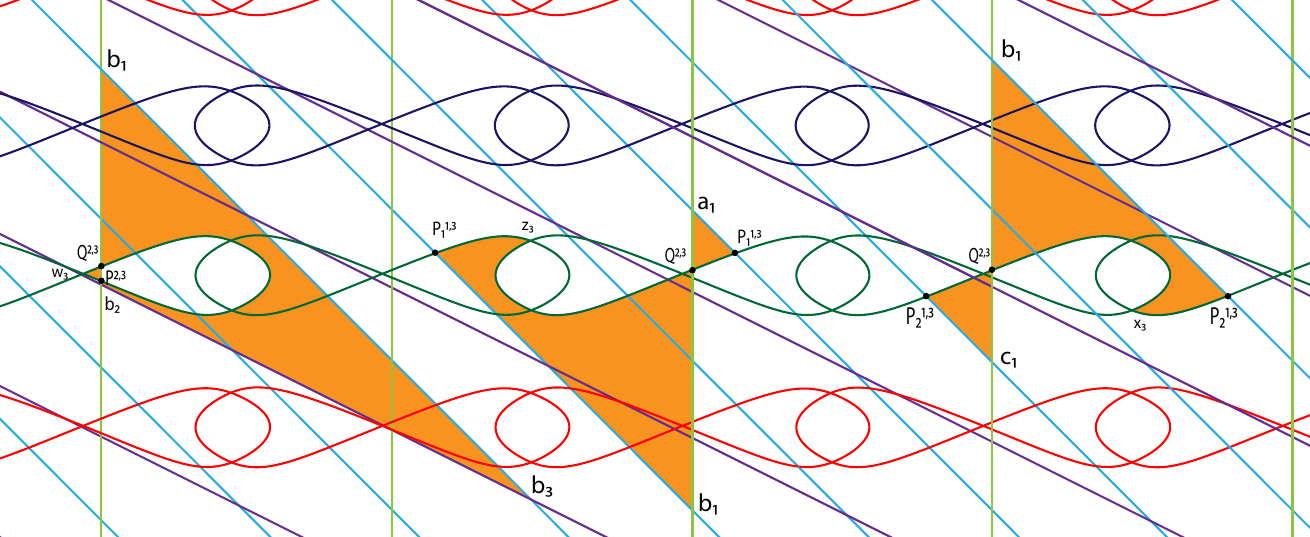}
		\caption{Polygons in $a_3^0$}
		\label{fig:KP2_universal_ai_0}
	\end{figure}
	
	\begin{figure}[H]
		\centering
		\includegraphics[scale=0.45,trim={7.5cm 11.5cm 9.5cm 2.5cm},clip]{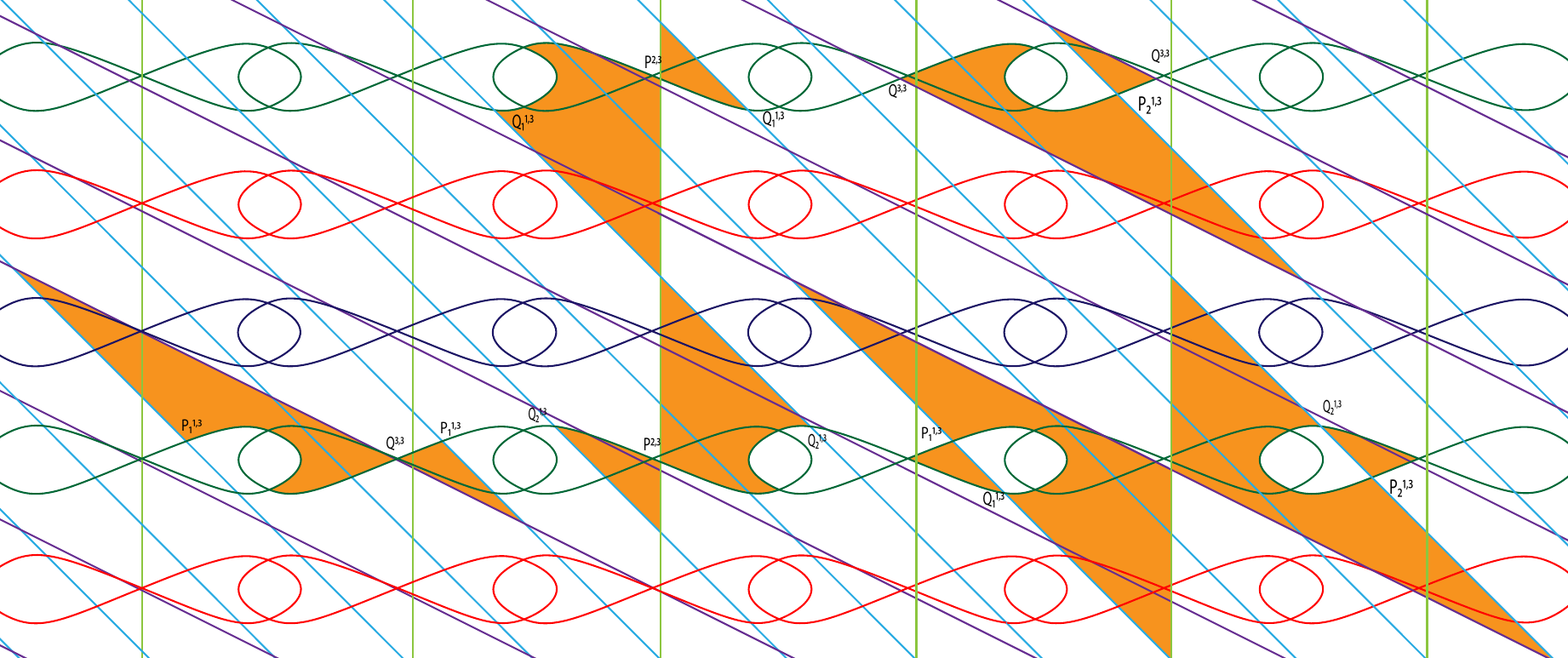}
		\caption{Polygons in $a_3^1$}
		\label{fig:KP2_universal_ai_1}
	\end{figure}
	
	\begin{figure}[H]
		\centering
		\includegraphics[scale=0.5,trim={10.5cm 11.3cm 11.5cm 7cm},clip]{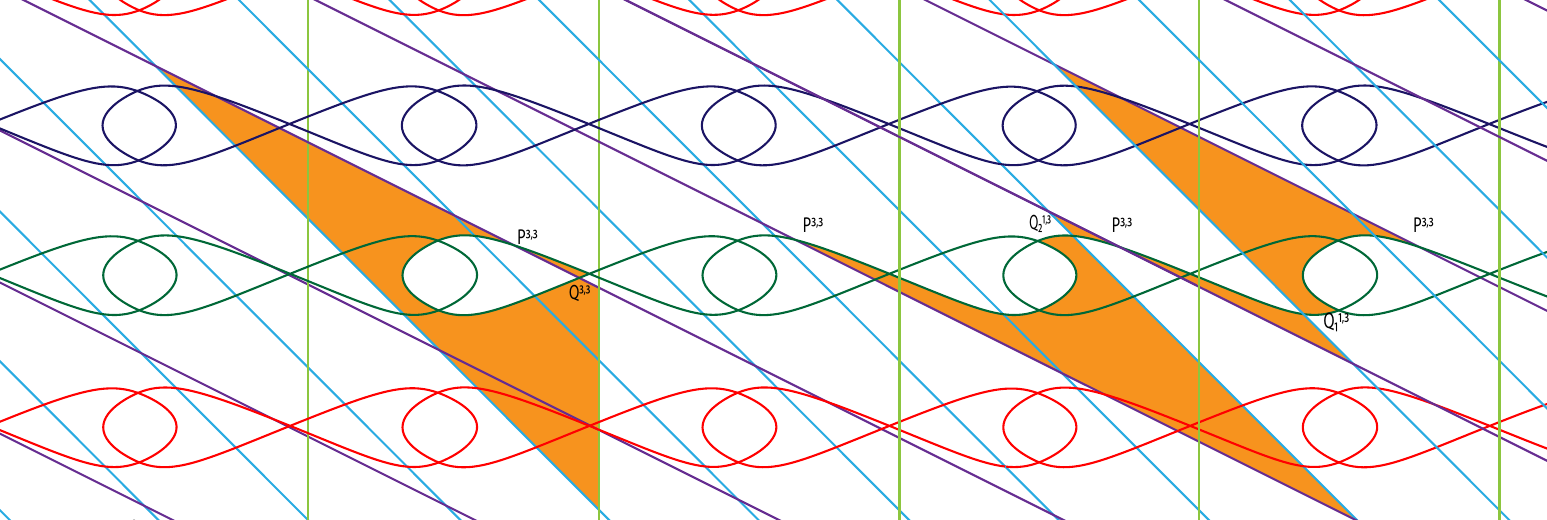}
		\caption{Polygons in $a_3^2$}
		\label{fig:KP2_universal_ai_2}
	\end{figure}
	
	\newpage
	
	Now we show polygons in $a_{30}^{1,0}$ and $a_{03}^{1,0}$. The polygons in other $a_{i0}^{1,0},a_{0j}^{1,0}$ are similar. We firstly show polygons in $a_{30}^{1,0}$ and $a_{03}^{1,0}$ where $\tilde{b}$ is not involved. 
	
	In the following pictures, pink polygons are the polygons in $a_{03}^{1,0}$ and orange polygons are the ones in $a_{30}^{1,0}$:
	
	\begin{figure}[H]
		\centering
		\includegraphics[scale=0.5,trim={9cm 11.3cm 13.5cm 7cm},clip]{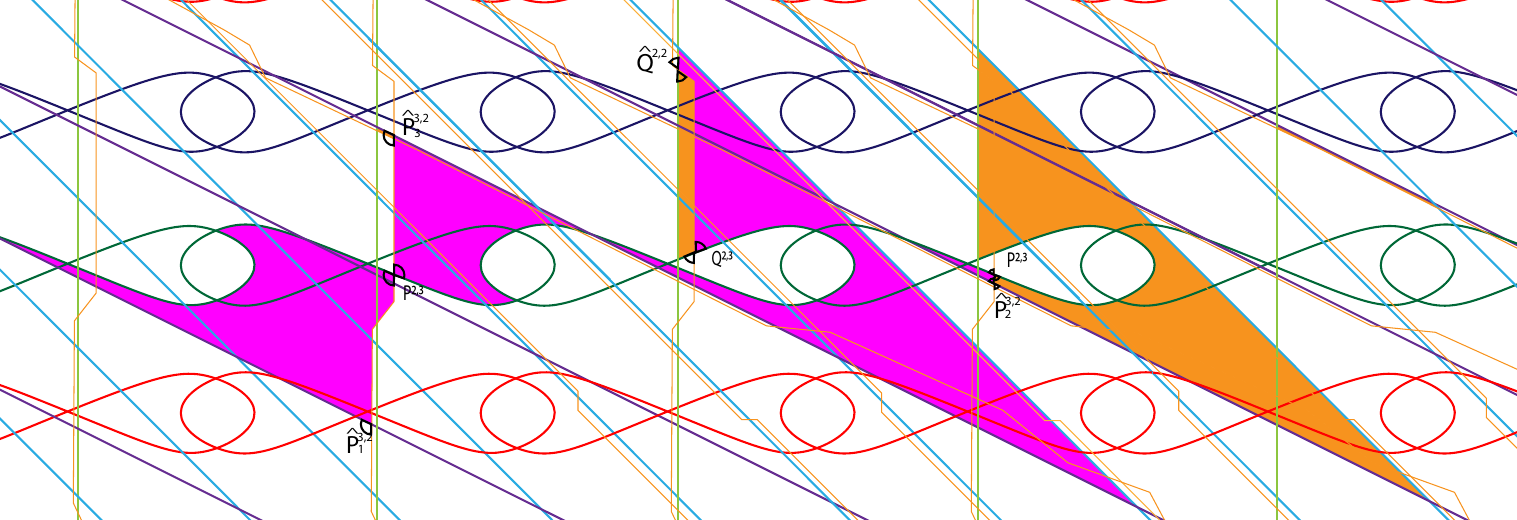}
	\end{figure}
	
	\begin{figure}[H]
		\centering
		\includegraphics[scale=0.5,trim={13cm 11.3cm 13.5cm 7cm},clip]{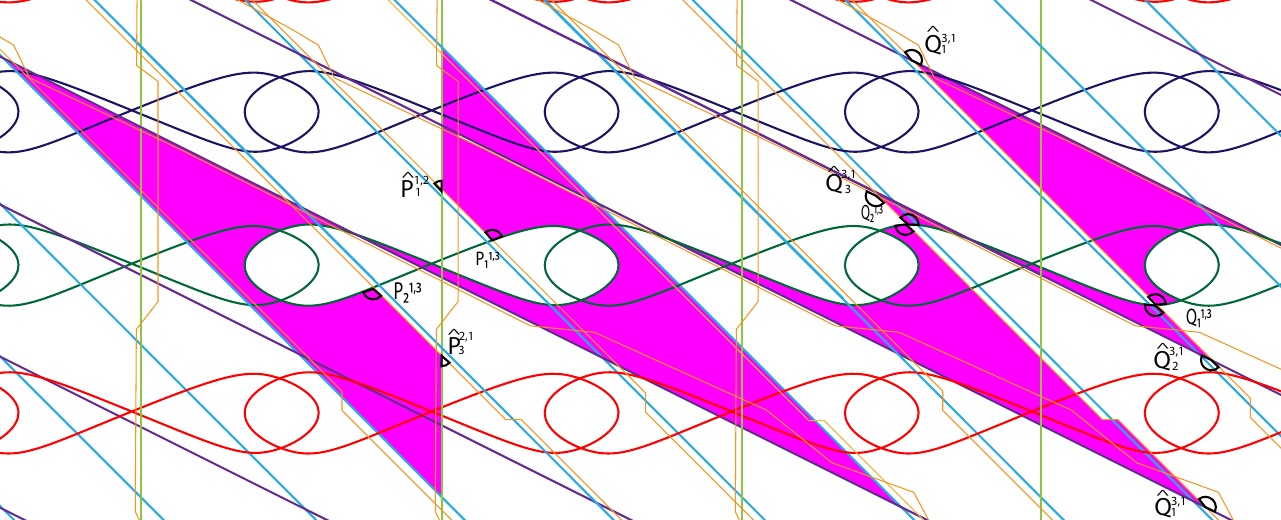}
		
	\end{figure}
	
	\begin{figure}[H]
		\centering
		\includegraphics[scale=0.5,trim={8cm 11.3cm 13.5cm 2cm},clip]{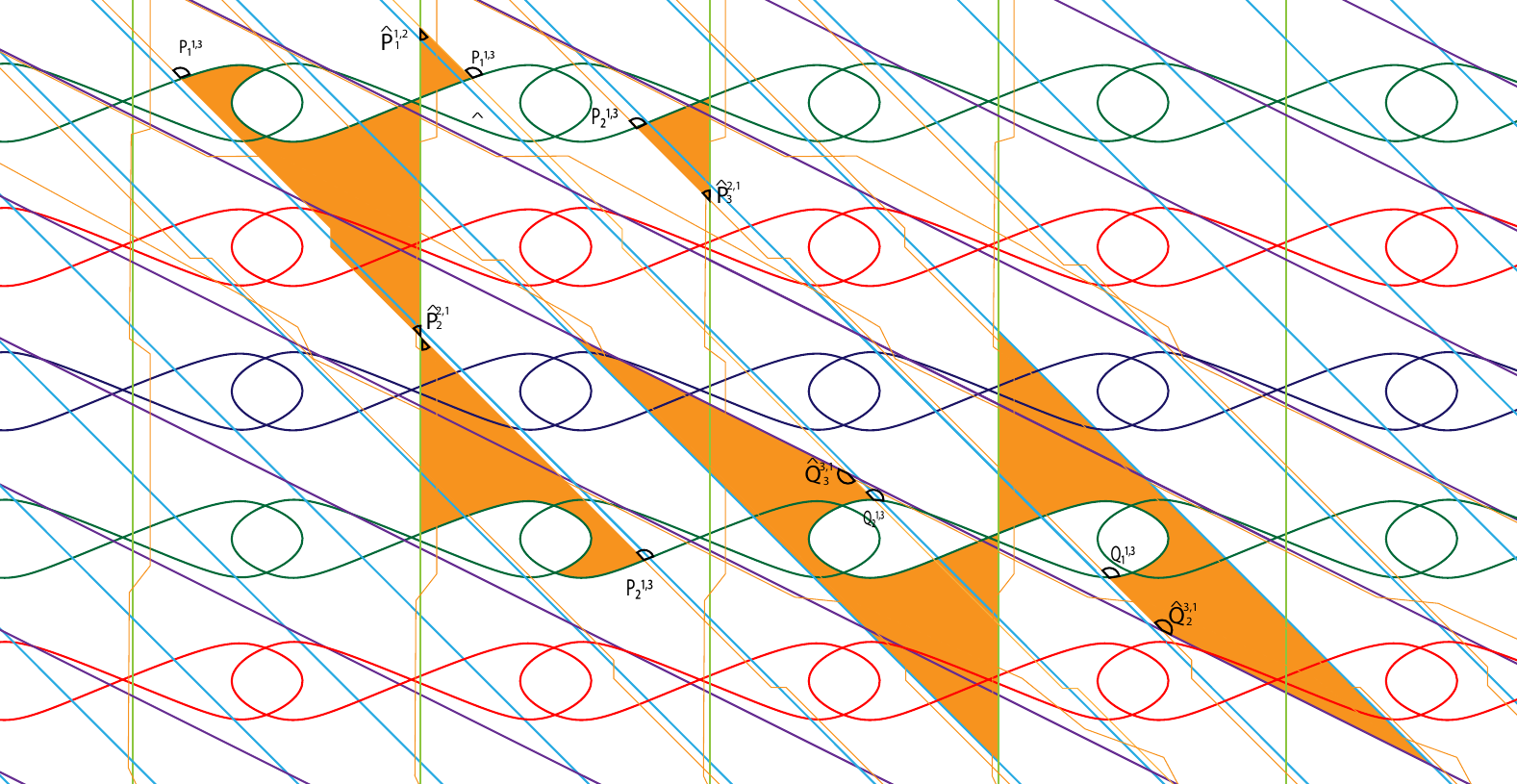}
		
	\end{figure}
	
	\begin{figure}[H]
		\centering
		\includegraphics[scale=0.5,trim={13cm 11.3cm 11.5cm 6cm},clip]{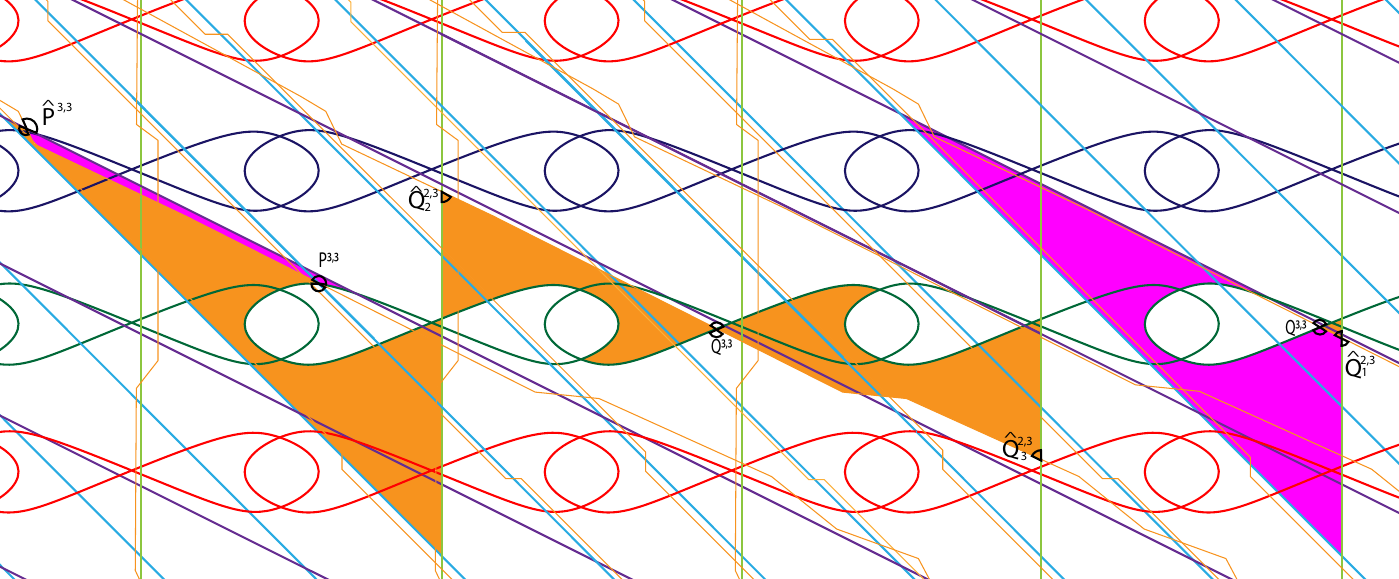}
		
	\end{figure}
	
	Then, we show polygons where $\tilde{b}$ is involved.
	
	\begin{figure}[H]
		\centering
		\includegraphics[scale=0.5,trim={7cm 11.3cm 14.5cm 6cm},clip]{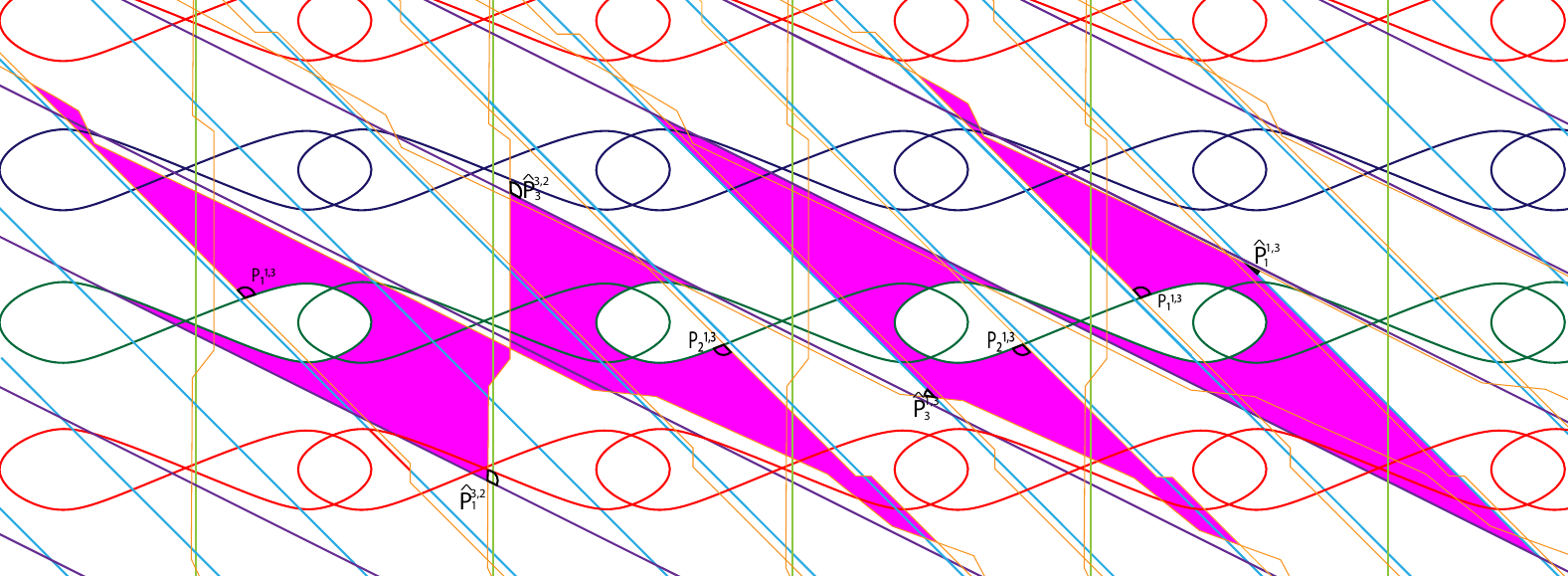}
		
	\end{figure}
	
	\begin{figure}[H]
		\centering
		\includegraphics[scale=0.5,trim={8cm 11.3cm 15.5cm 6cm},clip]{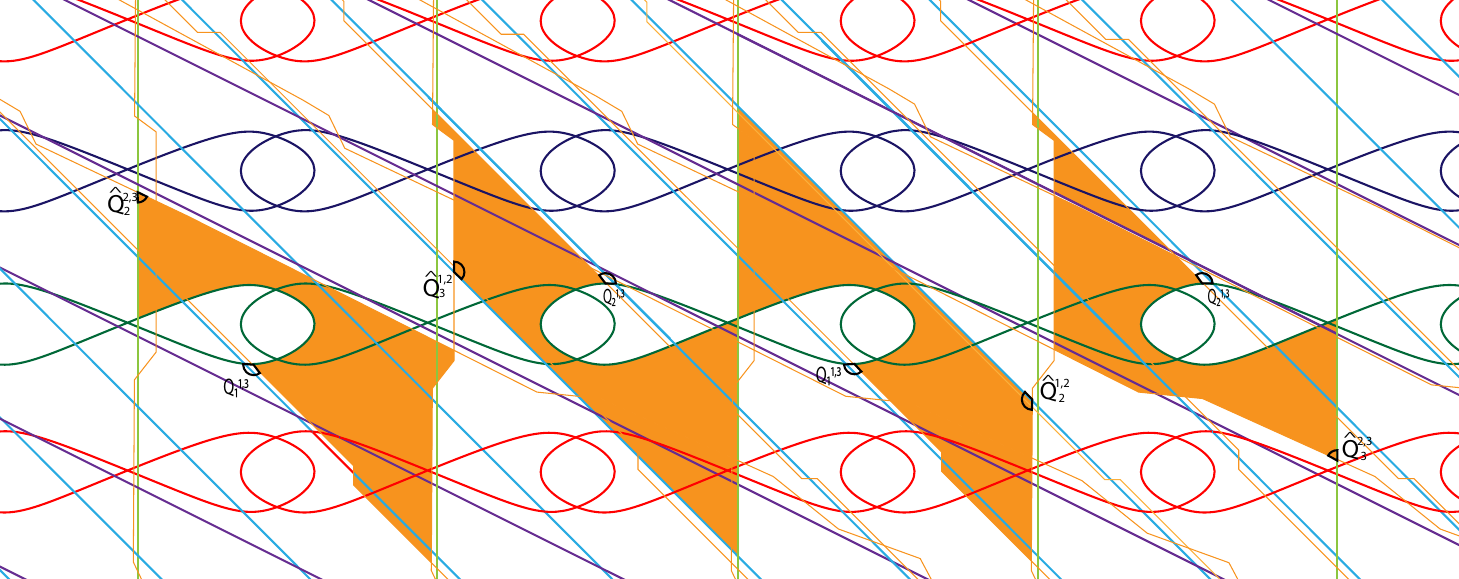}
	\end{figure}
	
	\begin{figure}[H]
		\centering
		\includegraphics[scale=0.5,trim={8cm 11.3cm 14.5cm 6cm},clip]{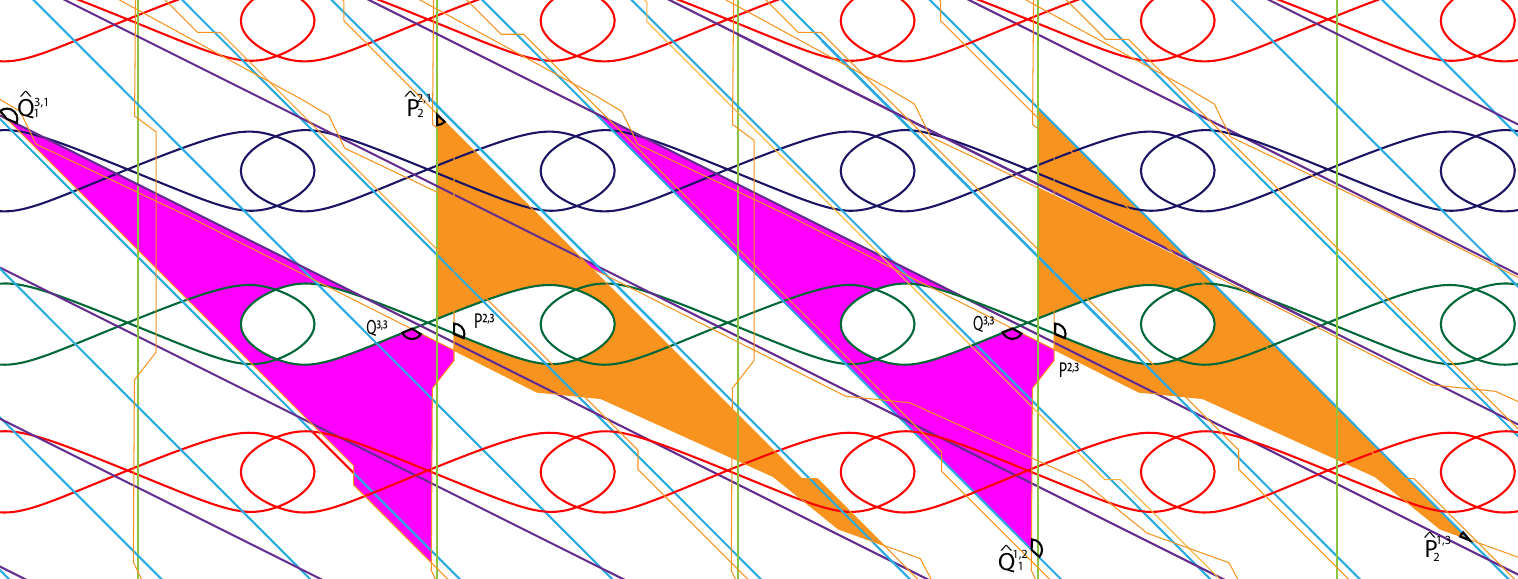}
	\end{figure}
	
\end{subsection}

\begin{subsection}{Polygons in $m_3$}\label{appendix:KP2_aijk}
	Like previous sections, we show the polygons in $m_3^{\bb_0,\bb_2,\bb_0,\bb'}(\alpha_2,\beta_2,\cdot)$. Other cases are similar.
	
	\begin{figure}[H]
		\centering
		\includegraphics[scale=0.5,trim={13cm 13cm 14.5cm 4cm},clip]{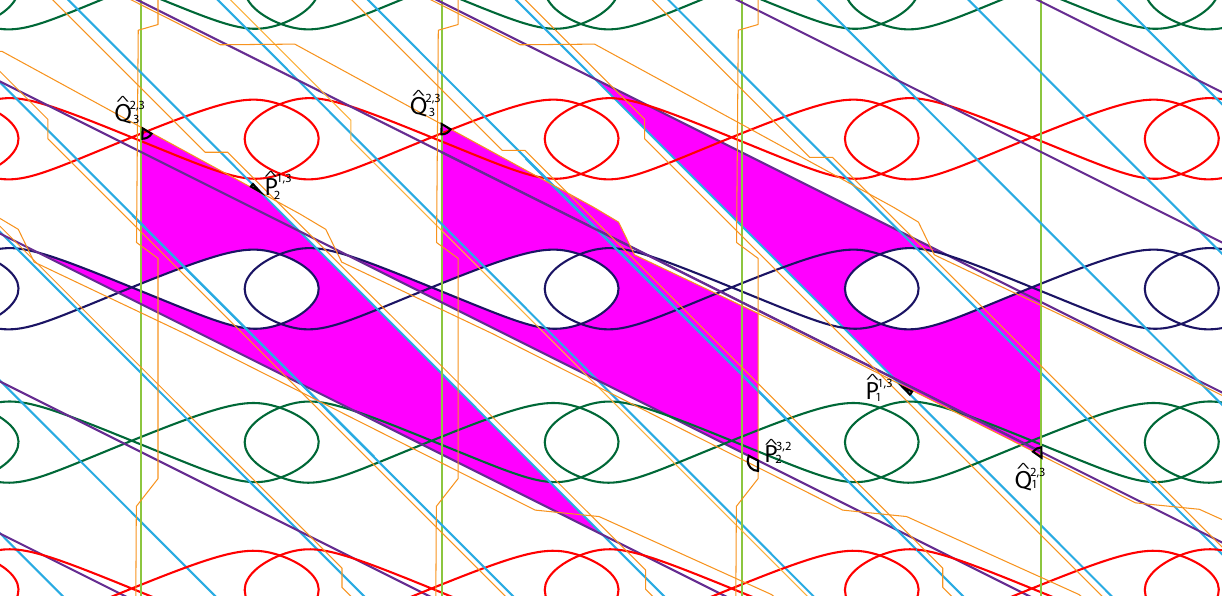}
	\end{figure}
	\begin{figure}[H]
		\centering
		\includegraphics[scale=0.5,trim={13cm 13cm 14.5cm 4cm},clip]{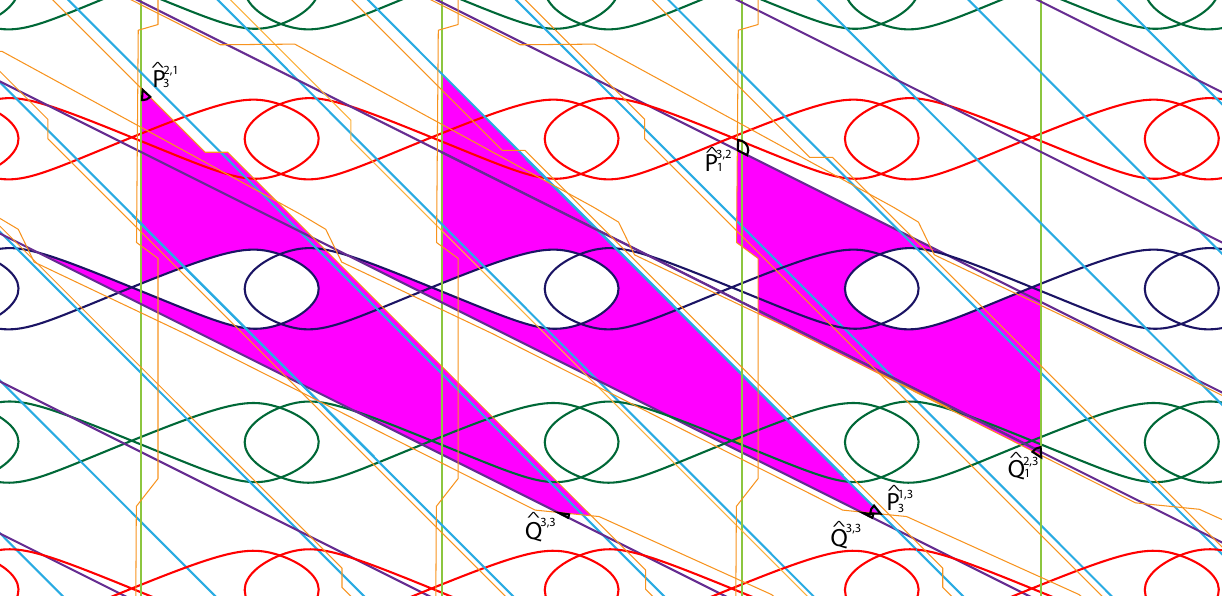}
	\end{figure}
	\begin{figure}[H]
		\centering
		\includegraphics[scale=0.5,trim={8cm 13cm 14.5cm 4cm},clip]{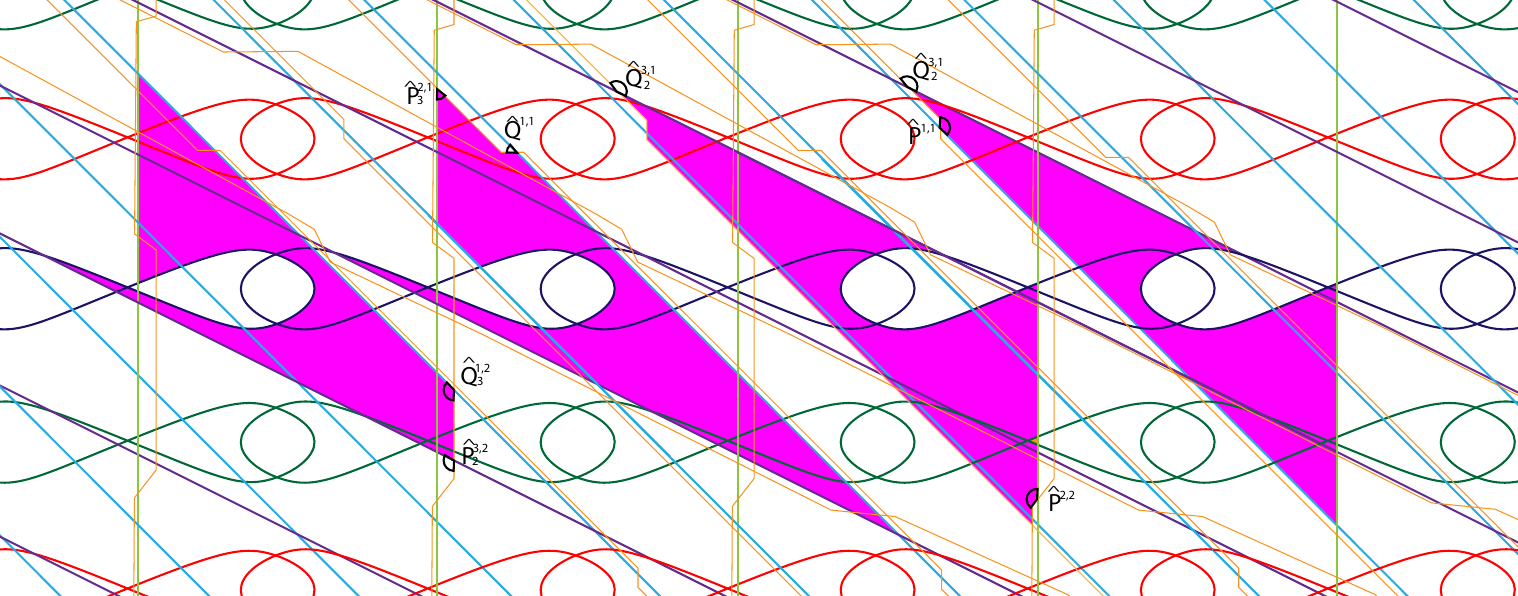}
	\end{figure}
	\begin{figure}[H]
		\centering
		\includegraphics[scale=0.5,trim={13cm 12cm 18cm 4cm},clip]{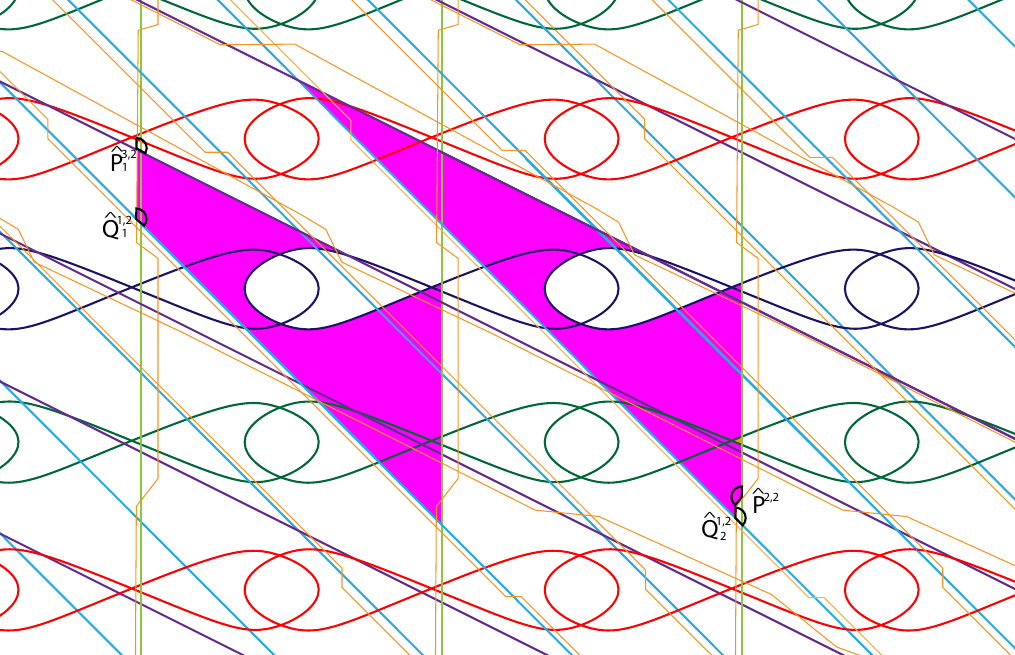}
	\end{figure}
\end{subsection}

\subsection{Computation of Arrows in Universal Bundles} \label{sec:U}
\subsubsection{Horizontal Arrows}\label{def:horizontal_KP2}
$$
a_3^0 = \begin{pmatrix}
	w_3 \bullet - T^{B} \bullet b_1b_3b_2 & - \bullet a_1 + T^{\frac{B}{2}} z_3 \bullet b_1 & -\bullet c_1 + T^{\frac{B}{2}+\hbar} x_3 \bullet b_1 \\
\end{pmatrix}
$$
$$
a_3^1 = 
\tiny{
	-\begin{pmatrix}
		0& T^{A_1} \bullet c_1 - T^{A_{(115)'}} x_3\bullet b_1 & -T^{A_1'}\bullet a_1 + T^{A_{115(3)'}} z_3\bullet b_1 \\
		-T^{A_1'} \bullet c_3 + T^{A_{11(35)'}} x_3 \bullet b_3 & 0 & -T^{A_1'}w_3\bullet + T^{3A_1+A_{5(35)'}}\bullet b_3b_2b_1\\
		T^{A_1}\bullet a_3 - T^{A_{5(11)
				'}} z_3\bullet b_3 & T^{A_1} w_3 \bullet - T^{3A_1'+A_{5(5)'}} \bullet b_3b_2b_1 & 0
\end{pmatrix}}
$$
$$
a_3^2 = 
\begin{pmatrix}
	w_3 \bullet - T^{B} \bullet b_2b_1b_3 \\ - \bullet a_3 + T^{\frac{B}{2}+\hbar} z_3 \bullet b_3 \\ -\bullet c_3 + T^{\frac{B}{2}} x_3 \bullet b_3 \\
\end{pmatrix}
$$

\subsubsection{Vertical Arrows}\label{def:vertical_KP2}
$$a_{32}^0 = 1$$
$$
a_{32}^1 = \tiny{
	\begin{pmatrix}
		-T^{\frac{B}{2}+2 \hbar} \tilde{b}_1\tilde{b}_3\tilde{c}_3^{-1}\tilde{c}_1^{-1}y_2\bullet & 0 & 0 \\
		\substack{T^{B}\bullet b_3b_2+T^{3A_1+A_{5(5)'}-A_{1}'}\tilde{b}_1\tilde{c}_1^{-1}\bullet c_3b_2 + T^{4A_1+A_{5(5)'}-A_{11}'} \tilde{b}_1\tilde{b}_3\tilde{c}_3^{-1}\tilde{c}_1^{-1}\bullet c_3c_2} & -T^{\frac{B}{2}}z_3\bullet& -T^{\frac{B}{2}+\hbar}x_3\bullet \\
		0 & 1 & 0 \\
\end{pmatrix}}
$$
$$
a_{32}^2= \tiny{
	\begin{pmatrix}
		T^{A_{1(35)'}}x_3\bullet & 0 & \substack{T^{4A_1+A_{5(5)'}-A_{11}'}\bullet b_2b_1+T^{3A_1+A_{5(5)'}-A_{1}'}\tilde{b}_1\tilde{c}_1^{-1}\bullet c_2b_1 + T^{B} \tilde{b}_1\tilde{b}_3\tilde{c}_3^{-1}\tilde{c}_1^{-1}\bullet c_2c_1} \\
		0 & 0 & -T^{\frac{B}{2}+\hbar} \tilde{b}_1\tilde{b}_3\tilde{c}_3^{-1}\tilde{c}_1^{-1}y_2 \bullet \\
		0 &  \tilde{b}_1\tilde{b}_3\tilde{c}_3^{-1}\tilde{c}_1^{-1}\bullet & -T^{A_{1(5)'} } \tilde{b}_1\tilde{b}_3\tilde{c}_3^{-1}\tilde{c}_1^{-1}z_2\bullet\\
\end{pmatrix}}
$$
$$
a_{32}^3 = \begin{pmatrix}
	\bullet b_1b_3c_3^{-1}c_1^{-1}\\
\end{pmatrix}
$$

$$a_{23}^0 = 1$$

$$
a_{23}^1 = \tiny{
	\begin{pmatrix}
		-T^{\frac{B}{2}-\hbar} \tilde{c}_1\tilde{c}_3\tilde{b}_3^{-1}\tilde{b}_1^{-1}x_3\bullet & 0 & 0 \\
		0 & 0& 1  \\
		\substack{T^{B}\bullet c_3c_2+T^{B-\hbar}\tilde{c}_1\tilde{b}_1^{-1}\bullet b_3c_2 + T^{\frac{B}{2}-\hbar} \tilde{c}_1\tilde{c}_3\tilde{b}_3^{-1}\tilde{b}_1^{-1}\bullet b_3b_2} & -T^{\frac{B}{2}}y_2\bullet&-T^{\frac{B}{2}+\hbar}z_2\bullet \\
\end{pmatrix}}
$$
$$
a_{23}^2 = \tiny{
	\begin{pmatrix}
		T^{\frac{B}{2}-\hbar}y_2\bullet & 	\substack{T^{\frac{B}{2}-\hbar}}\bullet c_2c_1+T^{B-\hbar}\tilde{c}_1\tilde{b}_1^{-1}\bullet b_2c_1 + T^{B} \tilde{b}_1\tilde{b}_3\tilde{c}_3^{-1}\tilde{c}_1^{-1}\bullet b_2b_1 & 0 \\
		0 & -T^{\frac{B}{2}+\hbar} \tilde{c}_1\tilde{c}_3\tilde{b}_3^{-1}\tilde{b}_1^{-1}z_3\bullet &  \tilde{c}_1\tilde{c}_3\tilde{b}_3^{-1}\tilde{b}_1^{-1}\bullet \\
		0 & -T^{\frac{B}{2}} \tilde{c}_1\tilde{c}_3\tilde{b}_3^{-1}\tilde{b}_1^{-1}x_3 \bullet  & 0\\
\end{pmatrix}}
$$
$$
a_{23}^3 = \begin{pmatrix}
	\bullet c_1c_3b_3^{-1}b_1^{-1}\\
\end{pmatrix}
$$

\subsubsection{Higher Homotopies}\label{def:homotopy_KP2}
For $k=0,2,3$, $$a_{321}^k = 0$$

$$
a_{321}^1 =\begin{pmatrix}
	\substack{T^{B - A_1'} \tilde{b}_1\tilde{b}_3\tilde{c}_3^{-1}\tilde{a}_1^{-1}\bullet a_2 
		+ T^{B - A_1'}  \tilde{b}_1\tilde{c}_1^{-1} \bullet b_2
		+ T^{3A_1+A_{5(5)'} -2A_1'} \tilde{b}_1\tilde{b}_3\tilde{c}_3^{-1}\tilde{c}_1^{-1} \bullet c_2}& 0 & 0\\
	0 & 0 & 0\\
	0 & 0 & 0\\
\end{pmatrix}
$$

\end{appendices}

\bibliographystyle{amsalpha}
\bibliography{geometry}

\end{document}